\documentclass[12pt]{amsart}

\usepackage{amsfonts}
\usepackage{amsmath}
\usepackage{amssymb}
\usepackage{enumerate}
\usepackage{comment}
\usepackage{hyperref}
\newtheorem{thm}{Theorem}[section]

\newtheorem{cor}[thm]{Corollary}
\newtheorem{lem}[thm]{Lemma}
\newtheorem{prop}[thm]{Proposition}
\newtheorem*{prob*}{Problem}
\newtheorem*{thm*}{Theorem}
\newtheorem*{quest*}{Question}

\theoremstyle{definition}
\newtheorem{defn}[thm]{Definition}
\newtheorem{example}[thm]{Example}

\newtheorem*{defn*}{Definition}
\newtheorem{rem}[thm]{Remark}

\newtheorem{rem*}[thm]{Remark}
\numberwithin{equation}{section}

\newcommand{\mP}{\mathbb {P}}
\newcommand{\mG}{\mathcal{G}}
\newcommand{\mL}{\mathcal{L}}
	\title{Host-Kra theory for $\bigoplus_{p\in P}\mathbb{F}_p$- Systems and multiple recurrence}
\date{\today}
\author{Or Shalom}

\begin{document}
	\begin{abstract}
		Let $\mathcal{P}$ be an (unbounded) countable multiset of primes (i.e. every prime may appear multiple times) and let $G=\bigoplus_{p\in \mathcal{P}}\mathbb{F}_p$. We develop a Host-Kra structure theory for the universal characteristic factors of an ergodic $G$-system. More specifically, we generalize the main results of Bergelson Tao and Ziegler from \cite{Berg& tao & ziegler} who studied these factors in the special case $\mathcal{P}=\{p,p,p,...\}$ for some fixed prime $p$. As an application we deduce a Khintchine-type recurrence theorem in the flavor of Bergelson Tao and Ziegler \cite{BTZ} and Bergelson Host and Kra \cite{BHK}.
	\end{abstract}
	\maketitle
\section{Introduction}
This paper is concerned with the study of the structure of the universal characteristic factors corresponding to multiple averages associated with the action of  $G=\bigoplus_{p\in \mathcal{P}}\mathbb{F}_p$ for some countable multiset of primes $P$ (formal definitions below).  The study of characteristic factors for multiple averages plays an important role in ergodic Ramsey theory. For example in the case of $\mathbb Z$-actions they are related to the theorem of Szemer\'edi on arithmetic progressions in sets of positive density in the integers. Universal characteristic factors corresponding to multiple averages associated with $\mathbb{Z}$-actions were studied by Host-Kra \cite{HK}, and Ziegler \cite{Z}, and for $\mathbb{F}_p^\omega$-actions by Bergelson-Tao-Ziegler \cite{Berg& tao & ziegler}. In this paper we develop the theory further to  $\bigoplus_{p\in \mathcal{P}}\mathbb{F}_p$ actions for a countable multiset of primes $P$. If $P$ is unbounded the universal characteristic factors may have a pathological structure as can be seen in Theorem \ref{example:thm}. This and other properties make the unbounded case differ significantly from that of $\mathbb{F}_p^\omega$-systems, see for instance Example \ref{Example} and the discussion afterwards.    \\
\ \\
We briefly survey our results. All definitions and statements used here are defined later in the paper where the results are formulated.
\begin{itemize}
\item {\textbf{Strongly Abramov $\bigoplus_{p\in P} \mathbb{F}_p$ - systems.}}
 In \cite{Berg& tao & ziegler} Bergelson Tao and Ziegler proved that the universal characteristic factors for $\mathbb{F}_p^\omega$-actions are \textit{strongly Abramov} (see Definition \ref{stronglyAbr}). We show that this is not generally true for $\bigoplus_{p\in P}\mathbb{F}_p$-actions and construct an example in Section \ref{example}. Most of the paper is devoted to proving an if and only if criterion for a $\bigoplus_{p\in P}\mathbb F_p$-system to be strongly Abramov (Theorem \ref{Mainr:thm}). Roughly speaking, in Theorem \ref{Main:thm} we prove that unless a system admits some pathological properties it must be strongly Abramov.
 \item {\textbf{$\bigoplus_{p\in P}\mathbb F_p$-systems as nilpotent systems}}
Host and Kra \cite{HK}, and Ziegler \cite{Z} proved that the universal characteristic factors for $\mathbb{Z}$-actions are inverse limits of nilsystems. In the case of Conze-Lesigne systems (systems of order $<3$) we prove a counterpart for the Host-Kra theorem for $\bigoplus_{p\in\mathcal{P}}\mathbb{F}_p$. We show that every such system is an inverse limit of \textit{nilpotent systems}; these are similar to nilsystems, but the Host-Kra group (the homogeneous group) as defined in \cite{HK}, is not necessarily a Lie group. For more details see Theorem \ref{nilpotentstructure}.
 \item {\textbf{Limit formula for multiple ergodic averages and a Khintchine-type recurrence Theorem.}}\\ Lesigne \cite{Les} and Ziegler \cite{Z1} proved a limit formula for the multiple ergodic averages associated with Szemer\'edi's theorem when the underlying space is a connected simply connected homogeneous space. We prove a counterpart for multiple averages associated with Szemer\'edi's theorem for $4$ term progressions on nilpotent systems (see Theorem \ref{formula}). We use our structure results to deduce a Khintchine-type recurrence result for four term progressions (Theorem \ref{recurrence:thm}). Similar averages for $\mathbb{Z}$ and $\mathbb{F}_p^\omega$ actions were studied by Bergelson-Host-Kra \cite{BHK} and by Bergelson-Tao-Ziegler \cite{BTZ} respectively.
\end{itemize}
\subsection{Universal Characteristic Factors}
We begin with some standard definitions, essentially taken from \cite{Berg& tao & ziegler}.
\begin{defn} A $G$-system is a quadruple $(X,\mathcal{B},\mu,G)$ where $(X,\mathcal{B},\mu)$ is a probability measure space which is separable modulo null sets\footnote{For technical reasons we assume that the space $(X,\mathcal{B},\mu)$ is regular, meaning that $X$ is compact, $\mathcal{B}$ is the Borel $\sigma$-algebra and $\mu$ is a Borel measure.}, together with an action of $G$ on $X$ by measure preserving transformations $T_g:X\rightarrow X$. For every topological group\footnote{All topological groups in this paper are implicitly assumed to be metrizable.} $(U,\cdot)$, measurable map $\phi:X\rightarrow U$ and element $g\in G$, we define the shift $T_g\phi = \phi \circ T_g$ and the multiplicative derivative $\Delta_g \phi = \frac{T_g \phi}{\phi}$. We say that a $G$-system $X$ is ergodic if the only functions in $L^2(X)$ which are invariant under the $G$-action are the constant functions.
\end{defn} 
The Gowers-Host-Kra seminorms play an important role throughout this paper
\begin{defn}
	[Gowers-Host-Kra (GHK) seminorms for an arbitrary countable abelian group $G$] \label{Gowers-Host-Kra}
	Let $G$ be a countable abelian group, let $X=(X,\mathcal{B},\mu)$ be a $G$-system, let $\phi\in L^\infty (X)$, and let $k\geq 1$ be an integer. The Gowers-Host-Kra seminorm $\|\phi\|_{U^k}$ of order $k$ of $\phi$ is defined recursively by the formula
	\[
	\|\phi\|_{U^1}:=\lim_{N\rightarrow\infty}\frac{1}{|\Phi_N^1|}\left\lVert\sum_{g\in\Phi_N^1}\phi\circ T_g\right\rVert_{L^2}
	\]
	for $k=1$, and
	\[
	\|\phi\|_{U^k}:=\lim_{N\rightarrow\infty}\left(\frac{1}{|\Phi_N^k|}\sum_{g\in\Phi_N^k}\left\lVert\Delta_g\phi\right\rVert_{U^{k-1}}^{2^{k-1}}\right)^{1/2^k}
	\]
	for $k\geq 1$, where $\phi_N^1,...,\phi_N^k$ are arbitrary F{\o}lner sequences.
\end{defn}
These seminorms where first introduced by Gowers in the special case where $G=\mathbb{Z}/N\mathbb{Z}$ in \cite{G}, where he derived quantitative bounds for Szemer{\'e}di's theorem about the existence of arbitrarily large arithmetic progressions in sets of positive upper Banach density, \cite{Sz}. Later, in \cite{HK} Host and Kra introduced essentially the above ergodic theoretical version of the Gowers norms for $G=\mathbb{Z}$ (see \cite[Proposition 16 Chapter 8]{HKbook} for this version).\\ 

Gowers' work raised a natural question about the behavior of functions with large $U^k$ norm. In the context of finite groups, this question is known as the inverse problem for the Gowers norms and it was answered partially by Gowers for the case $G=\mathbb{Z}/N\mathbb{Z}$. Inspired by the work of Host and Kra \cite{HK}, Green and Tao proved an inverse theorem for the Gowers $U^k$ norms  for $\mathbb{Z}/N\mathbb{Z}$ in the case $k=3$ in \cite{GT1} and together with Ziegler for general $k$ in \cite{GTZ}. Their work hints at a link between the ergodic theoretical structure of the universal characteristic factors and the inverse problem for the Gowers norms. Surprisingly, if one considers these problems in the context of vector spaces over finite fields this link becomes more concrete. Namely, in \cite{TZ} Tao and Ziegler deduced an inverse theorem for the Gowers norms over finite fields from an ergodic theoretical structure theorem for  $\mathbb{F}_p^\omega$-systems which they established together with Bergelson in \cite{Berg& tao & ziegler}.\\
 
 Another approach for the inverse problem is the study of nilspaces. In \cite{SzCam} Antol\'in Camarena and Szegedy introduced a purely combinatorial counterpart of the universal characteristic factors called a {\em nilspace}. The idea was to give a more abstract and general notion which describes the "cubic structure" of an ergodic system (see Host and Kra \cite[Section 2]{HK}). In \cite{Sz3} Candela and
Szegedy use nilspaces to prove a structure theorem for characteristic factors for GHK
seminorms associated with any nilpotent group, giving in particular an alternative proof
of the Host-Kra structure theorem. The notion "nilspace" is more general and abstract than the measure-theoretical counterpart. Thus, describing these nilspaces in a concrete way is often a difficult problem on its own. In a series of papers, \cite{Gut1},\cite{Gut2},\cite{Gut3} Gutman, Manners, and Varj{\'u} studied further the structure of nilspaces (more results in this direction can be found in \cite{Sz1},\cite{Sz2} by Candela Gonz\'alez-
S\'anchez and Szegedy and in \cite{GGY} by Glasner Gutman and Ye). By imposing another measure-theoretical aspect to these nilspaces, Gutman and Lian \cite{Gut4} gave another alternative proof of Host and Kra's theorem. 

In our work we do not pursue this approach, instead our goal is to generalize the ergodic theoretical structure for other groups and to deduce a Khintchine type recurrence.

The GHK seminorms correspond to a "factor" of $X$ (See Proposition \ref{UCF} below).
\begin{defn} [Factors, push-forwards and pullbacks] \label{Factor:def}
	Let $(X,\mathcal{B}_X, \mu_X, (T_g)_{g\in G})$ be a $G$-system. We say that a $G$-system $(Y,\mathcal{B}_Y,\mu_Y, (S_g)_{g\in G})$ is a factor of $(X,\mathcal{B}_X, \mu_X, (T_g)_{g\in G})$ if there is a measure preserving factor map $\pi_Y^X : X\rightarrow Y$ such that the push-forward of $\mu_X$ by $\pi^X_Y$ is $\mu_Y$ and $\pi^X_Y\circ T_g =S_g \circ \pi^X_Y$ $\mu_X$-a.e. for all $g\in G$.

	For a measure-space $U$ and a measurable map $f:Y\rightarrow U$ we define the pullback $(\pi_Y^X )^\star f: X\rightarrow U$ by $(\pi_Y^X )^\star f = f\circ\pi^X_Y$. We define the push-forward of $f$ by $\pi_Y^X$ to be the unique element $(\pi^X_Y)_\star f\in L^2(Y)$ such that $E(f|Y):=(\pi^X_Y)^\star (\pi^X_Y)_\star f $, where $E(f|Y)$ is the conditional expectation of $f$ with respect to the factor $Y$.\\
	We say that a function $f$ is $\mathcal{B}_Y$-measurable or measurable with respect to $Y$ if $f=E(f|Y)$, or equivalently if $f=(\pi^X_Y)^\star F$ for some $F\in L^2(Y)$.

	In this case we refer to $X$ as an extension of the $G$-system $Y$. Finally, we say that a factor $Y$ is generated by a collection $\mathcal{F}$ of measurable functions $f:X\rightarrow \mathbb{C}$ if $Y$ is the minimal factor of $X$ such that all $f\in\mathcal{F}$ are measurable with respect to $Y$.
\end{defn}
We have the following  Proposition\slash Definition (See Host and Kra, \cite[Lemma 4.3]{HK}).
\begin{prop}[Existence and uniqueness of the universal characteristic factors] \label{UCF} Let $G$ be a countable abelian group, let $X$ be a $G$-system, and let $k\geq 1$. Then there exists a factor $Z_{<k}(X)=(Z_{<k}(X),\mathcal{B}_{Z_{<k}(X)},\mu_{Z_{<k}(X)},\pi^X_{Z_{<k}(X)})$ of $X$ with the property that for every $f\in L^\infty (X)$, $\|f\|_{U^{k}(X)}=0$ if and only if $(\pi^{X}_{Z_{<k}(X)})_\star f = 0$ $($equivalently, $E(f|Z_{<k}(X))=0)$. This factor is unique up to isomorphism and is called the $k$-th universal characteristic factor of $X$.
\end{prop}
The structure of the universal characteristic factors for the GHK-norms for $\mathbb{Z}$-systems was studied by Host and Kra in \cite{HK} as a tool in the study of some non-conventional ergodic averages. Those averages were originally introduced by Furstenberg \cite{F1} in his proof of Szemer{\'e}di's Theorem. In \cite{Z} Ziegler defined these factors differently and in \cite{Leib} Leibman proved the equivalence.
\begin{thm} [Structure theorem for $Z_{<k}(X)$ for ergodic $\mathbb{Z}$-systems] \cite[Theorem 10.1]{HK} \cite[Theorem 1.7]{Z} \label{HK}
	For an ergodic system $X$, $Z_{<k}(X)$ is an inverse limit of $k$-step nilsystems\footnote{A $k$-step nilsystem is a quadruple $(\mathcal{G}/\Gamma,\mathcal{B},\mu,\mathbb{Z})$ where $\mathcal{G}$ is a $k$-step nilpotent Lie group, $\Gamma$ a co-compact subgroup, $\mathcal{B}$ is the Borel $\sigma$-algebra, $\mu$ the induced Haar measure and the action of $\mathbb{Z}$ is given by a left translation by an element in $\mathcal{G}$.}.
\end{thm}
This theorem led to various multiple recurrence and convergence results in ergodic theory, see for instance \cite{CL84},\cite{CL87},\cite{CL88},\cite{F&W} and \cite{BTZ}.
\subsection{Abelian cohomology and some notations}
We use the same notations as in \cite{Berg& tao & ziegler}.
\begin{defn} [Abelian cohomology] Let $G$ be a countable abelian group. Let $X=(X,\mathcal{B},\mu,G)$ be a $G$-system, and let $U=(U,\cdot)$ be a compact abelian group. 
	
	\begin{itemize}
		\item{We denote by $\mathcal{M}(X,U)$ the set of all measurable functions $\phi:X\rightarrow U$, with two functions $\phi,\phi'$ identified if they agree $\mu$-almost everywhere. $\mathcal{M}(X,U)$ is a topological group with respect to pointwise multiplication and the topology of convergence in measure.}
		\item {Similarly, let $\mathcal{M}(G,X,U)$ denote the set of all measurable functions $\rho:G\times X\rightarrow U$ with $\rho,\rho'$ being identified if $\rho(g,x)=\rho'(g,x)$ for $\mu$-almost every $x\in X$ and every $g\in G$.}
		\item {We let $Z^{1}(G,X,U)$ denote the subgroup of $\mathcal{M}(G,X,U)$ consisting of those $\rho:G\times X\rightarrow U$ which satisfy $\rho(g+g',x)=\rho(g,x)\rho(g',T_gx)$ for all $g,g'\in G$ and $\mu$-almost every $x\in X$. We refer to the elements of $Z^{1}(G,X,U)$ as cocycles.}
		\item {Given a cocycle $\rho:G\times X\rightarrow U$ we define the abelian extension $X\times_{\rho} U$ of $X$ by $\rho$ to be the product space $(X\times U,\mathcal{B}_X\times\mathcal{B}_U,\mu_X\times\mu_U)$ where $\mathcal{B}_U$ is the Borel $\sigma$-algebra on $U$ and $\mu_U$ the Haar measure. We define the action of $G$ on this product space by $(x,u)\mapsto (T_gx,\rho(g,x)u)$ for every $g\in G$. In this setting we define the vertical translations  $V_u(x,t)=(x,ut)$ on $X\times_{\rho} U$ for every $u\in U$. We note that this action of $U$ commutes with the $G$-action on this system.}
		\item {If $F\in \mathcal{M}(X,U)$, we define the derivative $\Delta F\in \mathcal{M}(G,X,U)$ of $F$ to be the function $\Delta F(g,x):=\Delta_g F(x)$. We write $B^1(G,X,U)$ for the image of $\mathcal{M}(X,U)$ under the derivative operation. We refer to the elements of $B^1(G,X,U)$ as $(G,X,U)$-coboundaries.}
		\item {We say that $\rho,\rho'\in \mathcal{M}(G,X,U)$ are $(G,X,U)$-cohomologous if $\rho/\rho'\in B^1(G,X,U)$. As usual, we denote the cohomology group by $H^1(G,X,U)=Z^1(G,X,U)/B^1(G,X,U)$.}
	\end{itemize}
\end{defn}
\begin{rem} \label{coh:rem} Observe that if $\rho$ and $\tilde{\rho}$ are $(G,X,U)$-cohomologous, then $X\times_{\rho} U$ and $X\times_{\tilde{\rho}}U$ are measure-equivalent systems. The isomorphism is given by $\pi(x,u)=(x,F(x)u)$ where $F:X\rightarrow U$ is a function such that $\rho = \tilde{\rho}\cdot \Delta F$.
\end{rem}
\subsection{Type of functions}
We introduce the notion of Cubic systems from \cite[Section 3]{HK} (Generalized for arbitrary countable abelian group).
\begin{defn} [Cubic measure spaces]  Let $X=(X,\mathcal{B},\mu,G)$ be a $G$-system for some countable abelian group $G$. For each $k\geq 0$ we define $X^{[k]} =(X^{[k]},\mathcal{B}^{[k]},\mu^{[k]},G^{[k]})$ where $X^{[k]}$ is the Cartesian product of $2^k$ copies of $X$, endowed with the product $\sigma$-algebra $\mathcal{B}^{[k]}=\mathcal{B}^{2^k}$, $G^{[k]}=G^{2^k}$ acting on $X^{[k]}$ in the obvious manner. We define the cubic measures $\mu^{[k]}$ and $\sigma$-algebras $\mathcal{I}_k\subseteq \mathcal{B}^{[k]}$ inductively. $\mathcal{I}_0$ is defined to be the $\sigma$-algebra of invariant sets in $X$, and $\mu^{[0]}:=\mu$. Once $\mu^{[k]}$ and $\mathcal{I}_k$ are defined, we identify $X^{[k+1]}$ with $X^{[k]}\times X^{[k]}$ and define $\mu^{[k+1]}$ by the formula 
	$$\int f_1(x)f_2(y) d\mu^{[k+1]}(x,y) = \int E(f_1|\mathcal{I}_k)(x)E(f_2|\mathcal{I}_k)(x) d\mu^{[k]}(x)$$
	for $f_1,f_2$ functions on $X^{[k]}$ and $E(\cdot|\mathcal{I}_k)$ the conditional expectation, and $\mathcal{I}_{k+1}$ being the $\sigma$-algebra of invariant sets in $X^{[k+1]}$.
\end{defn}
This construction leads to the following notion of type for functions and cocycles, see \cite[Definition 7.1]{HK} and \cite[Definition 4.1]{Berg& tao & ziegler}.
\begin{defn} [Functions of type $<k$]  \label{type:def} Let $G$ be a countable abelian group, let $X=(X,\mathcal{B},\mu,G)$ be a $G$-system. Let $k\geq 0$ and let $X^{[k]}$ be the cubic system associated with $X$ and $G$ acting on $X^{[k]}$ diagonally.
	\begin{itemize}
		\item{For each measurable $f:X\rightarrow U$, we define a measurable map $d^{[k]}f:X^{[k]}\rightarrow U$ to be the function $$d^{[k]}f((x_w)_{w\in \{-1,1\}^k}):=\prod_{w\in \{-1,1\}^k}f(x_w)^{\text{sgn}(w)}$$ where $\text{sgn}(w)=w_1\cdot w_2\cdot...\cdot w_k$.}
		\item {Similarly for each measurable $\rho:G\times X\rightarrow U$ we define a measurable map $d^{[k]}\rho:G\times X^{[k]}\rightarrow U$ to be the function 	$$d^{[k]}\rho(g,(x_w)_{w\in \{-1,1\}^k}):=\prod_{w\in \{-1,1\}^k} \rho(g,x_w)^{\text{sgn}(w)}.$$}
		\item {A function $\rho:G\times X\rightarrow U$ is said to be a function of type $<k$ if $d^{[k]}\rho$ is a $(G,X^{[k]},U)$-coboundary. We let $\mathcal{M}_{<k}(G,X,U)$ denote the subspace of functions $\rho:G\times X\rightarrow U$ of type $<k$. We let $\mathcal{C}_{<k}(G,X,U)$ denote the subspace of $\mathcal{M}_{<k}(G,X,U)$ consisting of elements of this space which are also cocycles.}
	\end{itemize}
\end{defn}

\begin{defn} [Phase polynomials]
	Let $G$ be a countable abelian discrete group, $X$ be a $G$-system, let $\phi\in L^\infty(X)$, and let $k\geq 0$ be an integer. A function $\phi:X\rightarrow \mathbb{C}$ is said to be a phase polynomial of degree $<k$ if we have $\Delta_{h_1}...\Delta_{h_k}\phi = 1$ $\mu_X$-almost everywhere for all $h_1,...,h_k\in G$. (In particular by setting $h_1=...=h_k=0$ we see that $\phi$ must take values in $S^1$, $\mu_X$-a.e.). We write $P_{<k}(X)=P_{<k}(X,S^1)$ for the set of all phase polynomials of degree $<k$. Similarly, a function $f:G\times X\rightarrow \mathbb{C}$ is said to be a phase polynomial of degree $<k$ if $f(g,\cdot)\in P_{<k}(X,S^1)$ for every $g\in G$. We let $P_{<k}(G,X,S^1)$ denote the set of all phase polynomials $f:G\times X\rightarrow \mathbb{C}$ of degree $<k$.
\end{defn}
\begin{defn} \label{Abramov:def} We write $\text{Abr}_{<k}(X)$\footnote{It was Abramov who studied systems of this type for $\mathbb{Z}$-actions, see \cite{A}.} for the factor of $X$ generated by $P_{<k}(X)$, and say that $X$ is an Abramov system of order $<k$ if it is generated by $P_{<k}(X)$, or equivalently if $P_{<k}(X)$ spans $L^2(X)$. 
	\end{defn}
\begin{rem}
	The notion of phase polynomials can be generalized for an arbitrary abelian group $(U,\cdot)$. A function $\phi:X\rightarrow U$ is said to be a phase polynomial of degree $<k$ if $\Delta_{h_1}...\Delta_{h_k}\phi = 1$ $\mu_X$-a.e. for all $h_1,...,h_k\in G$. We let $P_{<k}(X,U)$ and similarly $P_{<k}(G,X,U)$ denote the spaces of phase polynomials of degree $<k$ taking values in $U$.
\end{rem}

We recall some basic facts about functions of type $<k$ from \cite[Lemma 4.3]{Berg& tao & ziegler}.
\begin{lem} \label{PP}
	Let $G$ be a countable abelian group, let $X=(X,\mathcal{B},\mu,G)$ be an ergodic $G$-system, let $U=(U,\cdot)$ be a compact abelian group, and let $k\geq 0$.
	
	\begin{enumerate} [(i)]
		\item {Every function $f:G\times X\rightarrow U$ of type $<k$ is also of type $<k+1$.}
		\item {The set $\mathcal{M}_{<k}(G,X,U)$ is a subgroup of $\mathcal{M}(G,X,U)$ and it contains the group $B^1(G,X,U)$ of coboundaries. In particular every function that is $(G,X,U)$-cohomologous to a function of type $<k$, is a function of type $<k$.}
		\item {A function $f:G\times X\rightarrow U$ is a phase polynomial of degree $<k$ if and only if $d^{[k]}f=1$, $\mu^{[k]}$-almost everywhere. In particular every phase polynomial of degree $<k$ is of type $<k$.}
		\item{If $f:G\times X\rightarrow U$ is a $(G,X,U)$-coboundary, then $d^{[k]}f:G\times X^{[k]}\rightarrow U$ is a $(G,X^{[k]},U)$-coboundary.}
	\end{enumerate}
	
\end{lem}
We also recall some properties of functions of type $<k$ and phase polynomials from \cite[Lemma 1.15]{Berg& tao & ziegler}.
\begin{lem}  \label{PPP}
	Let $G$ be a countable abelian group, $X$ be a $G$-system and $k\geq 0$.
	\begin{enumerate}[(i)]
		\item{(monotonicity) We have $P_{<k}(X,U)\subseteq P_{<k+1}(X,U)$. In particular $Abr_{<k}(X)\leq Abr_{<k+1}(X)$.}
		\item {(Homomorphism) $P_{<k}(X,U)$ is a group under pointwise multiplication, and for each $h\in H$ $\Delta_h$ is a homomorphism from $P_{<k+1}(X,U)$ to $P_{<k}(X,U)$.}
		\item {(Polynomiality criterion) Conversely, if $\phi:X\rightarrow U$ is measurable and for every $g\in G$, $\Delta_g\phi\in P_{<k}(X,U)$ then $\phi\in P_{<k+1}(X,U)$.}
		\item {(Functoriality) If $Y$ is a factor of $X$ then the pullback $(\pi^X_Y)^\star$ is a homomorphism from $P_{<k}(Y,U)$ to $P_{<k}(X,U)$. Conversely if $f:Y\rightarrow U$ is such that $(\pi^X_Y)^\star f\in P_{<k}(X,U)$ then $f\in P_{<k}(Y,U)$.}
		
	\end{enumerate}
\end{lem}
We can now formulate the main result of Bergelson Tao and Ziegler \cite[Theorem 1.20]{Berg& tao & ziegler}.
\begin{thm}  [Structure theorem for $Z_{<k}(X)$ for ergodic $\mathbb{F}_p^\omega$-systems] \label{BTZ} There exists a constant $C(k)$ such that for any ergodic $\mathbb{F}_p^\omega$-system $X$, $L^2(Z_{<k}(X))$ is generated by phase polynomials of degree $<C(k)$.  Moreover if $p$ is sufficiently large with respect to $k$ then $C(k)=k$.
\end{thm}
\subsection{Main results}
We say that a system $X$ is of order $<k$ if $X=Z_{<k}(X)$. We begin with the following result of Host and Kra \cite[Proposition 6.3]{HK} generalized for arbitrary discrete countable abelian group action.
\begin{prop}  [Order $<k+1$ systems are abelian extensions of order $<k$ systems]\label{abelext:prop} Let $G$ be a discrete countable abelian group, let $k\geq 1$ and $X$ be an ergodic $G$-system of order $<k+1$. Then $X$ is an abelian extension $X=Z_{<k}(X)\times_{\rho} U$ for some compact abelian metrizable group $U$ and a cocycle $\rho:G\times Z_{<k}(X)\rightarrow U$ of type $<k$.
\end{prop}
In other words, for a countable discrete abelian group $G$, an ergodic $G$-system of order $<k+1$ is isomorphic to a tower of abelian extensions \begin{equation} \label{structure}U_0\times_{\rho_1}U_1\times...\times_{\rho_k}U_k
\end{equation} such that for each $1\leq i\leq k$, $\rho_i:G\times Z_{<i-1}(X)\rightarrow U_i$ is a cocycle of type $<i$. We call $U_1,...,U_k$ the structure groups of $X$. We are particularly interested in the structure of these groups. We need the following definitions.
	\begin{defn} [Totally disconnected systems and Weyl systems] \label{TD:def} Let $k\geq 1$. Let $G$ be a countable abelian group and $X$ be an ergodic $G$-system of order $<k$. We write $X=U_0\times_{\rho_1}U_1\times...\times_{\rho_{k-1}}U_{k-1}$ as in equation (\ref{structure}).
		\begin{itemize}
			\item {We say that $X$ is a totally disconnected system if $U_0,U_1,...,U_{k-1}$ are totally disconnected groups.}
			\item {We say that $X$ is a Weyl system if for every $1\leq i \leq k-1$ the cocycle $\rho_i$ is a phase polynomial.}
		\end{itemize} Note that we will show that totally disconnected systems are isomorphic to Weyl systems (Theorem \ref{TDisweyl}).
	\end{defn}
	
	We are particularly interested in systems whose abelian extensions by cocycles of finite type are Abramov of some finite order (see Definition \ref{Abramov:def}). More formally, we have the following definition.
\begin{defn} \label{stronglyAbr}
    Let $X$ be an ergodic $G$-system. We say that $X$ is \emph{strongly Abramov} if for every $m\in\mathbb{N}$ there exists $l_m\in\mathbb{N}$ such that for any compact abelian group $U$ and a cocycle $\rho:G\times X\rightarrow U$ of type $<m$ the extension $X\times_\rho U$ is Abramov of order $<l_m$.
\end{defn}
Note that if $X$ is strongly Abramov, then it is Abramov of order $<l_0$. To see this consider the trivial extension of $X$ with the trivial group and the trivial cocycle (which is of type $<0$).\\

For a natural number $m\geq 0$ and a $G$-system $X$, we define the cohomology class $H^1_{<m}(G,X,S^1) = Z^1_{<m}(G,X,S^1)/B^1(G,X,S^1)$ to be the set of all cocycles of type $<m$ modulo coboundary. 
Our first results is the following equivalent characterization of the strongly Abramov property. 
\begin{thm} [A criterion for the strong Abramov property] \label{Mainr:thm}
	Let $P$ be a countable (unbounded) multiset of primes and let $G=\bigoplus_{p\in P}\mathbb{F}_p$. Let $X=Z_{<k}(X)$ be an ergodic $G$-system of order $<k$. Then the factors $Z_{<1}(X),Z_{<2}(X),...,Z_{<k}(X)$ are strongly Abramov if and only if there exists a totally disconnected factor $Y$ of $X$ such that for every $1\leq l \leq k$ the homomorphism induced by the factor map $\pi_l:Z_{<l}(X)\rightarrow Z_{<l}(Y)$
	$$\pi_l^\star:H^1_{<m}(G,Z_{<l}(Y),S^1)\rightarrow H^1_{<m}(G,Z_{<l}(X),S^1)$$ is onto for every $m\in\mathbb{N}$. Moreover, we can take $l_m=O_{k}(1)$\footnote{Here we fix $m$. Namely, there may be a different bound for different $m$'s.} and if $k,m\leq \min P$ we can take $l_m=m+1$.
\end{thm}
In particular, we deduce the following result.
\begin{cor} \label{cor} Let $G$ be as in Theorem \ref{Mainr:thm} and let $X$ be an ergodic totally disconnected $G$-system of order $<k$. Fix an integer $m\in\mathbb{N}$, then any cocycle $\rho:G\times X\rightarrow S^1$ of type $<m$ is $(G,X,S^1)$-cohomologous to a phase polynomial of degree $<O_{k,m}(1)$. If $k,m<\min P$ then it is cohomologous to a phase polynomial of degree $<m$.
\end{cor}
Since all ergodic $\mathbb{F}_p^\omega$-systems of order $<k$ are totally disconnected (see Theorem \ref{ptorsion}) this result generalizes Theorem \ref{BTZ}.
A priori, it is not clear that systems which are not strongly Abramov exists. In fact, Bergelson Tao and Ziegler \cite[Theorem 4.5]{Berg& tao & ziegler} proved that in the case where $P$ is bounded every system is strongly Abramov. We show that when $P$ is unbounded, there exists an ergodic Abramov system which is not strongly Abramov.
\begin{thm} [An Abramov system that is not strongly Abramov]\label{example:thm}
	Let $P$ be the set of prime numbers. There exists a solenoid\footnote{A solenoid is a compact abelian finite dimensional group that is not a Lie group. These are known for their pathological properties.} $U$ with an ergodic $\bigoplus_{p\in P}\mathbb{F}_p$ action such that $(U,\bigoplus_{p\in P}\mathbb{F}_p)$ is of order $<2$ and a cocycle $\rho:\bigoplus_{p\in P}\mathbb{F}_p\times U\rightarrow S^1$ of type $<2$ that is not $(\bigoplus_{p\in P}\mathbb{F}_p,U,S^1)$-cohomologous to a phase polynomial of any degree. The extension $U\times_\rho S^1$ is an ergodic $\bigoplus_{p\in P}\mathbb{F}_p$-system of order $<3$ that is not Abramov of any order.
\end{thm}
This example is based on the work of Furstenberg and Weiss \cite{Furstenberg}. They showed that in the case of $\mathbb{Z}$-actions, not all systems are strongly Abramov. (In \cite{HK02} Host and Kra worked out their example in details).
In Furstenberg and Weiss' example, the group $U$ is the $2$ dimensional torus. While a torus can be given a structure of an ergodic $\bigoplus_{p\in P}\mathbb{F}_p$-system (Example \ref{Example}), it is impossible to generalize their example for $\bigoplus_{p\in P}\mathbb{F}_p$-actions with $U$ being a torus. In Theorem \ref{Main:thm} below we show that the torus and a much larger family of non-pathological ergodic systems are strongly Abramov. Before we formulate the theorem we recall the following results of Host and Kra \cite{HK} for $\mathbb{Z}$-actions and the result of Bergelson Tao and Ziegler \cite{Berg& tao & ziegler} for $\mathbb{F}_p^\omega$-actions.\\

Let $X$ be an ergodic $G$-system of order $<k$ and let $U_0,U_1,...,U_{k-1}$ be its structure groups as in (\ref{structure}). We say that $X$ is a Toral system if $U_1$ is a Lie group and for all $2\leq i\leq k-1$, $U_i$ is isomorphic to a finite dimensional torus. Host and Kra proved the following result \cite[Theorem 10.3]{HK}.
\begin{thm} Let $X$ be an ergodic $\mathbb{Z}$-system of order $<k$. Then $X$ is an inverse limit of Toral systems of order $<k$.
\end{thm}
We say that a group $(U,\cdot)$ is $n$-torsion if $u^n=1_U$ for all $u\in U$. Bergelson Tao and Ziegler proved the following result \cite[Theorem 4.8]{Berg& tao & ziegler}.
\begin{thm} \label{ptorsion}
	Let $X$ be an ergodic $\mathbb{F}_p^\omega$-system of order $<k$. Then every structure group $U_0,U_1,...,U_{k-1}$ is a $p^m$-torsion group\footnote{A group $(U,\cdot)$ is $n$-torsion for some $n\in\mathbb{N}$ if $u^n=1_U$ for every $u\in U$.} for some $m=O_k(1)$. In particular $X$ is a totally disconnected system.
\end{thm}
Let $X$ be a $G$-system and $\rho:G\times X\rightarrow U$ be a cocycle into some compact abelian group $U$. If $V$ is a quotient of $U$ with quotient map $p:U\rightarrow V$ then we refer to the cocycle $p\circ\rho:G\times X\rightarrow V$ as the projection of $\rho$ to the group $V$. \\ The next definition captures all non-pathological systems including Toral systems and totally disconnected systems. 
\begin{defn} [Splitting condition] \label{split:def}
	Let $G$ be a countable discrete abelian group and let $X$ be an ergodic $G$-system of order $<k$ and write $X=U_0\times_{\rho_1} U_1\times...\times_{\rho_{k-1}}U_{k-1}$ as in (\ref{structure}). If for every $0\leq i\leq k-1$ the group $U_i$ is isomorphic to $T_i\times D_i$ where $T_i$ is a torus (possibly zero dimensional, or infinite dimensional) and $D_i$ is totally disconnected and the projection of each $\rho_i$ to $D_i$ is invariant under $T_1\times...\times T_{i-1}$, then we say that $X$ satisfies the splitting condition (or that $X$ splits in short). In this case we denote by $\mathcal{T}(X):=\prod_{i=0}^{k-1}T_i$ the torus part of $X$ and by $D(X):=\prod_{i=0}^{k-1} D_i$ the totally disconnected part of $X$. 
\end{defn} 

\begin{rem}
	The assumption that the projection of $\rho_i$ to $D_i$ is invariant under the torus is necessary to avoid pathological systems such as in Theorem \ref{example:thm}. For instance there exists an ergodic system $(\mathbb{T},\bigoplus_{p\in P}\mathbb{F}_p)$ where $\mathbb{T}$ is a torus and a cocycle $\rho:\bigoplus_{p\in P}\mathbb{F}_p\times\mathbb{T}\rightarrow \Delta$ (which is non-constant) into a disconnected group $\Delta$ such that the extension $X:=\mathbb{T}\times_\rho \Delta$ is isomorphic to a solenoid. 
\end{rem}
Our main result implies that these systems are strongly Abramov.
\begin{thm} [Splitting implies strongly Abramov]\label{Main:thm} 
    Let $P$ be a countable (unbounded) multiset of primes and $G=\bigoplus_{p\in P}\mathbb{F}_p$. Let $k,m\in\mathbb{N}$ and let $X$ be an ergodic $G$-system of order $<k$ which splits. Then any cocycle $\rho:G\times X\rightarrow S^1$ of type $<m$ is $(G,X,S^1)$-cohomologous to a phase polynomial of degree $<O_{k,m}(1)$.
\end{thm}
We also prove an exact version in the case where $\min P$ is large. \begin{thm}\label{MainH:thm}
Under the assumptions of Theorem \ref{Main:thm}, if in addition we have $k,m<\min P $, then $\rho$ is $(G,X,S^1)$-cohomologous to a phase polynomial of degree $<m$.
\end{thm}
Theorem \ref{Main:thm} is the main result in this paper and it implies Theorem \ref{Mainr:thm} (see section \ref{proof}). In order to demonstrate some of our difficulties, we give a simple example of an ergodic $\bigoplus_{p\in P}\mathbb{F}_p$ action on the torus.
\begin{example}  [The torus as an ergodic $\bigoplus_{p\in P} \mathbb{F}_p$-system]\label{Example}
	Consider the circle $X=S^1$ with the Borel $\sigma$-algebra and the Haar measure. We define an action of $\bigoplus_{p\in P} \mathbb{F}_p$ on $X$ by $T_g x = \varphi(g)x$, where the homomorphism $\varphi:\bigoplus_{p\in P} \mathbb{F}_p\rightarrow S^1$ is given by the formula $$\varphi((g_p)_{p\in P}) = \prod_{p\in P} w_p^{g_p}$$
	where $w_p=e^{\frac{2\pi i}{p}}$ is the first root of unity of degree $p$. This formula is well-defined because $g_p=0$ for all but finitely many $p\in P$ (including multiplicities).\\
	Let $f\in L^\infty(X)$ be an invariant function. We write $f(x)=\sum_{n\in\mathbb{Z}} a_n x^n$ for the Fourier series of $f$. By comparing the Fourier coefficients of $f$ and $f\circ T_g$ we see that $a_n=0$ whenever $\varphi(g)^n\not=1$. In particular, if $P$ is unbounded, then for every $n\not = 0$ there exists $g\in \bigoplus_{p\in P}\mathbb{F}_p$ such that $\varphi(g)^n\not = 1$. This implies that $f$ must be a constant and $X:=(S^1,\bigoplus_{p\in P}\mathbb{F}_p)$ is ergodic. Moreover, it is easy to see that the characters $\chi_n(z)=z^n$ form an orthonormal basis of eigenfunctions in $L^2(X)$. From this it follows that $X$ is a system of order $<2$ and every phase polynomial of degree $<2$ is a constant multiple of a character.
	Observe the following fact.\\
	\textbf{Claim:} Let $X$ be as in Example \ref{Example}. Then every phase polynomial $F:X\rightarrow S^1$ is of degree $<2$ (There are no phase polynomials of higher degree).
	\begin{proof}  Fix $m\in\mathbb{N}$ and let $F:X\rightarrow S^1$ be a phase polynomial of degree $<m$. By Corollary \ref{ker:cor}, the map $s\mapsto \Delta_s F$ from $S^1$ to $P_{<m}(X)$ takes values in $P_{<1}(X)$. In ergodic systems phase polynomials of degree $<1$ are constants. Therefore, there exists a function $\chi:X\rightarrow S^1$ such that $\Delta_s F = \chi(s)$ for every $s\in X$. Using the cocycle identity (i.e. $\Delta_{st} F(x) = \Delta_s F(tx) \Delta_t F(x)$) we conclude that $\chi(st)=\chi(s)\chi(t)$ and $\chi$ is a character. It follows that $\Delta_s (F(x)/\chi(x))=1$ for every $s\in S^1$. Hence $F/\chi$ is a constant. Since $\chi$ is a phase polynomial of degree $<2$, we conclude that so is $F$.
	\end{proof}
\end{example}
	From this example we see that many results of Bergelson Tao and Ziegler \cite{Berg& tao & ziegler} for $\mathbb{F}_p^\omega$-systems could not be generalized directly for $\bigoplus_{p\in P}\mathbb{F}_p$ where $P$ is unbounded, due to several new difficulties including the following:
	
	\begin{enumerate}
		\item {Ergodic $\bigoplus_{p\in P} \mathbb{F}_p$-systems of finite order need not be totally disconnected.}
		\item {Phase polynomials need not take finitely many values.}
		\item {From the last claim we see that some phase polynomials, like the identity map, need not have phase polynomial roots.\footnote{By a phase polynomial root for a phase polynomial $p:X\rightarrow S^1$ we mean a phase polynomial $q$ such that $q^n = p$ for some $n>1$.}}
	\end{enumerate}

\subsubsection{Nilpotent systems, limit formula and Khintchine-type recurrence}	
We prove a structure result for systems of order $<3$ from which we deduce a limit formula and a Khintchine-type recurrence theorem.
\begin{defn} \label{nilpsystem:def}
	Let $G=\bigoplus_{p\in P}\mathbb{F}_p$. For an ergodic $G$-system $X$ let $\mG(X)$ be the Host-Kra group of $X$ as defined in \cite[Section 5]{HK} (see also Definition \ref{HKGroup:def}). We say that $X$ is a $2$-step nilpotent $G$-system if the Host-Kra group of $X$ is $2$-step nilpotent and it acts transitively on $X$. 
\end{defn}
\begin{rem}
In this case $X=(\mG/\Gamma,\mathcal{B},\mu,G)$ where $\Gamma$ is the stabilizer of some point $x\in X$, $\mathcal{B}$ is the Borel $\sigma$-algebra, $\mu$ is the Haar measure (which always exists because nilpotent groups are uni-modular) and the action of $g\in G$ is given by a left multiplication by $\varphi(g)\in\mathcal{G}$ for some homomorphism $\varphi:G\rightarrow\mathcal{G}$.
\end{rem}
 We prove the following counterpart of Host and Kra's structure theorem for systems of order $<3$.
\begin{thm}\label{nilpotentstructure}
	Let $X$ be an ergodic $G$-system of order $<3$. Then $X$ is an inverse limit of finite dimensional $2$-step nilpotent systems.
\end{thm}
Let $X$ be a $2$-step nilpotent system. We prove a (pointwise) limit formula for three term ergodic averages $$\lim_{N\rightarrow\infty}\mathbb{E}_{g\in \Phi_N} T_gf_1(x)T_{2g}f_2(x)T_{3g}f_3(x)$$ for any $f_1,f_2,f_3\in L^\infty(X)$ along a F{\o}lner sequence of $G$ (See Lemma \ref{formula} for the exact formulation). Recall that a subset $A\subseteq G$ is syndetic if there exists a finite subset $C\subseteq G$ such that $A+C = \{a+c : a\in A, c\in C\} = G$. Using the formula above we deduce the following Khintchine type recurrence for $\bigoplus_{p\in P}\mathbb{F}_p$-systems.
\begin{thm} \label{recurrence:thm}
	Let $\mathcal{P}$ be a multiset of primes and suppose that $\min_{p\in\mathcal{P}}p > 3$. Then, for every ergodic $\bigoplus_{p\in \mathcal{P}}\mathbb{F}_p$-system $(X,\mu)$, every measurable set $A\subseteq X$ of positive measure and every $\varepsilon>0$, the set $$\{g\in\bigoplus_{p\in\mathcal{P}}\mathbb{F}_p : \mu(A\cap T_g A \cap T_{2g} A \cap T_{3g} A)\geq \mu(A)^4-\varepsilon\}$$ is syndetic.
\end{thm} 
\begin{rem}
It is also possible to prove a counterpart for $3$-term progressions, namely that $$\{g\in\bigoplus_{p\in\mathcal{P}}\mathbb{F}_p : \mu(A\cap T_g A\cap T_{2g}A)\geq \mu(A)^3-\varepsilon\}$$ is syndetic. One can also extend this result for other configurations. For instance if $c_0,c_1,c_2,c_3\leq \min_{p\in\mathcal{P}}p$ are such that $c_i+c_j=c_k+c_l$ for some permutation $\{i,j,k,l\}$ of $\{0,1,2,3\}$, then the same argument as we provide in section \ref{khintchine} allow to replace $\mu(A\cap T_g A\cap T_{2g} A\cap T_{3g} A)$ with $\mu(T_{c_0g}A\cap T_{c_1g}A\cap T_{c_2g}A\cap T_{c_3g}A)$ in the theorem above.
\end{rem} 

\subsection{The structure of the paper}
Most of the paper (sections \ref{sr:sec}-\ref{high:sec}) will be devoted to the proof of Theorem \ref{Main:thm} which is also the main component in the proof of Theorem \ref{Mainr:thm} (section \ref{proof}). The proof follows the ideas of Bergelson Tao and Ziegler \cite{Berg& tao & ziegler} and those of Host and Kra \cite{HK} with various modifications. In section \ref{example} we give an example of an ergodic $\bigoplus_{p\in \mathcal{P}}\mathbb{F}_p$ Kronecker system (system of order $<2$) that is not strongly Abramov as in Theorem \ref{example:thm}. In the next section we show that Theorem \ref{Main:thm} implies a structure theorem for all ergodic $\bigoplus_{p\in \mathcal{P}}\mathbb{F}_p$-systems of order $<3$ as an inverse limit of $2$-step nilpotent systems. Finally, we use this structure theorem to prove a limit formula from which we deduce the Khintchine type recurrence as in Theorem \ref{recurrence:thm}.\\

\textbf{Acknowledgment} I would like to to thank my adviser Prof. Tamar Ziegler for many helpful discussions and suggestions. I also thank the anonymous referee for many valuable comments that contributed to the clarity of the paper. The author is supported by an ERC grant ErgComNum 682150.

\section{Standard reductions} \label{sr:sec}
Throughout the rest of the paper it will be convenient to use the letter $G$ to denote the group $\bigoplus_{p\in P}\mathbb{F}_p$ for some fixed and possibly unbounded countable multiset of primes $P$. We need the following definition.
\begin{defn} [Automorphism] \label{Aut:def} Let $X$ be a $G$-system. A measure-preserving transformation $u:X\rightarrow X$ is called an \textit{automorphism} if the induced action on $L^2(X)$ by $V_u(f)=f\circ u$ commutes with the action of $G$. In this case we define the multiplicative derivative with respect to $u$ by $\Delta_u f = V_uf\cdot \overline{f}$.
\end{defn}
Automorphisms arise naturally from Host-Kra's theory. For instance, given an abelian extension $Y\times_\rho U$, the group $U$ acts on this extension by automorphisms defined by $V_u (y,t)=(y,tu)$.\\
Our next goal is to show that Theorem \ref{Main:thm} follows from the following two theorems. The first asserts that in order to prove Theorem \ref{Main:thm} it suffices to do so on a totally disconnected system. We begin with the following definition.
\begin{defn} \label{Conze-Lesigne equation}
	Let $m\geq 1$ be a natural number. Let $X$ be an ergodic $G$-system and $f:G\times X\rightarrow S^1$ be a function of type $<m$. We say that $f$ is a C.L.\ function\footnote{Equation (\ref{Conze-Lesigne}) is named after the mathematicians Conze and Lesigne who studied these equations in the case where $m=2$, see \cite{CL84},\cite{CL87},\cite{CL88}.} if for every automorphism $t:X\rightarrow X$ there exists a phase polynomial $p_t\in P_{<m}(G,X,S^1)$ and a measurable map $F_t:X\rightarrow S^1$ such that for all $g\in G$
		\begin{equation} \label{Conze-Lesigne}\Delta_t f(g,x) = p_t(g,x)\cdot \Delta_g F_t(x).
		\end{equation}
		If the equation is satisfied only for a certain group of automorphisms we say that $f$ is a C.L.\ cocycle with respect to that group.
\end{defn}
\begin{thm} [Reduction to a totally disconnected factor]\label{TDred:thm}
	Let $k\geq 1$, and $X$ be an ergodic $G$-system of order $<k$ which splits. Then there exists a totally disconnected factor $Y$, with a factor map $\pi:X\rightarrow Y$ such that the following holds:\\
	For every $m\in\mathbb{N}$ and a cocycle $\rho:G\times X\rightarrow S^1$ of type $<m$, if $\rho$ is a C.L.\ cocycle, then $\rho$ is $(G,X,S^1)$-cohomologous to $\pi^\star \rho'$ for a cocycle $\rho':G\times  Y\rightarrow S^1$.
\end{thm}
The second theorem is a version of Theorem \ref{Main:thm} for totally disconnected systems. As in \cite[Theorem 5.4]{Berg& tao & ziegler} the theorem holds for general functions of type $<m$ that may not be cocycles.

\begin{thm} [The totally disconnected case]\label{MainT:thm} Let $k,m\geq 0$ and let $X$ be an ergodic totally disconnected $G$-system of order $<k$. Let $f:G\times X\rightarrow S^1$ be a function of type $<m$. Then $f$ is $(G,X,S^1)$-cohomologous to a phase polynomial of degree $<O_{k,m}(1)$.
\end{thm}
Note that in retrospect, knowing that Theorem \ref{Mainr:thm} holds, reducing the proof of Theorem \ref{Main:thm} to a totally disconnected case is a reasonable approach.
\subsection{Proof of Theorem \ref{Main:thm} assuming Theorem \ref{TDred:thm} and Theorem \ref{MainT:thm}}
We recall some results from Bergelson Tao and Ziegler, \cite[Proposition 8.11 and Lemma 5.3]{Berg& tao & ziegler}.
\begin{lem}  [Descent of type for cocycles]  \label{Cdec:lem}
	Let $k\geq 0$ and let $X$ be an ergodic $G$-system of order $<k$. Let $Y$ be a factor of $X$, with factor map $\pi:X\rightarrow Y$. Suppose that $\rho:G\times Y\rightarrow S^1$ is a cocycle. If $\pi^\star \rho$ is of type $<k$, then $\rho$ is of type $<k$.
\end{lem}
\begin{lem}  [Differentiation Lemma] \label{dif:lem} Let $k\geq 1$, and let $X$ be an ergodic $G$-system. Let $f:G\times X\rightarrow S^1$ be a function of type $<m$ for some $m\geq 1$. Then for every automorphism $t:X\rightarrow X$ which preserves $Z_{<k}(X)$ the function $\Delta_t f(x) := f(tx)\cdot\overline{f}(x)$ is of type $<m-\min(m,k)$.
\end{lem}
In \cite{Berg& tao & ziegler} Bergelson Tao and Ziegler assume that the automorphism $t$ in the lemma above is a transformation of the form $V_t(y,u) = (y,tu)$ where $X=Y\times_\rho U$ is an abelian group extension of a $G$-system $Y$. However the proof does not use this fact and the result holds for any automorphism as stated above.\\ 

We prove Theorem \ref{Main:thm} assuming Theorem \ref{TDred:thm} and Theorem \ref{MainT:thm}. This proof is similar to the proof of \cite[Theorem 5.4]{Berg& tao & ziegler}.
\begin{proof} [Proof of Theorem \ref{Main:thm} assuming Theorem \ref{TDred:thm} and Theorem \ref{MainT:thm}] \label{proof1}
	Let $k,m,X,\rho$ be as in Theorem \ref{Main:thm}. We claim that for every $0\leq j \leq m$ and automorphisms $t_1,...,t_j:X\rightarrow X$ we have,
	\begin{equation}\label{multiCL}\Delta_{t_1}...\Delta_{t_j}\rho \in P_{<O_{k,m,j}(1)}(G,X,S^1)\cdot B^1(G,X,S^1).
	\end{equation}
	We prove this by downward induction on $j$.
	For $j=m$ the claim follows by iterating Lemma \ref{dif:lem} (with no polynomial term). Fix $j<m$ and assume inductively that the claim is true for $j+1$. By Theorem \ref{TDred:thm} we can find a factor $Y$ with factor map $\pi:X\rightarrow Y$ so that  $\Delta_{t_1}...\Delta_{t_j} \rho $ is $(G,X,S^1)$-cohomologous to $\pi^\star \rho_{t_1,...,t_j}$ where $\rho_{t_1,...,t_j}:G\times Y\rightarrow S^1$ is a cocycle. By assumption $\rho$ is of type $<m$. Since $t_1,...,t_j$ commute with the $G$-action, $\Delta_{t_1}...\Delta_{t_j}\rho$ is also of type $<m$ and by Lemma \ref{PP} (ii) so is $\pi^\star \rho_{t_1,...,t_j}$. Thus, Lemma \ref{Cdec:lem} implies that $\rho_{t_1,...,t_j}:G\times Y\rightarrow S^1$ is a cocycle of type $<m$. Since $\rho_{t_1,...,t_j}$ is defined on a totally disconnected system $Y$, we see from Theorem \ref{MainT:thm} that $\rho_{t_1,...,t_j}$ is $(G,Y,S^1)$-cohomologous to a phase polynomial $P_{t_1,...,t_j}:G\times Y\rightarrow S^1$ of degree $<O_{k,m,j}(1)$. Lifting everything up by $\pi$, we conclude that $\rho_{t_1,...,t_j}$ is $(G,X,S^1)$-cohomologous to $\pi^\star P_{t_1,...,t_j}$ which by Lemma \ref{PPP} (iv) is a phase polynomial of degree $<O_{k,m,j}(1)$. This completes the proof of the claim. Theorem \ref{Main:thm} now follows from the case $j=0$.
\end{proof}
The rest of the paper is devoted to the proof of Theorem \ref{TDred:thm} and Theorem \ref{MainT:thm} assuming the induction hypothesis of Theorem \ref{Main:thm}.
\section{Reduction to a finite dimensional $U$}\label{fdred:sec}
In the next couple of sections we develop a theory which eventually allows us to reduce the proof of Theorem \ref{Main:thm} to the case where $X$ is totally disconnected. We work in general settings (we do not assume that $X$ splits) and so we potentially deal with some pathological groups. Since we work in full generality, one can adapt our proof to prove a more general version of Theorem \ref{Main:thm}. More concretely, it is possible to replace the torus in the splitting condition with any compact connected abelian group $U$ that has no non-trivial local isomorphism to $\hat G$ (i.e. there is no open neighborhood $U'$ of the identity in $U$ and a non-trivial map $\varphi:U'\rightarrow \hat G$ such that $\varphi(u\cdot u')=\varphi(u)\cdot \varphi(u')$ whenever $u,u',u\cdot u'\in U'$). For the sake of simplicity we only prove Theorem \ref{Main:thm} in the way formulated above.\\

We define a notion of dimension of compact abelian groups. First, we say that a compact abelian group is of dimension zero if it is totally disconnected. As for higher dimensions we have the following definition/proposition from  \cite[Theorem 8.22]{HM}.
\begin{defn} [Definition and properties of finite dimensional compact abelian groups] \label{FD:def} The following conditions are equivalent for a compact abelian group $H$ and a natural number $n$:
	
	\begin{enumerate}
		\item {There is a compact zero dimensional subgroup $\Delta$ of $H$ and an exact sequence $$1\rightarrow \Delta \rightarrow H\rightarrow \mathbb{T}^n\rightarrow 1$$ where $\mathbb{T}^n$ is the $n$-dimensional torus.}
		\item {There is a compact zero dimensional subgroup $\Delta$ of $H$ and a continuous quotient homomorphism $\varphi:\Delta\times\mathbb{R}^n\rightarrow H$ which has a discrete kernel.}
	\end{enumerate}
	We say that a compact abelian group $H$ is finite dimensional if it satisfies at least one of these conditions.
	
\end{defn}
\begin{rem}
	Note that condition $(1)$ in the previous definition is equivalent to the existence of some zero dimensional subgroup $D$ of $H$ and an exact sequence $$1\rightarrow D\rightarrow H\rightarrow \mathbb{T}^n\times C_k\rightarrow 1$$ for some finite group $C_k$. For if $H/D\cong \mathbb{T}^n\times C_k$, then $C_k$ is a subgroup of $H/D$. Hence, the isomorphism theorem implies that there exists $D\subseteq D'\subseteq H$ such that $D'/D\cong C_k$. It follows that $D'$ is also zero dimensional group and $H/D'\cong \mathbb{T}^n$.
\end{rem}
We need the following definition.
\begin{defn} [free action] 
	Let $U$ be a locally compact group acting on a probability space $X$ in a measure-preserving way. The action of $U$ is said to be free if $X$ is measure-equivalent to a system of the form $Y\times U$ and the action of $u\in U$ is given by $V_u(y,v)=(y,uv)$. 
\end{defn}
In this section we study C.L.\ cocycles $\rho:G\times X\rightarrow S^1$ with respect to some group $U$ of measure-preserving transformations on $X$. Note that in retrospect, if $X$ satisfies the splitting condition and Theorem \ref{Main:thm} holds, then any cocycle of type $<m$ is a C.L.\ cocycle. The main result in this section is the following proposition.
\begin{prop}\label{FD:prop}
	Let $X$ be an ergodic $G$-system. Let $U$ be a compact abelian group acting freely on $X$ and commuting with the action of $G$. Let $m\geq 0$ and $\rho:G\times X\rightarrow S^1$ be a C.L.\ cocycle of type $<m$ with respect to $U$. Then $\rho$ is $(G,X,S^1)$-cohomologous to a cocycle which is invariant under some closed connected subgroup $J$ of $U$ for which $U/J$ is a finite dimensional compact abelian group.
\end{prop}

We will take advantage of the following results of Bergelson Tao and Ziegler \cite[Lemma C.1 and Lemma C.4]{Berg& tao & ziegler}. 

\begin{lem}  [Separation Lemma]\label{sep:lem} Let $X$ be an ergodic $G$-system, let $k\geq 1$, and let $\phi,\psi\in P_{<k}(X,S^1)$ be such that $\phi/\psi$ is non-constant. Then $\|\phi-\psi\|_{L^2(X)} \geq \sqrt{2}/2^{k-2}$.
\end{lem}
Since $X$ is compact, $L^2(X)$ is separable. From this it follows that up to a constant multiple, there are only countably many phase polynomials of a given degree.\\
The next lemma asserts that the components in the Conze-Lesigne equations can be chosen in a measurable way.
\begin{lem}  [Measure selection lemma] \label{sel:lem} Let $X$ be an ergodic $G$-system and $k\geq 1$. Let $U$ be a compact abelian group. If $u\mapsto h_u$ is Borel measurable map from $U$ to $\mathcal{P}_{<k}(G,X,S^1)\cdot \mathcal{B}^1(G,X,S^1) \subseteq \mathcal{M}(G,X,S^1)$ where $\mathcal{M}(G,X,S^1)$ is the group of measurable maps of the form $G\times X\rightarrow S^1$ endowed with the topology of convergence in measure, then there is a Borel measurable choice of $f_u,\psi_u$ (as functions from $U$ to $M(X,S^1)$ and $U$ to $P_{<k}(G,X,S^1)$ respectively) obeying that $h_u = \psi_u \cdot \Delta f_u$.
\end{lem}
The next lemma deals with C.L.\ cocycles in which the polynomial term is trivial. In this case, up to coboundary, the cocycles are linear on an open subgroup \cite[Lemma C.9]{HK}.
\begin{lem} [Straightening almost translation-invariant cocycles] \label{cob:lem}
	Let $X$ be an ergodic $G$-system, let $K$ be a compact abelian group acting freely on $X$ by automorphisms and let $\rho:G\times X\rightarrow S^1$ be such that $\Delta_k\rho$ is a $(G,X,S^1)$-coboundary for every $k\in K$. Then $\rho$ is $(G,X,S^1)$-cohomologous to a cocycle which is invariant under the action of some open subgroup of $K$. 
\end{lem}
\begin{rem}
	Note that if $K$ is connected, then it has no non-trivial open subgroups (see Lemma \ref{connectedcomponent}). In this case we have that $\rho$ is $(G,X,S^1)$-cohomologous to a function which is invariant under $K$. Moreover, it is important to note that such result does not work for cocycles which takes values in an arbitrary compact abelian group. Finally, we note that lemma \ref{cob:lem} holds even if one replaces the cocycle $\rho$ with an arbitrary function $f:G\times X\rightarrow S^1$
\end{rem}
Taking advantage of the fact that there are only countably many phase polynomials of a given degree (modulo constants), we can assume that the polynomial term in all C.L.\ equations are almost linear. Formally we have the following lemma.
\begin{lem} [Linearization of the polynomial term] \label{lin:lem}
	Let $X$ be an ergodic $G$-system and let $U$ be a compact abelian group acting freely on $X$ by automorphisms. Let $\rho:G\times X\rightarrow S^1$ be a cocycle and suppose that there exists $m\in\mathbb{N}$ such that for every $u\in U$ there exist phase polynomials $p_u\in P_{<m}(G,X,S^1)$ and a measurable map $F_u:X\rightarrow S^1$ such that $\Delta_u\rho = p_u\cdot \Delta F_u$. Then, there exists a measurable choice $u\mapsto p'_u$ and $u\mapsto F'_u$ such that $\Delta_u\rho = p'_u\cdot \Delta F'_u$ for phase polynomials $p'_u\in P_{<m}(G,X,S^1)$ which satisfies that $p'_{uv} = p'_u\cdot V_u p'_v$ whenever $u,v,uv\in U'$ where $U'$ is some open neighborhood of the identity in $U$.
\end{lem}
The proof of this lemma is given in \cite{Berg& tao & ziegler} as part of the proof of Proposition 6.1 (see in particular equation (6.5) in that proof).
\begin{rem}
	The idea of linearizing the term $p_u$ as in the lemma above was originally introduced by Furstenberg and Weiss in \cite{F&W}. Later, this idea was used by Ziegler in \cite{Z} and by Bergelson Tao and Ziegler in \cite{Berg& tao & ziegler} who studied the structure of the universal characteristic factors for the groups $\mathbb{Z}$ and $\mathbb{F}_p^\omega$, respectively.
\end{rem}

Finally, we prove the following Lemma about phase polynomial cocycles and connected groups. This lemma plays an important role in the proof of Proposition \ref{FD:prop}.

\begin{lem}\label{ker:lem}
	Let $X$ be an ergodic $G$-system, let $H$ be a compact abelian group acting freely on $X$ by automorphisms and let $p:H\rightarrow P_{<m}(G,X,S^1)\cap Z^1(G,X,S^1)$ be a measurable map. Suppose that $p$ satisfies $p(hk)=p(h)V_h p(k)$ for every $h,k\in H$. Then $p(h,\cdot)=1$ for every $h\in H_0$, where $H_0$ is the connected component of the identity in $H$.
\end{lem}
\begin{proof}
By evaluating the polynomials at $g$, we see that the map $p:H\rightarrow P_{<m}(G,X,S^1)\cap Z^1(G,X,S^1)$ induces a map $p_g:H\rightarrow P_{<m}(X,S^1)$ with the property that $p_g(hk) = p_g(h)V_hp_g(k)$ for every $h,k\in H$. Let $g\in G$ be arbitrary and let $H_0$ be the connected component of the identity in $H$. Since this group is a subgroup of any open subgroup of $H$, we conclude by applying Corollary \ref{ker:cor} that the image of the restriction of $p_g$ to $H_0$ is an element in $P_{<1}(X,S^1)$. Since this holds for every $g\in G$, we conclude that $p(H_0)$ is a subset of $P_{<1}(G,X,S^1)\cap Z^1(G,X,S^1)$. By ergodicity, $P_{<1}(G,X,S^1)$ consists of maps of the form $c:G\rightarrow S^1$ (constant in $x$). If in addition $c\in Z^1(G,X,S^1)$, then $c$ is a character of $G$. In other words, we can identify $P_{<1}(G,X,S^1)\cap Z^1(G,X,S^1)$ with the dual $\hat G$. Observe that the topology on $P_{<1}(G,X,S^1)\cap Z^1(G,X,S^1)$ as a subset of $\mathcal{M}(G,X,S^1)$ coincides with the natural topology on $\hat G$ (of pointwise convergence). In other words the identification of $P_{<1}(G,X,S^1)\cap Z^1(G,X,S^1)$ with $\hat G$ is an isomorphism of topological groups. Therefore, we can assume that $p|_{H_0}$ takes values in $\hat G$. Since $p(hk)=p(h)V_hp(k)=p(h)p(k)$ for every $h,k\in H$, we see that $p|_{H_0}:H_0\rightarrow \hat G$ is a measurable homomorphism. By Lemma \ref{AC:lem}, $p|_{H_0}$ is continuous. Since $H_0$ is connected and $\hat G$ is totally disconnected this map is trivial.
\end{proof}
We can finally prove Proposition \ref{FD:prop}.
\begin{proof}[Proof of Proposition \ref{FD:prop}]
	Let $X,U,\rho,m,p_u,F_u$ be as in Proposition \ref{FD:prop}. By Lemma \ref{sel:lem}, we may assume that $u\mapsto p_u$ and $u\mapsto F_u$ are measurable. Moreover, by Lemma \ref{lin:lem} we may assume that there exists an open neighborhood $U'$ of the identity in $U$ on which $p_u$ is a cocycle in $u$ (i.e. $p_{uv}=p_u V_u p_v$ whenever $u,v,uv\in U'$). It is well known that every compact abelian group can be approximated by Lie groups (Theorem \ref{approxLieGroups}). In other words, there exists a closed subgroup $J'$ contained in $U'$ such that $U/J'$ is a Lie group. Let $p:J'\rightarrow P_{<m}(G,X,S^1)$ be the map $j\mapsto p_j$. From Lemma \ref{ker:lem} it follows that $p$ is trivial on the connected component of the identity in $J'$. Let $J:=J'_0$, we conclude that $\Delta_j \rho$ is a coboundary for every $j\in J$. By Lemma \ref{cob:lem} and Lemma \ref{connectedcomponent} we have that $\rho$ is $(G,X,S^1)$-cohomologous to a cocycle which is invariant under the action of $J$.\\
	
	It is left to show that $U/J$ is a finite dimensional group. We have the following exact sequence $$1\rightarrow J'/J\rightarrow U/J\rightarrow U/J'\rightarrow 1$$
Here $J'/J$ is totally disconnected and $U/J'$ is a Lie group. By the structure theorem for compact abelian Lie groups (Theorem \ref{structureLieGroups}) we know that $U/J'$ is a finite extension of a Torus. Hence, $U/J'$ is a finite dimensional compact abelian group as in property (1) of Definition \ref{FD:def}.
\end{proof}
\section{Reduction to the totally disconnected group case} \label{tdred:sec}
In this section we work with a finite dimensional group $U$ acting on a system $X$. By the structure of finite dimensional compact abelian groups (Definition \ref{FD:def}), we can write $U\cong\left(\mathbb{R}^n\times\Delta\right) /\Gamma$ for some $n\in\mathbb{N}$, a totally disconnected compact abelian group $\Delta$ and a discrete subgroup $\Gamma$ of $\mathbb{R}^n\times \Delta$. Throughout this section it will be convenient to identify $\mathbb{R}^n\times\{1\}$ with $\mathbb{R}^n$ and $\{1\}\times \Delta$ with $\Delta$.
\begin{defn}
	Let $(H,+)$ and $(K,\cdot)$ be two abelian groups. We say that a map $\varphi:H\rightarrow K$ is a homomorphism on some set $A\subseteq H$ if $\varphi(x+y)=\varphi(x)\cdot\varphi(y)$ whenever $x,y\in A$.
\end{defn}
Below is a simple lemma about homomorphisms on open subsets of $\mathbb{R}^n$.
\begin{lem} [Lifting of homomorphisms on open sets of $\mathbb{R}^n$]\label{lift:lem} Let $H$ be a compact abelian group and let $\varphi:\mathbb{R}^n\rightarrow H$ be a measurable map.
	Suppose that there exists an open ball around zero $W\subseteq \mathbb{R}^n$ such that $\varphi(v+w)=\varphi(v)+\varphi(w)$ whenever $v,w\in W$. Then there exists a homomorphism $\tilde{\varphi}:\mathbb{R}^n\rightarrow H$ such that $\tilde{\varphi}$ and $\varphi$ agree on $W$.
\end{lem}
\begin{rem}
	Lemma \ref{lift:lem} is a special case of \cite[Lemma 4.7]{Z}. Also we note that if $\varphi$ is measurable then so is $\tilde{\varphi}$.
\end{rem}

\begin{proof}
	Write $W=B(0,r)$ for the open ball of radius $r$ around $0$ for some $r>0$. We extend $\varphi$ to a function $\varphi_1$ which is a homomorphism on $1.5W = B(0,1.5 r)$ as follows:\\
	If $a\not\in 3W$, let $\varphi_1(a)=1$ (or any other element in $H$). Otherwise, choose any $x\in W$ and $y\in 2W$ with $a=x+y$ and let $\varphi_1(a)=\varphi(x)\varphi(y)$. Assume for now that $\varphi_1$ is well defined (i.e. independent of the choice of $x,y$), we claim that $\varphi_1$ is a homomorphism on $1.5W$. Indeed, let $x,y\in 1.5W$ then write $x=1.5u$ and $y=1.5v$ for $u,v\in W$. We have, $$\varphi_1(1.5u+1.5v)=\varphi_1 (u+(0.5u+0.5v)+v)$$
	as $u,v\in W$ we see that $0.5u+0.5v\in W$ and $0.5u+1.5v\in 2W$. If $\varphi_1$ is well defined we have
	$$\varphi_1 (u+(0.5u+0.5v)+v)=\varphi(u)\varphi(0.5u+0.5v+v)=\varphi(u)\varphi(0.5u+0.5v)\varphi(v)$$
	where the last equality follows from the linearity of $\varphi$ on $W$. Using the linearity of $\varphi$ few more times we have
	$$\varphi(u)\varphi(0.5u+0.5v)\varphi(v)=\varphi(u)\varphi(0.5u)\varphi(0.5v)\varphi(v)=\varphi_1(1.5u)\varphi_1(1.5v)$$
	Combining everything we conclude that $\varphi_1(x+y)=\varphi_1(x)\varphi_1(y)$ as desired.\\ 
	
	It is left to show that $\varphi_1$ is well defined. Let $x,x'\in W$ and $y,y'\in 2W$ be so that $x+y=x'+y'$. We want to show that $\varphi(x)\varphi(y)=\varphi(x')\varphi(y')$.\\
	By assumption, $x-x'=x+(-x')$ is a sum of two elements in $W$ hence we have that $\varphi(x-x')=\varphi(x)\varphi(x')^{-1}$ and so $\varphi(x)=\varphi(x')\varphi(x-x')$. It is thus enough to prove that $\varphi(y')=\varphi(y)\varphi(x-x')$. We have 
	$$\varphi(\frac{y'}{2})=\varphi(\frac{x-x'}{2}+\frac{y}{2})=\varphi(\frac{x-x'}{2})\varphi(\frac{y}{2}) $$
	Now, as $\frac{y'}{2},\frac{(x-x')}{2}$ and $\frac{y}{2}$ are elements in $W$ we see that $$\varphi(y')=\varphi(\frac{y'}{2})^2=(\varphi(\frac{x-x'}{2}))^2(\varphi(\frac{y}{2}))^2=\varphi(x-x')\varphi(y)$$
	We conclude that $\varphi_1$ is well defined.  
	\\
	
	Replacing $\varphi$ with $\varphi_1$ and $W$ with $1.5W$, by the same arguments as before we can construct $\varphi_2$ which extends $\varphi_1$ on $1.5W$ and is a homomorphism on $(1.5)^2W$. Continuing this way, we deduce that there exists a sequence $\varphi_1,\varphi_2,...$ with $\varphi_{i}$ extending $\varphi_{i-1}$ for each $i$, such that $\varphi_n$ is a homomorphism on $(1.5)^nW$ for every $n\in\mathbb{N}$.\\
	
	Thus, for every $x\in \mathbb{R}^n$ there exists $m$ so that $x\in B(0,(1.5)^mr)$. Let $\tilde{\varphi}(x)=\varphi_m(x)$, it is easy to see that $\tilde{\varphi}$ is well defined, is a homomorphism and it agrees with $\varphi$ on $W$.
\end{proof}
Given a finite dimensional compact abelian group $U$, its connected component $U_0$ is also finite dimensional (see \cite[Corollary 8.24]{HM}). Let $\rho$ be a C.L.\ cocycle, we use the previous lemma to obtain some results on the map $p:U_0\rightarrow P_{<m}(G,X,S^1)$ given by $u\mapsto p_u$.
\begin{lem}\label{real:lem}
	Let $X$ be an ergodic $G$-system. Let $U$ be a finite dimensional compact abelian group acting freely on $X$ and commuting with the action of $G$. Let $m\geq 1$ and let $\rho:G\times X\rightarrow S^1$ be a C.L.\ cocycle of type $<m$ and write $\Delta_u \rho = p_u\cdot \Delta F_u$ as in (\ref{Conze-Lesigne}). We denote by $U_0$ the connected component of the identity in $U$ and write $U_0=\left(\mathbb{R}^n\times\Delta\right)/\Gamma$ for some $n\in\mathbb{N}$, a totally disconnected group $\Delta$ and a discrete subgroup $\Gamma$. Let $\pi:\mathbb{R}^n\times \Delta \rightarrow U_0$ be the projection map. Then,
	\begin{enumerate}
		\item{There exists an open subgroup $V$ of $\Delta$ such that $$p_{\pi(uv)}=p_{\pi(u)}p_{\pi(v)}$$ whenever $u,v\in V$.}
		\item {There exists an open ball $W\subseteq\mathbb{R}^n$ such that $p_{\pi(w)}=1$ for all $w\in W$.}
		\item {$p_{\pi(w)}$ is a $(G,X,S^1)$-coboundary for every $\omega\in \mathbb{R}^n$.}
	\end{enumerate}
\end{lem}
\begin{proof}
	In order to keep the same notations as above, we use the multiplicative notation to denote the additive operation of $\mathbb{R}^n$.\\
	
	Applying the linearization Lemma (Lemma \ref{lin:lem}), there exists an open neighborhood $U'$ of the identity in $U_0$ on which $u\mapsto p_u$ is a cocycle (i.e. $p_{uv}=p_u V_u p_v$ whenever $u,v\in U'$). Since $\pi$ is continuous, $\pi^{-1}(U')$ is an open subset of $\mathbb{R}^n\times \Delta$ around zero and therefore it contains an open subset of the form $W\times V$ where $W$ is an open subset of $\mathbb{R}^n$ and $V$ is an open subset of $\Delta$. By shrinking $W$ and $V$, we can assume that $W$ is an open ball around zero and by Proposition \ref{opensubgroup} that $V$ is an open subgroup of $\Delta$. Since $\pi(V)$ is a subgroup of $U_0$ and it lies inside $U'$, we conclude that $$p_{\pi(uv)} = p_{\pi(u)} V_{\pi(u)} p_{\pi(v)} = p_{\pi(u)}p_{\pi(v)}$$ where the last equality follows from Proposition \ref{pinv:prop}. This proves $(1)$.\\
	Let $p:U_0\rightarrow P_{<m}(G,X,S^1)$ be the map $u\mapsto p_u$. Then, $p\circ \pi : \mathbb{R}^n\times \Delta\rightarrow P_{<m}(G,X,S^1)$ is a homomorphism on $W\times V$. Applying Lemma \ref{lift:lem} once for every $v\in V$, we see that there exists a homomorphism $\tilde{p}:\mathbb{R}^n\times V\rightarrow P_{<m}(G,X,S^1)$ which extends $p\circ \pi$ on $W\times V$.  We restrict $\tilde{p}$ to $\mathbb{R}^n$. Since the index of $P_{<1}(X,S^1)$ in $P_{<m}(X,S^1)$ is at most countable (Lemma \ref{sep:lem}) and $\mathbb{R}^n$ has no proper measurable subgroups of countable index (Corollary \ref{openker}), we conclude that $\tilde{p}|_{\mathbb{R}^n}$ takes values in $P_{<1}(G,X,S^1)$. By ergodicity every element in $P_{<1}(G,X,S^1)$ is a constant. Hence, since $p_u$ are cocycles for every $u\in U$, we conclude that $p_u:G\rightarrow S^1$ is a character for every $u\in\mathbb{R}^n$. Therefore $\tilde{p}|_{\mathbb{R}^n}$ gives rise to a homomorphism from $\mathbb{R}^n$ into $\hat G$. Since $\mathbb{R}^n$ is divisible such a homomorphism must be trivial. In other words the image of $\tilde{p}|_{\mathbb{R}^n}$ is trivial. This completes the proof of $(2)$ because $\tilde{p}$ and $p\circ \pi $ agree on $W$ and $\tilde{p}(W)=\{1\}$.\\
	For $(3)$ we let $$H=\{s\in\mathbb{R}^n : p_{\pi(s)} \text{ is a coboundary}\}.$$ Using the facts that $p_1=1$, $p_{st}$ is cohomologous to $p_s V_s p_t = p_s p_t$ and $p_s$ is cohomologous to $\overline{p_{s^{-1}}}$ for every $s,t\in U_0$, we have that $H$ is a subgroup of $\mathbb{R}^n$. Moreover from $(2)$ we have that $H$ contains the open ball $W$, therefore $H=\mathbb{R}^n$. This proves $(3)$.
\end{proof}

In this situation we would like to use Lemma \ref{cob:lem} in order to eliminate the $\mathbb{R}^n$ part of $U_0$. Unfortunately this is impossible, because the image of $\mathbb{R}^n$ is not necessarily a compact subgroup of $U_0$. In fact the image of $\mathbb{R}^n$ is dense and it is closed if and only if $U_0$ is a torus. We have the following Proposition as a corollary.
\begin{prop} \label{con:prop}
	In the settings of Lemma \ref{real:lem}, if $U_0$ is a torus, then $\rho$ is $(G,X,S^1)$-cohomologous to a cocycle which is invariant under the action of $U_0$.
\end{prop}
Assuming that $X$ splits, we combine this with the results of the previous section to conclude the following result.
\begin{thm} [Eliminating the connected components of a system $X$ of order $<k$ which splits] \label{con:thm} Let $k\geq 1$ be such that Theorem \ref{Main:thm} has already been proven for smaller values of $k$. Let $X$ be an ergodic $G$-system of order $<k$ which splits. We can then write $X=U_0\times_{\rho_1} U_1\times...\times_{\rho_{k-1}}U_{k-1}$, where $\rho_i:G\times Z_{<i}(X)\rightarrow U_i$ is a cocycle of type $<i$ for all $1\leq i \leq k-1$.
	Let $m\geq 0$ and $\rho:G\times X\rightarrow S^1$ be a C.L.\ cocycle of type $<m$. Then $\rho$ is $(G,X,S^1)$-cohomologous to a cocycle $\tilde{\rho}:G\times X\rightarrow S^1$ which is invariant under the action of $U_{j,0}$ on $X$ by translations for all $0\leq j \leq k-1$.
\end{thm}

\begin{proof}
	Fix $1\leq j \leq k-1$. We prove the theorem by induction on $k$. If $k=1$, then $X$ is just a point and the claim in the theorem is trivial. Fix $k\geq 2$ and assume that the claim holds for smaller values of $k$. We replace the cocycles $\rho_{j+1},...,\rho_{k-1}$ with cocycles that are invariant under the action of $U_{j,0}$. We begin with $\rho_{j+1}:G\times Z_{<j+1}(X)\rightarrow U_{j+1}$. Since $\rho_{j+1}$ is of type $<j+1$, we can apply  Theorem \ref{Main:thm} to the projection of $\rho_{j+1}$ to the torus of $U_{j+1}$ (i.e. $T(U_{j+1})$) by applying it to each coordinate separately. We conclude that $\rho_{j+1}$ is $(G,Z_{<j+1}(X),T(U_{j+1}))$-cohomologous to a phase polynomial which, by Proposition \ref{pinv:prop}, is invariant with respect to the action of $U_{j,0}$. Now, since $X$ splits the projection of $\rho_{j+1}$ to the totally disconnected part of $U_{j+1}$ is also invariant under the action of $U_{j,0}$. Hence, by Lemma \ref{cob:lem} (applied with $K=U_{j,0}$) and remark \ref{coh:rem} we can replace $\rho_{j+1}$ with a cocycle $\rho'_{j+1}$ that is invariant under $U_{j,0}$. We conclude that $Z_{j+2}(X)$ is isomorphic as $G$-systems to $Z_{<j+1}(X)\times_{\rho'_{j+1}}U_{j+1}$ (see Remark \ref{coh:rem}). On the new system, since $\rho'_{j+1}$ is invariant with respect to $U_{j,0}$, we see that for every $g\in G$ the transformation $T_g$ commutes with the translations $V_u$ for $u\in U_{j,0}$.\footnote{To see this let $(s,t)\in Z_{<j+1}(X)\times U_{j+1}$. Then, for every $u\in U_{j,0}$ and $g\in G$ we have $T_gV_u (s,t) = (S_g V_u s,\rho_{j+1}'(g,us)t)$, where $S_g$ is the action of $G$ on $Z_{<j+1}(X)$ and we abuse notation and denote the translation by $u$ on $Z_{<j+1}(X)$ by $V_u$ as well. Now, since $V_u$ commutes with $S_g$ and $\rho_{j+1}'$ is invariant to $u$ we conclude that $T_gV_u (s,t) = (S_g V_u s,\rho_{j+1}'(g,V_us)t)=(V_u S_g s,\rho_{j+1}'(g,s)t) = V_u T_g(s,t)$.} This means that the same argument can be applied this time for $\rho_{j+2}$. Continuing this way (using induction) we can assume that $\rho_{j+1},...,\rho_{k-1}$ are invariant with respect to the action of $U_{j,0}$.\\ In this case it follows that $U_{j,0}$ acts on $X$ by automorphisms.\\ Thus, by Proposition \ref{FD:prop}, $\rho$ is $(G,X,S^1)$-cohomologous to a cocycle $\rho':G\times X\rightarrow S^1$ which is invariant with respect to the translations by some connected subgroup $J$ of $U_{j,0}$ for which $U_{j,0}/J$ is a finite dimensional torus. Since $\rho_{j+1},...,\rho_{k-1}$ are invariant with respect to the action of $U_{j,0}$, they are also invariant to the action of $J$. In particular, we can define a factor $Y := U_0\times_{\rho_1}U_1\times...\times_{\rho_j} U_j/J\times_{\rho_{j+1}'} U_{j+1}\times_{\rho_{j+2}'}...\times_{\rho_{k-1}'}U_{k-1}$ of $X$ where $\rho_{j+1}',...,\rho_{k-1}'$ are induced from $\rho_{j+1},...,\rho_{k-1}$ under the projection $U_j\rightarrow U_j/J$. Let $\pi:X\rightarrow Y$ be the factor map $\pi(u_0,...,u_{k-1}) = (u_0,...,u_j+J,...,u_{k-1})$. We have that $\rho'$ is equal to $\pi^\star \rho''$ where $\rho'':G\times Y\rightarrow S^1$ is a cocycle and $U_{j,0}/J$ is a finite dimensional torus acting freely on $Y$ and commuting with the $G$-action.\\
	We want to apply Proposition \ref{con:prop}, but first we have to show that $\rho''$ is a C.L.\ cocycle of type $<m$ with respect to $U_{j,0}/J$.\\
	As $\rho'$ is $(G,X,S^1)$-cohomologous to $\rho$ we have that 
	\begin{equation} \label{eq'}
	\Delta_u \rho' = p_u \cdot \Delta F_u
	\end{equation}
	for a phase polynomial $p_u\in P_{<m}(G,X,S^1)$ and a measurable map $F_u:X\rightarrow S^1$. Recall that $\rho'$ is invariant with respect to the action of $J$. Since $J$ is connected, we conclude by Proposition \ref{pinv:prop} and equation (\ref{eq'}) that $\Delta F_u$ is also invariant with respect to $J$. Since the action of $J$ commutes with the $G$-action, we conclude that $\Delta_j F$ is invariant. By ergodicity and since $\Delta_{jj'}F = \Delta_j F V_j \Delta_{j'} F$ for every $j,j'\in J$, we see that $\Delta_j F_u = \chi_u(j)$ for all $j\in J$ for some character $\chi_u:U_j\rightarrow S^1$.\footnote{Initially $\chi_u$ is a character of $J$ which we can lift (arbitrarily) to a character of $U_j$ using pontryagin duality.}
	
	Applying Theorem \ref{Main:thm}, we have that $\chi_u\circ\rho_{j-1}$ is $(G,Z_{<i}(X),S^1)$-cohomologous to a phase polynomial $q$ of degree $<O_k(1)$, i.e. $\chi_u\circ\rho_{j-1} = q\cdot\Delta F$. Let $\phi_{\chi_u} = \chi_u\cdot F$. Then $\phi_{\chi_u}:Z_{<j}(X)\rightarrow S^1$ is a phase polynomial of degree $<O_k(1)$, and it satisfies that $\Delta_j \phi_{\chi_u} = \chi_u(j)$ for all $j\in J$. Let $\tilde{\phi}_{\chi_u}=(\pi^X_{Z_{<j}(X)})^\star\phi_{\chi_u}$. We have that $F'_u:=F_u/\tilde{\phi}_{\chi_u}$ is invariant with respect to the action of $J$ and that $\Delta_u \rho' = p'_u\cdot \Delta F'_u$ where $p'_u = p_u\cdot \Delta \tilde{\phi}_{\chi_u}$ is a phase polynomial of degree $<O_{k,m}(1)$. Now $\rho'$, $p'_u$ and $F'_u$ are all invariant under $J$ and so there exists $\tilde{p}_u :G\times Y\rightarrow S^1$ and $\tilde{F}_u:Y\rightarrow S^1$ such that $p_u = \pi^\star \tilde{p}_u$ and $F_u = \pi^\star \tilde{F}_u$. It follows that $\Delta_u\rho'' = \tilde{p}_u \cdot \Delta \tilde{F}_u$ for all $u\in U_j/J$ and by functoriality we have that $\tilde{p}_u$ are also phase polynomials of degree $<O_{k,m}(1)$.
	
	We can therefore apply Proposition \ref{con:prop}. It follows that $\rho''$ is $(G,Y,S^1)$-cohomologous to a cocycle $\tilde{\rho}:G\times Y\rightarrow S^1$ which is invariant under $U_{j,0}/J$. Therefore, $\pi^\star \tilde{\rho}$ is $(G,X,S^1)$-cohomologous to $\rho'=\pi^\star \rho''$, and is invariant under $U_{j,0}$. As $\rho$ is $(G,X,S^1)$-cohomologous to $\rho'$ it is also $(G,X,S^1)$-cohomologous to $\pi^\star \tilde{\rho}$ which completes the proof.
\end{proof}
\begin{rem}
	To obtain a C.L.\ equation for $\rho'$ it is also possible to use a similar argument as in \cite[Lemma B.11]{Berg& tao & ziegler}.
\end{rem}
Theorem \ref{TDred:thm} follows by repeatedly applying the theorem above. Each time eliminating the connected component of the next group. After $k-1$ iterations we get a totally disconnected factor of $X$. In the next section we study these totally disconnected systems.
\section{Topological structure theorem and Weyl systems} \label{TSTweyl:sec}
We denote by $C_p$ the group of all $p$-th roots of unity in the circle, equipped with discrete topology. We prove the following counterpart of \cite[Lemma 4.7]{Berg& tao & ziegler} for the group $G=\bigoplus_{p\in P}\mathbb{F}_p$ in the special case where $X$ is a totally disconnected system. 
\begin{thm}  [The structure of totally disconnected systems] \label{TST:thm} Let $k\geq 2$ be an integer such that Theorem \ref{Main:thm} has already been proven for smaller values of $k$. Let $X$ be a totally disconnected system of order $<k$ and suppose that $X=Z_{<k-1}(X)\times_{\sigma} U$ for some totally disconnected group $U$. Then, there exist $d=O_k(1)$, integers $0\leq m_1,m_2,... \leq d$, and primes $p_1,p_2,...\in P$ such that $U\cong \prod_{n=1}^\infty C_{p_n^{m_n}}$ as topological groups. Moreover, if $p_i$ sufficiently large ($p_i>k$) we can take $m_i\leq 1$.
\end{thm}
Recall that if $X$ is totally disconnected, then any phase polynomial $P:X\rightarrow S^1$ takes finitely many values (see Proposition \ref{TDPV:prop}, and Theorem \ref{HTDPV:thm}). In this case we can also prove that these phase polynomials have phase polynomial roots (see Corollary \ref{roots}).  These results imply the following useful lemma.
\begin{lem} \label{valp:lem}
	Let $X$ be a totally disconnected ergodic $G$-system of order $<k$. Let $\rho:G\times X\rightarrow S^1$ be a cocycle which takes values in $C_n$ and suppose that $\rho$ is $(G,X,S^1)$-cohomologous to a phase polynomial of degree $<d$. Then $\rho$ is $(G,X,C_n)$-cohomologous to a phase polynomial of degree $<O_{d,k}(1)$.
\end{lem}
\begin{proof}
	Write 
	\begin{equation}  \label{valp:eq0}
	\rho=q\cdot \Delta F
	\end{equation}
	where $q\in P_{<d}(G,X,S^1)$ and $F:X\rightarrow S^1$ is a measurable map. Then, it follows that
	\begin{equation} \label{valp:eq1}
	1=\rho^n = q^n\cdot \Delta F^n.
	\end{equation}
	We conclude that $P:=F^n$ is a phase polynomial of degree $<d+1$.\\
	As $X$ is totally disconnected, Theorem \ref{TDPV:thm} implies that $P$ takes values in some closed subgroup $H$ of $S^1$ (after rotating $F$ by a constant if necessary). Therefore, we can write $H\cong C_{{p_1}^{l_1}}\times...\times C_{{p_m}^{l_m}}$ for $l_1,...,l_m=O_k(1)$ and distinct primes $p_1,..,p_m$.\\
	Let $\pi_i: H\rightarrow C_{p_i^{l_i}}$ be one of the coordinate maps and let $P_i : = \pi_i\circ P$. Clearly, $P_i$ is also a phase polynomial of degree $<d+1$ and we have that $P=\prod_{i=1}^m P_i$.
	Our goal is to find for each coordinate $1\leq i \leq m$, a phase polynomial $\psi_i$ of degree $<O_{d,k}(1)$ such that $\psi_i^n = P_i$. \\
	
	Fix $1\leq i \leq m$, then $P_i$ takes values in $C_{{p_i}^{l_i}}$. Write $n=p_i^{r_i}\cdot n'$ for some $r_i\in\mathbb{N}$ and some integer $n'\in\mathbb{N}$ that is co-prime to $p_i$. First, we find an $n'$-th root of $P_i$. To do this, let $\alpha\in\mathbb{N}$ be such that $n'\cdot \alpha = 1 \text{ mod } p_i^{l_i}$, and let $\phi_i = P_i^{\alpha}$. We conclude that $$\phi_i^{n'} = P_i^ {n'\alpha} = P_i.$$ Clearly, $\phi_i$ is also a phase polynomial of degree $<d+1$. It is therefore left to find a $p_i^{r_i}$-th root for $\phi_i$. We have two cases. If $p_i,r_i=O_{k,d}(1)$, then the claim follows from Corollary \ref{roots}. Otherwise, suppose that either $p_i$ or $r_i$ are sufficiently large. Let $G=G_{p_i}\oplus G'$ where $G_{p_i}$ is the subgroup of $G$ of elements of order $p_i$. In this case we claim that $q(g,x)^n=1$ for all $g\in G_{p_i}$. Indeed, if $r_i$ is sufficiently large, then Proposition \ref{PPC} implies $q(g,x)^n=1$ for all $g\in G_{p_i}$. Otherwise if $p_i$ is sufficiently large $(p_i>k)$, then by Theorem \ref{HTDPV:thm} we have that $q(g,x)$ takes values in $C_{p_i}$, hence $q(g,x)^n = 1$.\\
From equation (\ref{valp:eq1}) we see that in both cases $P$ is invariant under $g\in G_{p_i}$ and therefore so is $P_i$. Recall that $P_i$ takes values in $C_{p_i^{l_i}}$ and so from Proposition \ref{PPC} it is invariant under $g\in G$ whose order is coprime to $p$ as well. It follows that $P_i$ is $G$-invariant, hence by ergodicity it is a constant. Since $\phi_i$ is some power of $P_i$ it is also a constant. Hence $\phi_i$ has an $p_i^{r_i}$-th root.\\

We conclude that either way, there exists a phase polynomial $\psi_i$ of degree $<O_{d,k}(1)$ with $\psi_i^n = P_i$. Now we glue all coordinates together. Let $\psi:X\rightarrow H$ be the product of all coordinates $\psi(x)=\psi_1(x)\cdots\psi_m(x)$. As $\psi_1,...,\psi_m$ are phase polynomials of degree $<O_{k,d}(1)$, so is $\psi$. Since $\psi_i^n=P_i$, we conclude that $\psi^n=P$.\\

To finish the proof we now let $F'=F/\psi$ and $q'=q\cdot \Delta \psi$. From equation (\ref{valp:eq0}), we see that $$\rho = q' \cdot \Delta F'.$$
	As $\psi$ is an $n$-th root of $P=F^n$, we have that $F'$ takes values in $C_n$ and therefore so does $q'$. Moreover as $\psi$ is a phase polynomial of degree $<O_{k,d}(1)$ we have that so is $q'$ as required.
\end{proof} 
\begin{proof} [Proof of Theorem \ref{TST:thm}]
	Let $X=Z_{<k-1}(X)\times_\sigma U$ be as in the theorem. By Proposition \ref{Sylow}, we have that $U=\prod_p U_p$ where $U_p$ are the $p$-sylow subgroups of $U$.\\
	Fix any prime $p$ and let $\chi:U_p\rightarrow S^1$ be any continuous character of $U_p$. As $U_p$ is a $p$-group, the image of $\chi$ is a cyclic group $C_{p^n}$ for some $n\in\mathbb{N}$. Noting that $\chi\circ\sigma$ is a cocycle in $Z^1(G,Z_{<k-1}(X),S^1)$, by our induction use of Theorem \ref{Main:thm} we have that $\chi\circ\sigma$ is $(G,X,S^1)$-cohomologous to a phase polynomial of degree $<O_k(1)$. Therefore, by Lemma \ref{valp:lem} it is $(G,X,C_{p^n})$-cohomologous to a phase polynomial $q:G\times X\rightarrow C_{p^n}$ of degree $<O_k(1)$ (potentially higher bound than before). It follows by Lemma \ref{PPC} that $q$ is trivial for all $g$ of order co-prime to $p$. If $g$ is of order $p$, then by Proposition \ref{TDPV:prop} we have that $q^{p^d}(g,x)=1$ and if $p$ is sufficiently large (greater than $k$), then by Theorem \ref{HTDPV:thm} we can take $d=1$. By the cocycle identity in $g$, it follows that $q^{p^d}(g,x)=1$ for all $g\in G$. Since $\sigma$ is cohomologous to $q$, we conclude that $\sigma^{p^d}$ is a coboundary.\\
	The system $Z_{<k}(X)\times_{\sigma^{p^d}} C_{p^{n-d}}$ is a factor of $Z_{<k-1}(X)\times_\sigma U$ and therefore a factor of $X$. Since $X$ is ergodic so is every factor, which means that $n=d$. Otherwise, if $n> d$ then write $\sigma^{p^d}=\Delta F$ for some measurable map $F:Z_{<k}(X)\rightarrow S^1$ and observe that $(s,u)\mapsto \overline{F}(s)u$ is a non-constant invariant function on $Z_{<k}(X)\times_{\sigma^{p^d}}C_{p^{n-d}}$.\\
	We conclude that $U_p$ is a $p^d$-torsion subgroup for some $d=O_k(1)$. Theorem \ref{torsion} implies that $U_p$ is a direct product of copies of $C_{p^r}$ for $r\leq d$. As required.
	\end{proof}
We can finally prove that all totally disconnected systems of order $<k$ are Weyl, assuming that Theorem \ref{Main:thm} holds for all smaller values of $k$.
\begin{thm} [Totally disconnected systems are Weyl] \label{TDisweyl} Let $k\geq 1$ be an integer such that Theorem \ref{Main:thm} has already been proven for all smaller values of $k$. Let $X$ be a totally disconnected system of order $<k$. Then $X$ is isomorphic to a Weyl system.
\end{thm}
\begin{proof} We prove the claim by induction on $k$. If $k=1$ then $X$ is a point and the claim is trivial. Let $k\geq 2$ and assume inductively that the theorem holds for systems of order $<k-1$. Write $X=Z_{<k-1}(X)\times_\sigma U$. By the induction hypothesis we can assume that $Z_{<k-1}(X)$ is a Weyl system. Moreover, from Theorem \ref{TST:thm} we have that $U\cong \prod_{i=1}^\infty C_{p_i^{m_i}}$ where $m_i=O_k(1)$ and $m_i=1$ if $p_i$ is sufficiently large. Let $\tau_i:U\rightarrow C_{p_i^{m_i}}$ be one of the coordinate maps. We think of $C_{p_i^{m_i}}$ as a subgroup of $S^1$. By Theorem \ref{Main:thm} applied for the factor $Z_{<k-1}(X)$, we have that $\tau_i\circ\sigma$ is $(G,Z_{<k-1}(X),S^1)$-cohomologous to a phase polynomial of degree $<O_k(1)$ (into $S^1$). Finally, from Lemma \ref{valp:lem} we have that $\tau_i\circ\sigma$ is $(G,Z_{<k-1}(X),C_{p_i^{m_i}})$-cohomologous to a phase polynomial $q_i:G\times X\rightarrow C_{p_i^{m_i}}$ of possibly higher but bounded ($<O_k(1)$) degree. Since this is true for every coordinate map, we conclude that $\sigma$ is $(G,Z_{<k-1}(X),U)$-cohomologous to $q:G\times Z_{<k-1}(X)\rightarrow U$ where $q(x):=(q_1(x),q_2(x),...)$. Since all of the phase polynomials $q_i$ are of bounded degree which only depends on $k$ (and is independent of $i$), it is easy to see that $q$ is of bounded degree. Since cohomologous cocycles defines isomorphic systems (see Remark \ref{coh:rem}), it follows that $X=Z_{<k-1}(X)\times_\sigma U$ is isomorphic to $Z_{<k-1}(X)\times_q U$. Since $q$ is a phase polynomial and $Z_{k-1}(U)$ is a Weyl system this completes the proof. 	
\end{proof}

\section{Proof of Theorem \ref{MainT:thm}} \label{low:sec}
The proof of Theorem \ref{MainT:thm} follows by similar methods as in \cite{Berg& tao & ziegler}. However, the multiplicity of generators of different prime orders in $G$ leads to some new difficulties in the "finite group case". These can be solved by working out each prime separately.
\subsection{Reduction of Theorem \ref{MainT:thm} to solving a Conze-Lesigne type equation on a totally disconnected system}

Arguing as in the proof of Theorem \ref{Main:thm} in section \ref{proof1} (see in particular equation (\ref{multiCL})), we can reduce matters into solving a Conze-Lesigne type equation. It is thus left to show the following result.
\begin{thm}  [Conze-Lesigne type equation for functions on totally disconnected system]\label{TDCL:thm}
	Let $k\geq 1$ be such that Theorem \ref{Main:thm} has already been proven for smaller values of $k$ and let $X=Z_{<k-1}(X)\times _{\rho} U$ be a totally disconnected ergodic $G$-system of order $<k$. Let $f:G\times X\rightarrow S^1$ be a C.L.\ function of type $<m$ for some $m\in\mathbb{N}$ as in Definition \ref{Conze-Lesigne}. Then $f$ is $(G,X,S^1)$-cohomologous to $P\cdot \pi^\star \tilde{f}$ for some $P\in P_{<O_{k,m}(1)}(G,X,S^1)$ and a measurable $\tilde{f}:Z_{<k-1}(X)\rightarrow S^1$ where $\pi:X\rightarrow Z_{<k-1}(X)$ is the factor map.
\end{thm}
\begin{rem}
	Note that we do not require that the function $f:G\times X\rightarrow S^1$ is a cocycle. This theorem is a counterpart of \cite[Theorem 4.5]{Berg& tao & ziegler}.
\end{rem}
\begin{proof}[Proof of Theorem \ref{Main:thm} assuming Theorem \ref{TDCL:thm}]
We prove the Theorem by induction on $k$. If $k=1$ then the system $X$ is a point and the claim follows. Let $k\geq 1$ and suppose that Theorem \ref{Main:thm} has already been proven for all smaller values of $k$ and let $\rho:G\times X\rightarrow S^1$ be a cocycle of type $<m$. Let $t_1,...,t_m$ be any automorphisms of $X$. We claim by downward induction on $0\leq j \leq m$ that equation (\ref{multiCL}) holds. The case $j=m$ follows by iterating Lemma \ref{dif:lem}. Fix $j<m$ and assume inductively that the equation holds for $j+1$. Let $\rho_{t_1,...,t_j}  = \Delta_{t_1}...\Delta_{t_j}\rho$. By the induction hypothesis on $j$, $\rho_{t_1,...,t_j}$ is a C.L.\ cocycle. Therefore, by Theorem \ref{TDred:thm}, we can reduce matters to the case where $X$ is totally disconnected. Now by Theorem \ref{TDCL:thm}, we see that there exists a phase polynomial $P$ of degree $<O_{k,m,j}(1)$ such that $P\cdot \rho_{t_1,...,t_j}$ is measurable with respect to $Z_{<k-1}(X)$. Since $X$ is totally disconnected, we can apply the induction hypothesis on $k$. We conclude that $P\cdot \rho_{t_1,...,t_j}$ (and $\rho_{t_1,...,t_j}$) are cohomologous to a phase polynomial of degree $<O_{k,m,j}(1)$, as required. The case $j=0$ in equation (\ref{multiCL}) implies Theorem \ref{Main:thm}.
\end{proof}
\subsection{Reduction to a finite $U$}
Now we turn to the proof of Theorem \ref{TDCL:thm} assuming the induction hypothesis of Theorem \ref{Main:thm}.\\
Just like in \cite[Proposition 6.1]{Berg& tao & ziegler} we first show that it suffices to prove the theorem in the case where the group $U$ is a finite group.\\

We recall the following results from \cite[Lemma 5.1 and Lemma B.6]{Berg& tao & ziegler}.
\begin{lem}  [Descent of type]\label{dec:lem} Let $Y$ be a $G$-system, let $k,m\geq 1$, and let $X=Y\times_{\rho} U$ be an ergodic extension of $Y$ by a phase polynomial cocycle $\rho:G\times Y\rightarrow U$ of degree $<m$. Let $\pi:X\rightarrow Y$ be a factor map and let $f:Y\rightarrow S^1$ be a function such that $\pi^\star f$ is of type $<k$. Then $f$ is of type $<k+m+1$.
\end{lem}

\begin{lem} [Polynomial integration lemma] \label{PIL:lem}
	Let $m,k\geq 1$, let $X=Y\times_{\rho} U$ be an ergodic abelian extension of a $G$-system $Y$ by a cocycle $\rho:G\times Y\rightarrow U$ that is also a phase polynomial of degree $<k$. For all $u\in U$ let $q_u: X\rightarrow S^1$ be a phase polynomial of degree $<m$ which obeys the cocycle identity $q_{uv}=q_u V_u q_v$ for all $u,v\in U$. Then there exists a phase polynomial $Q:X\rightarrow S^1$ of degree $<O_{k,m}(1)$ such that $\Delta_u Q = q_u$ for all $u\in U$.
\end{lem}
The following Proposition is a corollary of Lemma \ref{PIL:lem}.
\begin{prop} [Descent of the C.L.\ equation] \label{inv:prop} 
	Let $X,U,\rho$ be as in Lemma \ref{PIL:lem}. Let $m\geq 0$ and $f:G\times X\rightarrow S^1$ be a $U$-invariant function. Suppose that $f=p\cdot \Delta F$ for some phase polynomial $p\in P_{<m}(G,X,S^1)$ and a measurable map $F:X\rightarrow S^1$. Then there exists a phase polynomial $p':G\times X\rightarrow S^1$ of degree $<O_{k,m}(1)$ and a measurable map $F':X\rightarrow S^1$ such that $p'$ and $F'$ are invariant under the action of $U$ and $f=p'\cdot \Delta F'$. 
\end{prop}
\begin{proof}
We take the derivative of both sides of the equation $f=p\cdot \Delta F$  by some $u\in U$. Since the action of $U$ commutes with the action of $G$ we have, $$\Delta_u p \cdot \Delta \Delta_u F=1.$$
	It follows that $\Delta_u F$ is a phase polynomial of degree $<m+1$. Moreover, it satisfies the cocycle identity $\Delta_{uv} F = \Delta_u F V_u \Delta_v F$. Therefore, applying Lemma \ref{PIL:lem}, we conclude that there exists a phase polynomial $Q:X\rightarrow S^1$ of degree $<O_{k,m}(1)$ with the property that $\Delta_u F = \Delta_u Q$. Let $p'=p\cdot \Delta Q$ and $F'=F/Q$. We have,
	$$f= p'\cdot \Delta F'.$$
	As $F'$ is invariant under $U$ and $f$ is invariant under $U$, we conclude that $p'$ is invariant under $U$ and this completes the proof.
\end{proof}
We need the following basic fact about the product topology.\\
\textbf{Fact:} Let $U\cong \prod_{n=1}^{\infty} C_{p_n^{m_n}}$ be as in Theorem \ref{TST:thm}. The cylinder neighborhoods of the identity (i.e. the co-finite sub-products of these cyclic groups) form a basis for the topology of $U$ around the identity. That is, any open neighborhood of the identity in $U$ contains a cylinder set.\\

Below we prove the following counterpart of Proposition 6.1 from \cite{Berg& tao & ziegler}. We fix an integer $k$ and assume (inductively) that Theorem \ref{Main:thm} holds for smaller values of $k$.
\begin{thm} \label{MainF:thm}
	In order to prove Theorem \ref{TDCL:thm} it suffices to do so in the case where $U$ is finite.
\end{thm}
\begin{proof}
	Let $X$ be a totally disconnected system of order $<k$ and write $X=Z_{<k-1}(X)\times_{\sigma} U$. By Theorem \ref{TDisweyl} we can assume that $\sigma:G\times Z_{<k-1}(X)\rightarrow U$ is a phase polynomial of degree $<O_{k}(1)$. Let $f:G\times X\rightarrow S^1$ be a function of type $<m$ and assume that for every $u\in U$ we have that $\Delta_u f=p_u\cdot \Delta F_u$ for some $p_u\in P_{<m}(G,X,S^1)$ and $F_u\in\mathcal{M}(X,S^1)$. From the linearization Lemma (Lemma \ref{lin:lem}) we know that there exists an open neighborhood $U'$ of the identity in $U$ such that $u\mapsto p_u$ is a cocycle on $U'$ (i.e. $p_{uv}=p_u V_u p_v$ whenever $u,v\in U'$). By Theorem \ref{TST:thm}, $U'$ contains a cylinder neighborhood. Therefore, by shrinking $U'$ we can write $U=U'\times W$ for some finite group $W$.\\
	
	We pass from $U$ to $U'$. First we write $X=Y\times_{\sigma'} U'$ where $Y=Z_{<k-1}(X)\times_{\sigma''} W$ and $\sigma'$ and $\sigma''$ are the projections of $\sigma$ to $U'$ and $W$ respectively. By construction of $U'$ we have that $p_{uv}=p_u V_u p_v$, for all $u,v\in U'$. We integrate $p_u(g,\cdot)$ by applying Lemma \ref{PIL:lem} once for every $g\in G$. We deduce that there exists a phase polynomial $Q:G\times X\rightarrow S^1$ of degree $<O_k(1)$ such that $\Delta_u Q = p_u$ for every $u\in U'$ (note that $Q$ may not be a cocycle in $g$). In particular, for every $u\in U'$ we have that $\Delta_u (f/Q)$ is a $(G,X,S^1)$-coboundary. Therefore, by Lemma \ref{cob:lem} we see that $f/Q$ is $(G,X,S^1)$-cohomologous to a function $f':G\times X\rightarrow S^1$ that is invariant under the action of some open subgroup $U''$ of $U'$.\\
	Let $\varphi:U\rightarrow U/U''$ be the quotient map, let $X'=Z_{<k-1}(X)\times_{\varphi\circ\sigma} U/U''$ and let  $\pi:X\rightarrow X'$ be the factor map. We can then write  $f'=\pi^\star\tilde{f}$ where $\tilde{f}:G\times X'\rightarrow S^1$.\\
	We claim that $\tilde{f}$ is of bounded type. To see this recall that $f$ is of type $<m$ and $Q$ is a phase polynomial of degree $<O_{k,m}(1)$. By Lemma \ref{PP} (iii) $Q$  and $f/Q$ are of type $<O_{k,m}(1)$. Since $\pi^\star \tilde{f}$ is $(G,X,S^1)$-cohomologous to $f/Q$ it is also of type $<O_{k,m}(1)$. Therefore, by Lemma \ref{dec:lem}, $\tilde{f}$ is of type $<O_{k,m}(1)$.\\
	Since $Q$ is a phase polynomial of degree $<O_{k,m}(1)$ we have that $\Delta_u (f/Q)$ is $(G,X,S^1)$-cohomologous to a phase polynomial. In particular, $\pi^\star \Delta_{\varphi(u)} \tilde{f}=\Delta_u \pi^\star \tilde{f}$ is $(G,X,S^1)$-cohomologous to a phase polynomial of degree $<O_{k,m}(1)$. Thus for every $u\in U/U''$ we can write \begin{equation}\label{up}\pi^\star \Delta_u\tilde{f}=q_u\cdot \Delta F_u
	\end{equation}
	for some phase polynomial $q_u$ of degree $<O_{k,m}(1)$. Applying Proposition \ref{inv:prop} we can assume that $q_u$ and $F_u$ are invariant under $U''$. Therefore all of the functions in equation (\ref{up}) are invariant with respect to $U''$ and so everything can be pushed to $X'$. We conclude that $\Delta_u\tilde{f}$ is $(G,X',S^1)$-cohomologous to a phase polynomial of degree $<O_{k,m}(1)$. Now, applying Theorem \ref{TDCL:thm} for the system $X'$, which is an extension of $Z_{<k-1}(X')$ by a finite group $U/U''$, we conclude that $\tilde{f}$ is $(G,X',S^1)$- cohomologous to $P\cdot \tilde{\pi}^\star f_0$ where $P\in P_{<O_{k,m}(1)}(G,X',S^1)$, $f_0:Z_{<k-1}(X)\rightarrow S^1$ is a measurable map and $\tilde{\pi}:X\rightarrow Z_{<k-1}(X)$ is the factor map. As $P$ is a phase polynomial of degree $<O_{k,m}(1)$ and $\tilde{f}$ is of type $<O_{k,m}(1)$, arguing as before we have that $f_0$ is also of type $<O_{k,m}(1)$.
	Finally, applying Theorem \ref{Main:thm} for the system $Z_{<k-1}(X)$, we have that $f_0$ is $(G,Z_{<k-1}(X),S^1)$-cohomologous to a phase polynomial of degree $<O_{k,m}(1)$. Lifting everything up, we conclude that $f$ is $(G,X,S^1)$-cohomologous to a phase polynomial of degree $<O_{k,m}(1)$.
\end{proof}
\subsection{Proving the theorem for a finite $U$}
We prove the following counterpart of Proposition 7.1 from \cite{Berg& tao & ziegler} for totally disconnected systems.
\begin{thm} [Theorem \ref{TDCL:thm} for a finite $U$]\label{FC:thm}
	Let $k\geq 1$ and $U$ be a finite group. Suppose that $X=Z_{<k-1}(X)\times _{\rho} U$ is a totally disconnected ergodic $G$-system of order $<k$ and let $f:G\times X\rightarrow S^1$ be a C.L.\ function of type $<m$. Then $f$ is $(G,X,S^1)$-cohomologous to $P\cdot\pi^\star \tilde{f}$ for some $P\in P_{<O_{k,m}(1)}(G,X,S^1)$ and a measurable $\tilde{f}:Z_{<k-1}(X)\rightarrow S^1$ where $\pi:X\rightarrow Z_{<k-1}(X)$ is the factor map.
\end{thm}
The proof follows the general lines of the proof of Proposition 7.1 in \cite{Berg& tao & ziegler}. The main technical difference is that the results in Appendix \ref{roots:app} (the counterpart of Appendix D in \cite{Berg& tao & ziegler}) are only valid for polynomials which takes values in a finite cyclic group $C_{p^n}$ where $p$ is a prime and $n\in\mathbb{N}$. Recall that if $X$ is a totally disconnected system then by Proposition \ref{TDPV:prop} any phase polynomial $p:X\rightarrow S^1$ (after constant multiplication) takes values in a finite subgroup $H$ of $S^1$. These subgroups takes the form $H\cong \prod_{i=1}^N C_{q_i^{l_i}}$ for some primes $q_i$ and $l_i\in\mathbb{N}$ (and the $l_i$'s are bounded by a constant which only depends on the degree of the polynomials and the order of $X$). Therefore in order to apply the results from Appendix \ref{roots:app}, we have to study each coordinate of $p$ with respect to $H$ separately. The formal proof is given below.
\begin{proof}
	By Theorem \ref{TST:thm} we have that $U$ is isomorphic to $\prod_{i=1}^N C_{{p_i}^{n_i}}$ for some $n_1,n_2,...,n_N = O_k(1)$ and $N$ unbounded but finite. Moreover, if $p_i$ is sufficiently large with respect to $k$ we can take $n_i=1$. \\
	Let $e_1,...,e_N$ be the standard basis. By the assumption of $f$ being a C.L.\ function, for each $1\leq i \leq N$ we can write 
	\begin{equation}\label{fe1}
	\Delta_{e_i} f = Q_i \cdot \Delta F_i
	\end{equation}
	for a phase polynomial $Q_i\in P_{<m}(G,X,S^1)$ and $F_i:X\rightarrow S^1$ a measurable map.\\
		The idea here is to find a measurable map $F$ such that $F_i = \Delta_{e_i} F$ times a polynomial error. This will imply that $\Delta_{e_i}f/\Delta F$ is a phase polynomial. The cocycle identity then implies that $\Delta_u (f/\Delta F)$ is also a phase polynomial (and a cocycle in $u$). At this point the claim in the theorem follows by the polynomial integration lemma.\\
	The group $U$ is isomorphic to the free group with $N$ generators $e_1,...,e_N$ modulo the relations of the form $e_j^{p_j^{n_j}}=1$ for every $1\leq j \leq N$ and $[e_i,e_j]=1$ for every $1\leq i,j\leq N$. This will help us to construct the function $F$.\\
	First observe that we have the following telescoping identity.
	$$\prod_{t=0}^{p_j^{n_j}-1} V_{e_j}^t\Delta_{e_j} f = 1$$
	and so equation (\ref{fe1}) implies that 
	\begin{equation} \label{Q}
	\Delta \prod_{t=0}^{p_j^{n_j}-1} V_{e_j}^tF_j = \prod_{t=0}^{p_j^{n_j}-1}V_{e_j}^t \overline{Q_j}\in P_{<m}(G,X,S^1).
	\end{equation}
In particular we have that $$\prod_{t=0}^{p_j^{n_j}-1} V_{e_j}^tF_j\in P_{<m+1}(X,S^1).$$
Observe that a priori, if $F_j=\Delta_{e_j} F\cdot P_j$ for some function $F$ and a phase polynomial $P_j$ of degree $<O_{k,m}(1)$, then the term above can be written as $\prod_{t=0}^{p_j^{n_j}-1}V_{e_j}^t P_j$. In particular, it means that we must be able to to replace $F_j$ with a new function $\tilde{F}_j=F_j/ P_j$ such that $F_j/\tilde{F}_j$ is a phase polynomial of degree $<O_{k,m}(1)$ and $$\Delta \prod_{t=0}^{p_j^{n_j}-1} V_{e_j}^t\tilde{F}_j = 1.$$
To do this, we claim that there exists a phase polynomial $\psi_j$ of degree $<O_{k,m}(1)$ which is invariant to the translation $V_{e_j}$ and is a ${p_j}^{n_j}$-th root of $\prod_{t=0}^{p_j^{n_j}-1} V_{e_j}^tF_j$. Then we set $P_j=\psi_j$. Observe that $\prod_{t=0}^{p_j^{n_j}-1} V_{e_j}^t F_j$ is invariant to the translation by $e_j$. We study two cases. First, if $p_j=O_{k,m}(1)$, then we view $\prod_{t=0}^{p_j^{n_j}-1} V_{e_j}^t F_j$ as a phase polynomial on the factor induced by quotienting out $\left<e_j\right>$. At this point we apply Corollary \ref{roots} in order to find a phase polynomial of degree $<O_{k,m}(1)$ that is also a $p_j^{n_j}$-th root of $\prod_{t=0}^{p_j^{n_j}-1} V_{e_j}^t F_j$. Lifting everything up we conclude that there exists a phase polynomial $\psi_j$ of degree $<O_{k,m}(1)$ that is invariant under $e_j$ and is a $p_j^{n_j}$-th root of $\prod_{t=0}^{p_j^{n_j}-1} V_{e_j}^t F_j$, as required.\\
	Otherwise, we have that $p_j$ is sufficiently large. In this case Theorem \ref{TST:thm} implies that $n_j=1$. Since $X$ is totally disconnected, Proposition \ref{TDPV:prop} implies that up to a constant multiple, $\prod_{t=0}^{p_j-1} V_{e_j}^t F_j$ takes values in a finite subgroup $H$ of $S^1$. Rotating $F_j$ by a $p_j$-th root of this constant, we may assume that $\prod_{t=0}^{p_j-1} V_{e_j}^t F_j$ takes values in $H$ without changing equation (\ref{fe1}).
	
Recall that a finite subgroup $H\leq S^1$ takes the form $H\cong \prod_{i=1}^{N'}C_{q_i^{l_i}}$ where $q_i$ are primes. The idea now is to find a root for each of the coordinates of $\prod_{t=0}^{p_j-1} V_{e_j}^t F_j$ (similar to the proof of Lemma \ref{valp:lem}).\\

Let $\pi_i:H\rightarrow C_{q_i^{l_i}}$ be one of the coordinate maps of $H$, we study the term $\pi_i\circ\prod_{t=0}^{p_j-1} V_{e_j}^t F_j$. We have two cases, if $q_i\not = p_j$, then $p_j$ is invertible modulo $q_i^{l_i}$ (i.e. there exists $a\in\mathbb{N}$ such that $ap_j = 1 \mod q_i^{l_i}$). In this case we conclude that some power of $\prod_{t=0}^{p_j-1} V_{e_j}^t \pi_i\circ F_j$ is a $p_j$-th root that is also a phase polynomial of degree $<m+1$. This power is clearly invariant under $V_{e_j}$ and the claim follows. We denote this phase polynomial root by  $\psi_{i,j}$. Otherwise we have that $q_i=p_j$. Since $p_j$ is large so is $q_i$ and so by Theorem \ref{HTDPV:thm} we have that $l_i=1$. We conclude, by equation (\ref{Q}) that the derivative of $\pi_i\circ\prod_{t=0}^{p_j-1}V_{e_j}^t F_j$ by $\Delta_g$ satisfies
	\begin{equation} \label{fe1.1}
	\pi_i\circ\left(\prod_{t=0}^{p_j-1} V_{e_j}^t Q_j\right) = \pi_i\circ\left(\prod_{t=0}^{p_j-1} \left( \Delta_{e_j} ^t  Q_j\right)^{\binom{p_j}{t+1}}\right).
	\end{equation}
	
	Write $G=G_{p_j} \oplus G'$ where $G_{p_j}$ is the subgroup of elements of order $p_j$ and $G'$ be its complement. The terms in equation (\ref{fe1.1}) are of order $q_i=p_j$. In particular, this means that for $g\in G'$ both of them are trivial (by Proposition \ref{PPC}). For $g\in G_{p_j}$ we claim that the right hand side is trivial. 
	For such $g\in G_{p_j}$ we have by Proposition \ref{PPC} and Theorem \ref{HTDPV:thm} that $Q_j(g,\cdot)$ takes values in $C_{p_j}$. Observe that $p_j$ divides $\binom{p_j}{t+1}$ for every $0\leq t< p_j-1$, and since $p_j$ is sufficiently large  $\Delta_{e_j}^{p_j-1}$ eliminates $Q_j$ (by Lemma \ref{vdif:lem}), hence for $g\in G_{p_j}$ we have, $\prod_{t=0}^{p_j-1} \left( \Delta_{e_j} ^t  Q_j(g,x)\right)^{\binom{p_j}{t+1}}=1$.  Since $Q_j$ is a cocycle in $g$ and every $g\in G$ can be written as $g_{p_j}+g'$ where $g_{p_j}\in G_{p_j}$ and $g'\in G'$, we conclude from the above that the terms in equation (\ref{fe1.1}) are trivial for every $g\in G$. Therefore by ergodicity $\pi_i\circ\prod_{t=0}^{p_j-1} V_{e_j}^t F_j$ is a constant and we can find a $p_j$-th root which we denote by $\psi_{i,j}$.\\
	
	Thus, in any case every coordinate of $\prod_{t=0}^{p_j^{n_j}-1} V_{e_j}^t\circ F_j$ has a root that is also a polynomial of degree $<m+1$ and is invariant with respect to the translation $V_{e_j}$. Gluing all the coordinates $\psi_{i,j}$ together, we see that there exists a phase polynomial $\psi_j(x) = \prod_{i=1}^{N'} \psi_{i,j}$ of degree $<m+1$ such that $\psi_j^{p_j^{n_j}}=\prod_{t=1}^{p_j^{n_j}-1} V_{e_j}^t F_j$.
	Now set $\tilde{F}_j  := F_j/\psi_j$ and $\tilde{Q}_j := Q_j\cdot \Delta \psi_j$ then $\tilde{Q}_j$ is a phase polynomial of degree $<O_{k,m}(1)$ and we have
	
	\begin{equation}\label{fe2}
	\Delta_{e_j} f = \tilde{Q}_j\cdot \Delta \tilde{F}_j
	\end{equation}
	and
	\begin{equation} \label{fe3}
	\prod_{t=1}^{p_j^{n_j}-1} V_{e_j}^t \tilde{F}_j=1
	\end{equation}
	Now that we have already dealt with the torsion relations, we proceed by defining the desired $F$ and then we will work out the relations which comes from the commutators $[e_i,e_j]=1$.\\ Write $[t_1,...,t_N]=e_1^{t_1}\cdot...\cdot e_N^{t_N}$ and $X=Z_{<k-1}(X)\times_{\sigma} U$ and let $F:X\rightarrow S^1$ be the function 
	$$F(y,[t_1,...,t_N]) = \prod_{i=1}^N\prod_{0\leq t_i'<t_i} \tilde{F}_i(y,[t_1,...,t_{i-1},t_i',0,...,0])$$
	with the convention that $\prod_{0\leq t_i' < t_i} a_{t_i'}=\left(\prod_{t_i<t_i'\leq 0} a_{t_i'}\right)^{-1}$. Equation (\ref{fe2}) implies that $F$ is well defined.\\
	We compute the derivatives of $F$. We have,
	$$\Delta_{e_j}F(y,[t_1,...,t_N])=\prod_{i=1}^N \prod_{0\leq t_i'<t_i} \Delta_{e_j} F_j(y,[t_1,...,t_{i-1},t_i',0,...,0]).$$
	For any $1\leq j\leq N$. On the other hand, we have that telescoping identity
	\[
	\prod_{j=1}^N \prod_{0\leq t_j'<t_j} \Delta_{e_j} F_i(y,[t_1,...,t_{j-1},t_j',0,...,0]) = \frac{F_i(y,[t_1,...,t_N])}{F_i(y,1)}
	\]
	and thus,
	\begin{equation}\label{fe4}
	\Delta_{e_j}F(y,[t_1,...,t_N])=\frac{\tilde{F}_j(y,[t_1,...,t_N])}{\tilde{F}_j(y,1)}\prod_{i=1}^N \prod_{0\leq t_i'<t_i} \omega_{i,j}(y,[t_1,...,t_{i-1},t_i',0,...,0])
	\end{equation}
	where $\omega_{i,j}=\frac{\Delta_{e_i}\tilde{F}_j}{\Delta_{e_j}\tilde{F}_i}$. Since $[e_j,e_i]=1$, we see that $e_i$ and $e_j$ commute and so by (\ref{fe2}) we have that $\omega_{i,j}$ are phase polynomials of degree $<O_{k,m}(1)$.\\
	We study these polynomials. Recall that $\sigma$ is a phase polynomial of degree $<O_k(1)$. This means that the map $(y,u)\mapsto u$ is also a phase polynomial (the derivative is $\sigma$). Since $[t_1,...,t_n]\mapsto[t_1,...,t_{i-1},t_i',0,...,0]$ is a constant multiple of a homomorphism, we have that $(y,[t_1,...,t_n])\rightarrow [t_1,...,t_{i-1},t_i',0,...,0]$ is also a phase polynomial of degree $<O_k(1)$. From Lemma \ref{B.5}, we conclude that $(y,[t_1,...,t_n])\mapsto \omega_{i,j}(y,[t_1,...,t_{i-1},t_i',0,...,0])$ are all phase polynomials of degree $<O_{k,m}(1)$.\\
	We claim that the map $$\xi_j(y,[t_1,...,t_N]) =\prod_{i=1}^N\prod_{0\leq t_i'<t_i}\omega_{i,j}(y,[t_1,...,t_{i-1},t_i',0,...,0])$$ is a phase polynomial of degree $<O_{k,m}(1)$. Clearly $\xi_j$ is a function of phase polynomials. From Theorem \ref{TDPV:thm}, it follows that there exists a finite subgroup $H\leq S^1$ such that all $\omega_{i,j}(y,[t_1,...,t_{i-1},t_i',...])$ take values in $H$ (up to a constant multiple). Therefore so does $\xi_j$. We show that every coordinate of $\xi_j$ is a phase polynomial. Let $\pi_n:H\rightarrow C_{q_n^{l_n}}$ be one of the coordinate maps. We study two cases. If $q_n=O_{k,m}(1)$ then Corollary \ref{funofpoly} implies that $\pi_n\circ\xi_j$ is a phase polynomial of degree $<O_{k,m}(1)$. Otherwise, we have that $q_n$ is sufficiently large. In this case $\pi_n\circ \omega_{i,j}(y,[t_1,...,t_{i-1},t_i',0,...,0])$ is a phase polynomial in $t_i'$ that, by Theorem \ref{HTDPV:thm}, takes values in $C_{q_n}$. Note that since there is an unbounded number of $\omega_{i,j}$'s in the definition of $\xi_j$, we can not yet deduce that $\xi_j$ is a phase polynomial of bounded degree. Instead, we consider the Taylor expansion of $\omega_{i,j}$. We have,
	$$ \omega_{i,j}(y,[t_1,...,t_{i-1},t_i',0,...,0])=\prod_{0\leq j \leq O_{k,m}(1)} \left[ \Delta_{e_i}^j \omega_{i,j}(y,[t_1,...,t_{i-1},t_i',0,...,0])\right]^{\binom{t_i'}{j}}$$
	and thus
	$$\prod_{0\leq t_i'\leq t_i} \omega_{i,j}(y,[t_1,...,t_{i-1},t_i',0,...,0])=\prod_{0\leq j \leq O_{k,m}(1)} \left[ \Delta_{e_i}^j \omega_{i,j}(y,[t_1,...,t_{i-1},t_i',0,...,0])\right]^{\binom{t_i}{j+1}}.$$
	Since the product is over a bounded number of $\omega_{i,j}$'s, we see by Lemma \ref{B.5} that $\pi_n\circ\xi_j$ is a phase polynomial of degree $<O_{k,m}(1)$. Since the degree is independent on the coordinate, we conclude that $\xi_j\in P_{<O_{k,m}(1)}(X,S^1)$ as required.\\
	From (\ref{fe4}) we finally have that
	$$\Delta_{e_j} F (y,u) \in  \frac{\tilde{F}_j(y,u)}{\tilde{F}_j(y,1)}\cdot P_{<O_{k,m}(1)}(X,S^1).$$
	Now let $f'=f/\Delta F$. Then,
	\begin{equation}\label{eqp}
	\Delta_{e_j}f' \in \left(\pi^\star F_j'\right)\cdot P_{<O_{k,m}(1)}(G,X,S^1) 
	\end{equation}
	where $\pi^\star F_j' (y,u) = \tilde{F}_j (y,1)$ (i.e. $\Delta_{e_j}f'$ is a multiplication of $\pi^\star F_j'$ with a phase polynomial of some bounded degree).\\
	We repeat the same argument as above now with $f'$ instead. Note that since $\pi^\star F_j'$ is invariant with respect to the action of $U$, we do not need to work out the commutator relations again. Namely, for every $i,j$ we have that $\Delta_{e_i} \pi^\star F_j' = \Delta_{e_j} \pi^\star F_i'=1$. We deal with the torsion relations. As before, we have the telescoping identity
	$$\prod_{t=0}^{p_j^{n_j}-1}V_{e_j}^tf'=1$$
	Which implies that
	$$\pi^\star \Delta (F_j')^{p_j^{n_j}}\in P_{<O_{k,m}(1)}(G,X,S^1).$$ Therefore, $(F_j')^{p_j^{n_j}}$ is a phase polynomial of degree $<O_{k,m}(1)$ on $Z_{<k-1}(X)$.\\
	We apply the same argument as before. We see that there exists a $p_j^{n_j}$-th root $P_j$ for the phase polynomial $(F_j')^{p_j^{n_j}}$. Let $F_j'' = F_j'/P_j$. Then, $F_j''$ takes values in $C_{p_j^{n_j}}$ and from (\ref{eqp}) we have $$(\Delta_{e_j} f') \in \pi^\star F_j'' \cdot P_{<O_{k,m}(1)}(G,X,S^1).$$
	We define $F^\star :X\rightarrow S^1$ by $F^\star(y,[t_1,...,t_N]):=\prod_{j=1}^N F_j''(y)^{t_j}$. This map is well defined since $F_j''$ takes values in $C_{{p_j}^{n_j}}$. We observe that $\pi^\star \Delta F_j'' =\Delta_{e_j}\Delta F^\star$ and let $f''=f'/\Delta F^\star$. We have that $f''$ is cohomologous to $f$ and $\Delta_{e_j}f''\in P_{<O_{k,m}(1)}(G,X,S^1)$ for all $1\leq j \leq N$. The cocycle identity implies that $\Delta_u f''\in P_{<O_{k,m}(1)}(G,X,S^1)$ for every $u\in U$.  Integrating this term by Lemma \ref{PIL:lem} once for every $g\in G$, we have that $\Delta_u f'' = \Delta_u P$ for some $P\in P_{<O_{k,m}(1)}(G,X,S^1)$. It follows that $f''/P$ is invariant under $U$ and so $f''=P\cdot \pi^\star \tilde{f}$ where $\tilde{f}:G\times Z_{<k-1}(X)\rightarrow S^1$.
\end{proof}

\section{The High Characteristic Case} \label{high:sec}
Throughout this section, we denote $\text{char} (G) = \min\{p:p\in\mathcal{P}\}$. We prove the following version of Theorem \ref{Main:thm}.
\begin{thm}\label{reductionH:thm}
	Let $1\leq k,j \leq \text{char}(G)$ and $X$ be an ergodic $G$-system of order $<j$ which splits. Let $\rho:G\times X\rightarrow S^1$ be a cocycle of type $<k$, then $\rho$ is $(G,X,S^1)$-cohomologous to a phase polynomial of degree $<k$.
\end{thm}
As with the original version, we want to prove this theorem in two steps. First we reduce matters to a totally disconnected case and then prove the theorem in that case.\\
We begin by introducing some definitions and an important lemma from \cite{Berg& tao & ziegler}.
\begin{defn} [Quasi-cocycles]
	Let $X$ be an ergodic $G$-system, let $f:G\times X\rightarrow S^1$ be a function. We say that $f$ is a quasi-cocycle of order $<k$ if for every $g,g'\in G$ one has $$f(g+g',x)=f(g,x)\cdot f(g',T_gx)\cdot p_{g,g'}(x)$$ for some $p_{g,g'}\in P_{<k}(X,S^1)$. 
\end{defn}
The following lemma is given in \cite[Proposition 8.11]{Berg& tao & ziegler}.
\begin{lem} [Exact descent] \label{Edec:lem} Let $X$ be an ergodic $G$-system of order $<k$ for some $k\geq 0$. Let $\pi:X\rightarrow Y$ be a factor map. Suppose that a function $f:G\times Y\rightarrow S^1$ is a quasi-cocycle of order $<k$. If $\pi^\star f$ is of type $<k$ then so is $f$.
\end{lem}
\begin{defn} [Line-cocycle]\label{lc:def} 
	Let $X$ be an ergodic $G$-system and let $f:G\times X\rightarrow S^1$ be a function. We say that $f$ is a line cocycle if for every $g\in G$ of order $n$ we have $\prod_{t=0}^{n-1} f(g, T_g^tx) = 1$.
\end{defn}

We claim that Theorem \ref{reductionH:thm} is a consequence of the following result.
\begin{thm} (The totally disconnected case in high characteristics) \label{TDH:thm}
	Let $1\leq j,k \leq \text{char} (G)$. Let $X$ be a totally disconnected and Weyl ergodic $G$-system of order $<j$ where the phase polynomial cocycles $\sigma_1,...,\sigma_j$ that define $X$ are of degrees $<1,...,<j$. Then every $f:G\times X\rightarrow S^1$ of type $<k$ which is also a line cocycle and a quasi-cocycle of order $<k-1$, is $(G,X,S^1)$-cohomologous to a phase polynomial $P$ of degree $<k$. Moreover, for $g\in G$ of order $n$ we have that $P(g,\cdot)$ takes values in $C_n$.
\end{thm}
We note that the fact about the values of $P$ is only needed for inductive reasons (see for example equation (\ref{Heq1})).
\begin{proof}[Proof of Theorem \ref{reductionH:thm} assuming Theorem \ref{TDH:thm}]
	Let $k,j,X,\rho$ be as in Theorem \ref{reductionH:thm}. Our goal is to reduce matters to the case where $X$ is totally disconnected and then apply Theorem \ref{TDH:thm}. We prove Theorem \ref{reductionH:thm} by induction on $k$. The case $k=1$ is trivial. Fix $k$ and assume that the claim holds for smaller values of $k$. We observe that since that the proof of Theorem \ref{Main:thm} is now complete, we have that $\rho$ is cohomologous to a phase polynomial of some degree (possibly higher than $k$). Therefore, we can apply Theorem \ref{con:thm} to eliminate the connected components of $X$. As in Theorem \ref{TDred:thm}, $X$ admits a totally disconnected and Weyl factor $Y$ and by the induction hypothesis on $k$ we have that the cocycles which define $Y$, $\sigma_1,...,\sigma_j$ are phase polynomials of degrees $<1,...,<j$. Moreover, $\rho$ is cohomologous to some $\pi^\star \rho'$ where $\rho':G\times Y\rightarrow S^1$. By Lemma \ref{Cdec:lem}, we have that $\rho'$ is of type $<k$ and therefore by Theorem \ref{TDH:thm} we have that it is cohomologous to a phase polynomial $P$ of degree $<k$. We conclude that $\rho$ is cohomologous to $\pi^\star P$ which is also a phase polynomial of degree $<k$.
\end{proof}
It is left to prove Theorem \ref{TDH:thm}. The proof method is very similar to \cite[Theorem 8.6]{Berg& tao & ziegler}. We first deal with the easy case $k=1$. In this case $f$ is a quasi-cocycle of order $<0$, and is thus a cocycle. It is well known (Lemma \ref{type0}) that a cocycle of type $<1$ is cohomologous to a constant $c(g)$. Therefore, we can write $f(g,x)=c(g)\cdot \Delta_g F(x)$ for some measurable maps $F:X\rightarrow S^1$ and $c:G\rightarrow S^1$. Since $f$ is a cocycle, $c$ is a character of $G$ and therefore $c(g)\in C_n$ for every $g\in G$ of order $n$ and the claim follows.

Now suppose that $2\leq k \leq \text{char}(G)$ and assume inductively that the claim has already been proven for smaller values of $k$. We have the following analogue of Theorem \ref{TST:thm}.
\begin{thm}  [Exact topological structure theorem for totally disconnected systems] \label{TSTH:thm}Let $1\leq k \leq \text{char}(G)$ be such that Theorem \ref{TDH:thm} holds for all values smaller or equal to $k$. Let $X$ be an ergodic totally disconnected Weyl  $G$-system of order $<k$ where the phase polynomial cocycles $\sigma_1,...,\sigma_{k-1}$ are of degree $<1,...,<k-1$. Then there exists a multiset of primes $A$ such that $Z_{<j}(X)=Z_{<j-1}(X)\times _{\sigma_{j-1}} U_j$, where $U_j \cong \prod_{p\in A} C_p$.
\end{thm}
When $j=1$, $X$ is just a point and Theorem \ref{TDH:thm} is trivial. Now suppose $2\leq j\leq \text{char}(G)$ and assume inductively the claim has already been proven for the same value of $k$ and smaller values of $j$.

We first deal with the lower case $j\leq k$, write $X=U_0\times_{\sigma_1}U_1\times...\times_{\sigma_{j-1}} U_{j-1}$ and let $t\in U_{j-1}$. We have the following result \cite[Lemma 8.8]{Berg& tao & ziegler}.
\begin{lem} \label{ED:lem} 
	$\Delta_t f$ is a line cocycle, is of type $<k-j+1$  and is a quasi-coboundary of order $<k-j$.
\end{lem}
By the induction hypothesis, $\Delta_t f$ is $(G,X,S^1)$-cohomologous to a phase polynomial $q_t\in P_{<k-j+1}(G,X,S^1)$, and $q_t(g,\cdot)$ takes values in $C_n$ for $g$ of order $n$. Since $\Delta_t f$ is a line cocycle and a quasi-cocycle of order $<k-j$ so is $q_t$.
\subsection{Reduction to the finite $U$ case} We now argue as in Theorem \ref{MainF:thm}. We show that is suffices to prove Theorem \ref{TDH:thm}, in the case when $U_{j-1}$ is finite.\\
We will take advantage of the following result of Bergelson Tao and Ziegler \cite[Proposition 8.9]{Berg& tao & ziegler}.
\begin{lem}  [Exact integration Lemma]\label{Eint:lem}
	Let $j \geq 0$, let $U$ be a compact abelian group, and let $X = Y \times_{\rho} U$ be an ergodic $G$-system with $Y \geq Z_{<j}(X)$, where $\sigma:G\times Y\rightarrow U$ is a phase polynomial cocycle of degree $<j$. For any $t\in U$, let $p_t:X\rightarrow S^1$ be a phase polynomial of degree $< l$, and suppose that for any $t,s\in U$, $p_{t\cdot s} = p_s(x) p_t(V_sx)$. Then there exists a phase polynomial $Q:X\rightarrow S^1$ of degree $<l+j$ such that $\Delta_t Q(x) = p_t(x)$. Furthermore we can take $Q(t,uu_0):=p_u (y,u_0)$ for some $u_0\in U$.
\end{lem}
We continue the proof of Theorem \ref{TDH:thm}. Let $f$ be as above and write $\Delta_t f = q_t\cdot \Delta F_t$. By Lemma \ref{lin:lem} (linearization Lemma) there exists an open neighborhood of the identity $U_{j-1}'$ in $U_{j-1}$ such that $q_{ts}=q_t V_t q_s$ whenever $s,t\in U_{j-1}'$. As in Theorem \ref{MainF:thm}, we can write $U_{j-1} = U'\times W$ for some finite $W$. Let $X = Y\times _{\sigma'} U'$ where $Y=Z_{<j-1}(X)\times_{\sigma''} W$ and $\sigma',\sigma''$ are the projections of $\sigma_{j-1}$ to $U'$ and $W$, respectively. Note that as $\sigma_{j-1}$ is a phase polynomial of degree $<j-1$, so is the projection $\sigma'$.\\
Applying Lemma \ref{Eint:lem} once for every $g\in G$, we can find a phase polynomial $Q:G\times X\rightarrow S^1$ such that $q_t = \Delta_t Q$ for every $t\in U_{j-1}$. In fact we can take $Q(g,y,uu_0)=q_u(g,y,u_0)$ for $y\in Y$, $u\in U'$ and some $u_0\in U'$. As $q_u(g,\cdot)$ takes values in $C_n$ whenever $g$ is of order $n$, so does $Q$.\\

We claim that $Q:G\times X\rightarrow S^1$ is a quasi-cocycle of order $<k-1$. The proof follows the arguments of Bergelson Tao and Ziegler from \cite{Berg& tao & ziegler}. For the sake of completeness we repeat the proof. For every $g,h\in G$ and $x=(y,uu_0)\in X$, we have the following computation.
$$\frac{Q(g+h,x)}{Q(g,x)Q(h,T_gx)} = \frac{Q(g+h,x)}{Q(g,x)Q(h,x)\Delta_g Q(h,x)}$$
$$=\frac{q_u(g+h,y,u_0)}{q_u(g,y,u_0)q_u(h,y,u_0)\Delta_g Q(h,x)}$$
$$=\frac{q_u(g+h,y,u_0)}{q_u(g,yu,u_0)q_u(h,T_gy,\sigma_{j-1}(g,y)u_0)}\frac{q_u(h,T_gy,\sigma_{j-1}(g,y)u_0)}{q_u(h,y,u_0)\Delta_g Q(h,x)}$$
$$=P_{u,g,h}(y,u_0)\Delta_g \left( \frac{q_u(h,x)}{Q(h,x)}\right)$$
where $P_{u,g,h}(x)=\frac{q_u(g+h,x)}{q_u(g,x)q_u(h,T_g x)}$. Since $q_u$ is a phase polynomial of degree $<k-j+1$ and $Q$ of degree $<k$, we conclude that $\Delta_g \frac{q_u(h,x)}{Q(h,x)}$ is a phase polynomial of degree $<k-1$. As $q_u$ is a quasi-cocycle of order $<k-j$, we have that $P_{u,g,h}(x)$ is a phase polynomial of degree $<k-j$. Moreover, since $u\mapsto q_u$ is a cocycle in $u$ so is $u\mapsto P_{u,g,h}(x)$. By the furthermore part of Lemma \ref{PIL:lem} the map $(y,uu_0)\mapsto P_{u,g,h}(y,u_0)$ is a phase polynomial of degree $<k-1$ and the claim follows.\\

We continue the proof of Theorem \ref{TDH:thm}. Set $f'=f/Q$. It follows from the construction of $Q$ that $\Delta_u f'\in B^1(G,X,S^1)$ for every $u\in U'$. Therefore, by Lemma \ref{cob:lem}, $f'$ is $(G,X,S^1)$-cohomologous to a function $f''$ which is invariant with respect to the action of some open subgroup $U''$ of $U'$. The projection map $U_{j-1}\rightarrow U_{j-1}/U''$ gives rise to a factor map  $\pi:Y\times_{\overline\sigma_{j-1}} U_{j-1}/U''$. We let $\tilde{f}=\pi_\star f''$ be the push-forward of $f''$ to $Y$.\\
Observe that if $g$ is of order $n$ then $Q(g,\cdot)$ takes values in $C_n$. Therefore, by Proposition \ref{linecocycle:prop} applied with $F=Q(g,\cdot)$, we see that $Q$ is a line cocycle. As $f''$ is cohomologous to $f/Q$ and $f''=\pi^\star \tilde{f}$, we conclude that $\pi^\star \tilde{f}$ and $\tilde{f}$ are line cocycles. Similarly, since $f$ and $Q$ are quasi-cocycles of order $<k-1$ so is $\tilde{f}$. \\
We show that $\tilde{f}$ is of type $<k$. First, observe that $Q$ is a phase polynomial of degree $<k$ and so by Lemma \ref{PP}(iii) it is of type $<k$. Since $f$ is also of type $<k$, $f/Q$ and $\pi^\star \tilde{f}$ are of type $<k$. Therefore, by Lemma \ref{Edec:lem}, we have that $\tilde{f}$ is of type $<k$. Since $\tilde{f}$ is measurable with respect to an extension of $X$ by a finite group we can apply the finite case of Theorem \ref{TDH:thm} to complete the proof. $\square$
\subsection{The finite group case}
We now establish Theorem \ref{TDH:thm} for a finite $U_j$. \\
First we recall the following Lemma \cite[Lemma C.8]{HK}.
\begin{lem} [Free actions of compact abelian groups have no cohomology] \label{B.4}
	Let $U$ be a compact abelian group acting freely on $X$ by measure preserving transformations. Then every cocycle $\rho:U\times X\rightarrow S^1$ is a $(U,X,S^1)$-coboundary.
\end{lem}

Using Theorem \ref{TSTH:thm} we can write $U_j = C_{p_1}^{L_1}\times...\times C_{p_m}^{L_m}$ for some finite $L_1,...,L_m$ and a finite $m\in\mathbb{N}$. We proceed by induction on $L=L_1+...+L_m$, the case $L=0$ is trivial. Suppose that $L\geq 1$, then without loss of generality $L_1\geq 1$ and we assume inductively that the claim has already been proven for $L-1$. Write $U_j = C_{p_1}^{L_1-1}\times \left<e\right>\times C_{p_2}^{L_2}\times... C_{p_m}^{L_m}$ where $e$ is a generator for $C_{p_1}$. For simplicity we denote $p=p_1$.

Recall from equation (\ref{Heq1}) that $\Delta_e f$ is $(G,X,S^1)$-cohomologous to a phase polynomial $q_e$ of degree $<k-j+1$ and that $q_e(g,\cdot)$ takes values in $C_n$ where $n$ is the order of $g$. We write,
\begin{equation}\label{Heq1}
\Delta_e f = q_e \cdot \Delta F_e.
\end{equation}
Arguing as in Theorem \ref{FC:thm} we can assume that $F_e$ satisfies that \begin{equation}\label{Heq2}
\prod_{i=0}^{p-1} V_e^i F_e = 1
\end{equation} without changing equation (\ref{Heq1}) and $q_e$ is still a phase polynomial of degree $<k-j+1$.\\
We now define a function $q_{e^s}(g,x)$ for all $0\leq s <p$ by the formula $$q_{e^s}(g,x) :=  \prod_{i=0}^{s-1} q_e(g, V_{e}^i x)$$ Since $q_e$ is a phase polynomial of degree $<k-j+1$, so is $q_{e^s}$ and by equations (\ref{Heq1}) and (\ref{Heq2}), we have that $\prod_{i=0}^{p-1} V_e^i q_e = 1$. It follows that $u\mapsto q_u$ is a cocycle for $u\in \left<e\right>$.\\
By the cocycle identity we have
\begin{equation}
\Delta_{e^s} f (g,x)= \prod_{i=0}^{s-1}\Delta_{e} f(g,V_e^ix)
\end{equation}
Since $\Delta_e f$ is cohomologous to $q_e$, it follows that $\Delta_u f$ is cohomologous to $q_u(g,x)$ for every $u\in \left<e\right>$.\\
Now, as $q_u$ is cocycle in $u$, applying the Polynomial integration lemma (Lemma \ref{Eint:lem}), we see that there exists a phase polynomial $Q$ of degree $<k$ such that $\Delta_u Q = q_u$, and $Q(g,\cdot)$ take values in $C_n$ where $n$ is the order of $g$. Moreover, arguing as in the previous section we also have that $Q$ is a quasi-cocycle of order $<k-1$.\\
From equation (\ref{Heq1}) we have $$\Delta_e (f/Q) = \Delta F_e.$$
From the telescoping identity
$$\prod_{i=0}^{p-1}V_e^i \Delta_e (f/Q)=1$$ we conclude that $\Delta \prod_{i=0}^{p-1}V_{e^i}F_e=1$ and so by ergodicity $\prod_{i=0}^{p-1}V_{e^i}F_e$ is a constant in $S^1$. Thus, we can rotate $F_e$ by a $p$-th root of this constant and assume that  $$\prod_{i=0}^{p-1}V_{e^i}F_e=1.$$
Now, by (\ref{Heq2}) we can define $$F_{e^s} :=\prod_{i=0}^{s-1}V_e^i F_e$$ Direct computation shows that $F_u$ is a cocycle for $u\in \left<e\right>$, hence a coboundary (Lemma \ref{B.4}). Thus, we can write $F_e = \Delta_e F$ for some $F:X\rightarrow S^1$. We conclude that $\Delta_e (f/(Q\cdot \Delta F))=1$ and so $f/Q$ is cohomologous to $f'$ which is invariant under $\left<e\right>$, arguing as before and using the induction hypothesis we conclude that $f/Q$ is cohomologous to a phase polynomial $P$. From Theorem \ref{HTDPV:thm} we have that $P(g,\cdot)$ takes values in $C_n$ where $n$ is the order of $g$. Since $f$ is $(G,X,S^1)$-cohomologous to $Q\cdot P$, this completes the proof. $\square$
\subsection{The higher order case}
The case $j> k$ is completely analogous to the proof of Bergelson Tao and Ziegler, \cite[Section 8.5]{Berg& tao & ziegler} and so is omitted.\\

We therefore completed the proof of Theorem \ref{MainH:thm}. 
\section{The proof of Theorem \ref{Mainr:thm}} \label{proof}
Now that Theorem \ref{Main:thm} is established we can prove Theorem \ref{Mainr:thm}. The theorem has two directions, we start with the following one (Compare with \cite[Theorem 3.8]{Berg& tao & ziegler}).
\begin{thm} Let $X$ be an ergodic $G$-system of order $<k$. Let $1\leq l \leq k$ and suppose that $Z_{<l-1}(X)$ is strongly Abramov and that there exists a totally disconnected factor $Y$ and a factor map $\pi_l:Z_{<l}(X)\rightarrow Z_{<l}(Y)$ such that $$\pi_l^\star: H^1_{<m} (G,Z_{<l}(Y),S^1)\rightarrow H^1_{<m}(G,Z_{<l}(X),S^1)$$ is onto for every $m\in\mathbb{N}$. Then, for every $1\leq l\leq k$, $Z_{<l}(X)$ is strongly Abramov.
\end{thm}
\begin{proof}
	We prove the claim by induction on $l$. If $l=1$, then $Z_{<1}(X)$ is a point. In this case we can take $Y$ to be the trivial system and the claim follows. Let $1< l \leq k$ and suppose that the claim has already been proven for smaller values of $l$. Let $Y$ be as in the theorem. Fix $m\in\mathbb{N}$ and let $\rho:G\times Z_{<l}(X)\rightarrow U$ be a cocycle of type $<m$ into some compact abelian group $U$. From the assumption, we see that for all $\chi\in\hat U$, $\chi\circ\rho$ is $(G,Z_{<l}(X),S^1)$-cohomologous to a cocycle $\rho_\chi$ which is measurable with respect to $Z_{<l}(Y)$. Let $\pi:Z_{<l}(X)\rightarrow Z_{<l}(Y)$ be the projection map. By Lemma \ref{Cdec:lem}, the push-forward $\pi_\star\rho_\chi$ is a cocycle of type $<m$ on $Z_{<l}(Y)$. Since $Z_{<l}(Y)$ is totally disconnected, Theorem \ref{Main:thm} implies that $\pi_\star\rho_\chi$ is $(G,Z_{<l}(Y),S^1)$-cohomologous to a phase polynomial of degree $<l_m$ for some $l_m=O_{k,m}(1)$. Lifting everything back to $Z_{<l}(X)$ using $\pi^\star$, we conclude that $\chi\circ\rho$ is $(G,Z_{<l}(X),S^1)$-cohomologous to a phase polynomial of degree $<l_m=O_{k,m}(1)$. Write $\chi\circ \rho = q_\chi \cdot \Delta F_\chi$ for some $F_\chi:Z_{<l}(X)\rightarrow S^1$ and a phase polynomial $q_\chi: G\times Z_{<l}(X)\rightarrow S^1$. The map $\Phi(x,u) = \chi(u)\cdot \overline{F}(x)$ is a phase polynomial of degree $<l_m+1$ on $Z_{<l}(X)\times_\rho U$ with derivative $q_\chi$. Since $Z_{<l-1}(X)$ is strongly Abramov, there exists some $r_l$ such that $Z_{<l}(X)$ is Abramov of degree $<r_l$. In particular, $\overline{F}$ can be approximated by phase polynomials of some bounded degree. Therefore, $\chi$ can be approximated by phase polynomials of degree $<\max\{r_l,l_m+1\}$. Finally, since $L^2(Z_{<l}(X)\times_\rho U)$ is generated by $L^2(Z_{<l}(X))$ and the characters in $\hat U$ this completes the proof.
\end{proof}
This proves one of the directions of Theorem \ref{Mainr:thm}. It is left to prove the other direction.
\begin{thm} Let $X$ be an ergodic $G$-system of order $<k$ and suppose that  $Z_{<1}(X),...,Z_{<k}(X)$ are strongly Abramov. Then, there exists a totally disconnected factor $Y$ such that for every $m\in\mathbb{N}$, the homomorphism $$\pi_l^\star : H_{<m}^1(G,Z_{<l}(Y),S^1)\rightarrow H_{<m}^1(G,Z_{<l}(X),S^1)$$ is onto for every $1\leq l \leq k$, where $\pi_l:Z_{<l}(X)\rightarrow Z_{<l}(Y)$ is the factor map.
\end{thm}
\begin{proof}
 We prove the claim by induction on $k$. If $k=1$, then the system $X$ is trivial and we can take $Y=X$. Let $k>1$ and suppose that the claim has already been proven for systems of order $<k-1$. Fix $1\leq l<k$ and recall that $Z_{<l+1}(X) = Z_{<l}(X)\times_{\sigma_l} U_l$ for some compact abelian group $U_l$ and a cocycle $\sigma_l:G\times Z_{<l}(X)\rightarrow U_l$ of type $<l$. Therefore, by applying the induction hypothesis and Theorem \ref{Main:thm} we have that $\chi\circ\sigma_l$ is $(G,Z_{<l}(X),S^1)$-cohomologous to a phase polynomial for every $\chi\in\hat U_l$.\\

Our main tool now is Proposition \ref{pinv:prop} which asserts that polynomials are invariant to the action of connected groups. We prove by induction on $l<k$ that there exists a compact connected abelian group $\mathcal{H}_l$ which acts on $Z_{<l}(X)$ by automorphisms and contains the group of transformations $U_{l-1,0}$. For $l=1$, we let $\mathcal{H}_1$ be the trivial group. Suppose inductively that we have already constructed $\mathcal{H}_l$ for some $l\geq 1$. For every $\chi\in\hat U_l$, we have that $\chi\circ\sigma_l$ is cohomologous to a phase polynomial. Therefore, we can write \begin{equation} \label{chi}\chi\circ\sigma_l = q_\chi\cdot \Delta F_\chi
\end{equation}
for some phase polynomial cocycle $q_\chi:G\times Z_{<l}(X)\rightarrow S^1$ and a measurable map $F_\chi:Z_{<l}(X)\rightarrow S^1$. Recall that by Proposition \ref{pinv:prop}, $q_\chi$ is invariant with respect to the action of the connected group $\mathcal{H}_l$. We take the derivative of both sides of equation (\ref{chi}) by some $h\in\mathcal{H}_l$, since the action of $\mathcal{H}_l$ commutes with the action of $G$, we conclude that $\Delta_h \chi\circ\sigma_l$ is a $(G,Z_{<l}(X),S^1)$-coboundary for every $h\in\mathcal{H}_l$. Since this holds for every $\chi\in \hat U_l$ we conclude that $\Delta_h \sigma_l$ is a $(G,Z_{<l}(X),U_l)$-coboundary.\\

For every $h\in\mathcal{H}_l$ and $F:Z_{<l}(X)\rightarrow U_l$, we define the measure preserving transformation $S_{h,F}(x,u) := (hx,F(x)u)$ on $Z_{<l}(X)\times U_l$. Let $H_{l+1}:=\{S_{h,F}:\Delta_h\sigma_l=\Delta F\}$. Direct computation reveals that the action of $H_{l+1}$ on $Z_{<l+1}(X)$ commutes with the $G$-action. Since $\mathcal{H}_{l}$ is abelian and commutes with the $G$-action on $Z_{<l}(X)$, we have that $H_{l+1}$ is $2$-step nilpotent. Indeed, if $S_{h,F}, S_{h',F'}\in H_{l+1}$, then $[S_{h,F},S_{h',F'}] = S_{1,\frac{\Delta_{h'}F}{\Delta_h F'}}$. Now, for every $h,h'\in H_{l+1}$ we have $$\Delta\left(\frac{\Delta_{h'}F}{\Delta_h F'}\right) = \frac{\Delta_{h'}\Delta F}{\Delta_h \Delta F'} = \frac{\Delta_{h'} \Delta_h \sigma_l}{\Delta_h \Delta_{h'} \sigma_l} = 1$$ and therefore by ergodicity $c:=\frac{\Delta_{h'}F}{\Delta_h F'}$ is a constant. A simple computation implies that $S_{1,c}$ commutes with $S_{s,F}$ and we conclude that $H_{l+1}$ is $2$-step nilpotent. We can view every element $u\in U_l$ as a function $Z_{<l}(X)\rightarrow U_l$ which sends every $s\in Z_{<l}(X)$ to this $u$. This defines an embedding of $U_l$ in $H_{l+1}$ by $u\mapsto S_{1,u}$ and so we can view $U_l$ as a subgroup of $H_{l+1}$. Finally since $\Delta_h \sigma_l$ is a $(G,Z_{<l}(X),U_l)$-coboundary, the projection $H_{l+1}\rightarrow \mathcal{H}_l$ is onto. In other words we have a short exact sequence
$$1\rightarrow U_{l+1}\rightarrow H_{l+1}\rightarrow \mathcal{H}_l\rightarrow 1.$$ Since $U_{l+1}$ and $\mathcal{H}_l$ are compact so is $H_{l+1}$ (Corollary \ref{compact}).\\

Now let $\mathcal{H}_{l+1}$ be the connected component of the identity in $H_{l+1}$. Then $\mathcal{H}_{l+1}$ is a compact connected nilpotent group and so is abelian (Proposition \ref{connilabel:prop}). Clearly it contains $U_{l,0}$ and it acts on $Z_{<l+1}(X)$ by automorphisms. This proves the induction step.\\ Let $\mathcal{H}=\mathcal{H}_k$, and define an equivalence relation $\sim_{\mathcal{H}}$ on $X$ such that $x\sim_{\mathcal{H}}y$ if there exists an element $h\in\mathcal{H}$ so that $hx=y$. Since $\mathcal{H}$ acts on $X$ by automorphisms, the quotient space $Y=X/\sim_{\mathcal{H}}$ has a natural $G$-action and is a factor of $X$. We return to the proof of the original claim.

Fix $1\leq l \leq k$ and $m\in\mathbb{N}$ and let $\rho:G\times Z_{<l}(X)\rightarrow S^1$ be a cocycle of type $<m$. Let $\tilde{X}:= Z_{<l}(X)\times_\rho S^1$. Since $Z_{<l}(X)$ is strongly Abramov, we have that $\tilde{X}$ is Abramov of order $<l_m$. We prove that $\rho$ is cohomologous to a phase polynomial. To do this we consider two cases:\\
\textbf{Case I:} $\tilde{X}$ is ergodic. If $P:\tilde{X}\rightarrow S^1$ is a phase polynomial, then by Proposition \ref{pinv:prop} we see that $\Delta_s P = \chi(s)$ for all $s\in S^1$ for some character $\chi:S^1\rightarrow S^1$. Therefore, $P = \chi \cdot F$ where $F:Z_{<l}(X)\rightarrow S^1$.\\ If $\chi:S^1\rightarrow S^1$ is not the identity, then $P=\chi\cdot F$ is orthogonal to the map $(x,u)\mapsto u$ on $\tilde{X}=Z_{<l}(X)\times_\rho S^1$. Since the system is Abramov, the polynomials generate $L^2(\tilde{X})$ and therefore some polynomial must be of the form $P(x,s)=sF(x)$ for $x\in Z_{<l}(X),s\in S^1$ and $F:Z_{<l}(X)\rightarrow S^1$.
By taking derivatives we see that $$\rho \cdot \Delta F = \Delta P.$$ Since $\Delta \Delta_s P = 1$ and $s$ commutes with the action of $G$, the cocycle $\Delta P$ is defines a phase polynomial on $Z_{<l}(X)$. In other words, $\rho$ is $(G,Z_{<l}(X),S^1)$-cohomologous to a phase polynomial of degree $<l_m-1$. By Proposition \ref{pinv:prop}, this polynomial is measurable with respect to $Z_{<l}(X)/\mathcal{H}_l \cong Z_{<l}(Y)$. This completes the proof in the case where $\tilde{X}$ is ergodic.\\
\textbf{Case II} If $\tilde{X}$ is not ergodic, then by the theory of Mackey $\rho$ is $(G,Z_{<l}(X),S^1)$-cohomologous to a minimal cocycle $\tau:G\times Z_{<l}(X)\rightarrow S^1$ which takes values in a proper closed subgroup of $S^1$ (see \cite[Corollary 3.8]{Zim}). In particular, this means that $\tau^n = 1$ for some $n\in\mathbb{N}$. Let $X$ be as in the theorem. We claim by induction on $m$ that every $(G,X,S^1)$-cocycle of type $<m$ is cohomologous to a phase polynomial of degree $<l_m$. If $m=0$ the claim is trivial. Therefore, using a proof by induction, we may assume that as a cocycle into $S^1$, $\Delta_h \tau$ is $(G,X,S^1)$-cohomologous to a phase polynomial for all $h\in \mathcal{H}_l$. Write,
$$\Delta_h \tau = p_h\cdot \Delta F_h$$ Since $\tau^n=1$, we conclude that $p_h^n$ is a coboundary. The cocycle identity and Proposition \ref{pinv:prop} implies that $p_{h^n}$ is a coboundary for all $h\in \mathcal{H}_l$. Since $\mathcal{H}_l$ is connected, it is also divisible (Lemma \ref{divisible}) and so $\Delta_h \tau$ is a coboundary for all $h\in\mathcal{H}_l$. Finally, since $\mathcal{H}_l$ acts freely on $Z_{<l}(X)$ by automorphisms, Lemma \ref{cob:lem} implies that $\tau$ is $(G,Z_{<l}(X),S^1)$-cohomologous to a cocycle that is invariant under $\tilde{H}_l$. This cocycle is measurable with respect to $Z_{<l}(Y)$ and therefore is cohomologous to a phase polynomial (by Theorem \ref{Main:thm}). Since $\tau$ and $\rho$ are cohomologous this completes the proof.
\end{proof}
\section{A $G$-system that is not Abramov} \label{example}
In this section we provide an example of a $G$-system $X$ of order $<3$ that is not an Abramov system. This example is based on the Furstenberg-Weiss example for a $\mathbb{Z}$-cocycle that is not cohomologous to a polynomial (which is given in detail in \cite{HK02}). Our example is constructed as a circle-group extension of a finite dimensional compact abelian group (see Definition \ref{FD:def}) that is not a Lie group. Such groups are called solenoids and are known for their pathological properties.\\ We first construct the underlying group and the $G$-action on it.
\subsection{The underlying group}
Let $\mP$ denote the set of all prime numbers. We construct a $2$-dimensional compact abelian group as follows: Let $\mP_1,\mP_2$ be disjoint infinite sets such that $\mP=\mP_1\bigsqcup\mP_2$ and let $\Delta_1 := \prod_{p\in \mP_1} C_p$ and $\Delta_2:= \prod_{p\in \mP_2} C_p$. Fix $i\in\{1,2\}$ and let $\mathbb{P}_i=\{p_1,p_2,...\}$. The map $\imath_i(n)= (n,\omega_{p_1}^n,\omega_{p_2}^n,\omega_{p_3}^n,...)$ where $\omega_{p_k}$ is the first root of unity of order $p_k$, defines an embedding of $\mathbb{Z}$ as a closed subgroup of $\mathbb{R}\times \Delta_i$. We abuse notation and denote by $(\mathbb{R}\times \Delta_i)/\mathbb{Z}$ the quotient of $\mathbb{R}\times \Delta_i$ by $\imath_i(\mathbb{Z})$. This gives rise to two $1$-dimensional compact abelian groups which we denote by $U_1:=(\mathbb{R}\times \Delta_1)/\mathbb{Z}$ and $U_2:=(\mathbb{R}\times \Delta_2)/\mathbb{Z}$. We note that for every $i\in\{1,2\}$, $$\{1\}\rightarrow \Delta_i\rightarrow U_i\rightarrow \mathbb{R}/\mathbb{Z}\rightarrow \{1\}$$ is a short exact sequence, where the embedding $\Delta_i\rightarrow U_i$ is given by $t\mapsto (1,t)\mathbb{Z}$ and the quotient $U_i\rightarrow \mathbb{R}/\mathbb{Z}$ by $(r,t)\mathbb{Z}\mapsto r\mathbb{Z}$.\\
\textbf{Claim}: The $2$-dimensional compact abelian group $U:=U_1\times U_2$ is connected.
\begin{proof}
Let $U_0\leq U$ be the connected component of the identity in $U$. By Proposition \ref{connectedcomponent}, $U/U_0$ is a totally disconnected group. If by contradiction $U$ is not connected, then by Corollary \ref{chartdg} there exists a non-trivial character $\chi:U\rightarrow S^1$ with a finite image. By composing $\chi$ with the projection map $\mathbb{R}^2\times \Delta_1 \times \Delta_2\rightarrow U$ we obtain a character $\tilde{\chi}:\mathbb{R}^2\times \Delta_1 \times \Delta_2\rightarrow S^1$ with finite image. We conclude that $\ker\tilde{\chi}$ is an open subgroup of $\mathbb{R}^2\times \Delta_1 \times \Delta_2$ which is invariant to translations by $\imath_1(\mathbb{Z})$ and $\imath_2(\mathbb{Z})$, hence $\ker\tilde{\chi} = \mathbb{R}^2\times \Delta_1 \times \Delta_2$ and $\tilde{\chi}=1$. It follows that $\chi$ is trivial, which is a contradiction.
\end{proof}
\subsection{The $G$-action} \label{action}
Let $G=G_1\oplus G_2$ where $G_{i} = \bigoplus_{p\in\mP_i}\mathbb{F}_p$. We construct a homomorphism $\sigma:G\rightarrow U$ such that the $G$-system $(U,G)$ is an ergodic Kronecker system:\\ 
Let $i,j$ such that $\{i,j\}=\{1,2\}$ and $g$ be a generator of the component $\mathbb{F}_p$ of $G_i$. We denote by $v_g\in \Delta_j$ the unique $p$-th root of the element $(\omega_{p_1},\omega_{p_2},...)\in\Delta_j$. This root is easily constructed component by component, since by our assumption $p\not\in \mathbb{P}_i$. We let $\sigma_i(g)$ be the element $(\frac{1}{p},v_g)\in U_j$ (so $G_i$ acts on $U_j$). This is an element of order $p$ in $U_j$ hence $\sigma$ extends uniquely to a homomorphism.\\
Then, the group $G$ acts on $U$ by $T_gu = \sigma(g)\cdot u$ where $\sigma(g)=(\sigma_1(g),\sigma_2(g))$. We claim that this action is ergodic. Equivalently, we need to show that the image of $G$ under $\sigma$ is dense (see \cite[Corollary 3.8]{Zim}). Let $\pi:\mathbb{R}\times\Delta_1 \times\mathbb{R}\times \Delta_2\rightarrow U$ be the quotient map.
We consider a general open set of the form $W_1\times v_1\cdot V_1\times W_2 \times v_2\cdot V_2$ where $W_1,W_2$ are balls in $\mathbb{R}$, $v_iV_i$ is a co-set of some open subgroup $V_i\leq \Delta_i$ for $i=1,2$ (by Proposition \ref{opensubgroup} every open subset contains a set of this form). Rotating by an element in $\mathbb{Z}$ we may assume that $W_1,W_2$ intersects non-trivially with $(0,1)$. It is enough to show that $\pi^{-1}(\sigma(G))$ intersect with this set. Let $s$ and $t$ denote the sizes of $\Delta_1/V_1$ and $\Delta_2/V_2$ respectively. Let $n$ be a sufficiently large number depending only on $W_1,W_2,s,t$ that we will choose later. Let $g_1\in G_1$ and $g_2\in G_2$ be two generators of orders $p_1,p_2$ for $p_1,p_2>n$.
Recall that $v_{g_i}^{p_i}$ is a generator of $\Delta_i$ and therefore so is $v_{g_i}$. Since $\Delta_1/V_1,\Delta_2/V_2$ are finite, we can find powers $m_i,m_{i}'\leq \max\{s,t\}$ such that $v_{g_i}^{m_i}\in v_i\cdot V_i$ and $v_{g_i}^{m_{i}'}\in V_i$.
Hence for every $k$ we have that $v_{g_i}^{m_i+km_{i}'}$ is in $v_i\cdot V_i$. Since $m_i'$ depend only on $s$ and $t$, we see that for $n$ sufficiently large, one of the elements in $\{\frac{m_i+km_i'}{p_i}:k\in\mathbb{N}\}$ must intersect with $W_i\cap (0,1)$ for all $p_i>n$. Let $g=((m_1+km_1')g_1,(m_2+km_2')g_2)$, then $\sigma(g)$ is an element in the image of $W_1\times v_1\cdot V_1 \times W_2\times v_2 V_2$ under $\pi$. Thus, $\sigma(G)$ intersects with any open subset in $U$, therefore is dense.
\subsection{A cocycle that is not cohomologous to a phase polynomial}
For every $g\in G_1\oplus G_2$ choose any $\alpha_g=(a_g,b_g)\in\mathbb{R}^2$ such that $a_g \text{ mod 1}$ is equal to the first coordinate of $\sigma(g)$ in $U_1$ and $b_g \text{ mod 1}$ to the first coordinate in $U_2$. In particular, if $g\in G_1$ then $a_g=0$ and if $g\in G_2$ then $b_g=0$.\\
We define a bilinear form $\phi:\mathbb{R}^2\times \mathbb{R}^2\rightarrow S^1$ by the formula $\phi(x,y)=e(x_1y_2-x_2y_1)$ where $e(a)=e^{2\pi i a}$ and let $\tilde{f}(x) = \phi(x,\lfloor x\rfloor)$ where $\lfloor x\rfloor$ is the integer part of $x$ coordinate-wise. For every $k\in\mathbb{Z}^2$ we have,
\begin{equation}\label{eq1}
\tilde{f}(x+k)=\tilde{f}(x)\cdot \phi(x,k)
\end{equation}
This gives rise to a function $f:G\times \mathbb{R}^2\rightarrow S^1$ which is given by $$f(g,x) = \frac{\tilde{f}(x+\alpha_g)}{\tilde{f}(x)}\cdot \overline{ \phi(\alpha_g,x)}.$$ Equation (\ref{eq1}) implies that $f$ is well defined on $(\mathbb{R}/\mathbb{Z})^2$. We note that $(\mathbb{R}/\mathbb{Z})^2\cong U/(\Delta_1\times \Delta_2)$ and so we can think of $f$ as a function on $U$ which is invariant under multiplication by elements in $\Delta_1\times \Delta_2$.
\begin{lem} \label{fCL}
    For every $t\in\mathbb{R}^2$, $\Delta_t f(g,x)$ is $(G,U,S^1)$-cohomologous to $\overline{\phi(\alpha_g,t)}{\phi(t,\alpha_g)}$.
\end{lem}
\begin{proof}
	Let $t\in\mathbb{R}^2$ and let $F_t(x) :=\frac{\tilde{f}(x+t)}{\tilde{f}(x)}\cdot \overline{ \phi(t,x)}$, equation (\ref{eq1}) implies that $F_t$ is well defined on $(\mathbb{R}/\mathbb{Z})^2$. A direct computation shows that $\Delta_t f = \overline{ \phi(\alpha_g,t)}\phi(t,\alpha_g)\cdot \Delta_{\alpha_g} F_t$.
\end{proof}
Observe that for every $g\in G$, the constants $\overline{ \phi(\alpha_g,t)}\phi(t,\alpha_g)$ give rise to a character of $\mathbb{R}^2$, $\chi(g,x):= \phi(\alpha_g,x)\cdot \overline{\phi(x,\alpha_g)}$. By the lemma above we have that $\Delta_t (f\cdot \chi) = \Delta F_t$, for every $t\in\mathbb{R}^2$.\\
We extend $\chi$ to a character of $U$. To do this, we define a map $\varphi:G\times \Delta_1\times \Delta_2 \rightarrow S^1$ by the formula $\varphi(g,v_1,v_2):=(v_2(g))^2\cdot (v_1(g))^{-2}$ where $v_1,v_2$ are identified with the corresponding elements in $\hat G\cong \Delta_1\times \Delta_2$. It is not hard to see that $\chi\cdot\varphi:G\times \mathbb{R}^2\times \Delta_1\times\Delta_2\rightarrow S^1$ is well defined as a function (homomorphism) on $U$ (i.e. it is invariant under multiplication by an element in $\mathbb{Z}$).\\
$f\cdot \chi\cdot\varphi$ is not a cocycle, but it satisfies the following facts:
\begin{thm} \label{cocycle}
	$f\cdot\chi\cdot \varphi$ is a line cocycle and $\Delta_h f\cdot \chi\cdot \varphi (g,\cdot) = \Delta_g f\cdot \chi\cdot \varphi(h,\cdot)$.
\end{thm}
Assuming this theorem, there is a unique cocycle $\rho:G\times U\rightarrow S^1$ which agrees with $f\cdot\chi\cdot\varphi$ on the generators of $G$. This cocycle is given by the following equation.\footnote{This fact is the cocycle-counterpart of the fact that a homomorphism is uniquely determined by the values it gives to a generating set.}
\begin{equation} \label{rho}\rho(g,u) = \prod_{i=1}^\infty T_{g_1}T_{g_2}...T_{g_{i-1}} \prod_{k=0}^{g_i} f\cdot\chi\cdot \varphi(e_i,T_{e_i}^k)
\end{equation}
We note that the infinite product is well defined because it is trivial outside of a finite set for every given $g\in G$.

\begin{proof}[Proof of Theorem \ref{cocycle}]
	Let $g\in G$ be an element of order $n$. We then have,
	$$\prod_{k=0}^{n-1} f\cdot\chi\cdot\varphi(g,T_g^k(x,v)) = \prod_{k=0}^{n-1} \frac{\tilde{f}(x+(k+1)\alpha_g)}{\tilde{f}(x+k\alpha_g)}\overline{\phi(x+k\alpha_g,\alpha_g)} \cdot \varphi(g,v+v_g^k)$$
	We can break the product on the right hand side of the equations to three products. That is, $\prod_{k=0}^{n-1}\frac{\tilde{f}(x+(k+1)\alpha_g)}{\tilde{f}(x+k\alpha_g)}$, $\prod_{k=0}^{n-1}\overline{\phi(x+k\alpha_g,\alpha_g)}$ and $\prod_{k=0}^{n-1}\varphi(g,v+v_g^k)$. We compute each term separately. The first term is a telescoping series and so equals to $\frac{\tilde{f}(x+n\alpha_g)}{\tilde{f}(x)}$. Since $\phi$ is bilinear and $\phi(\alpha_g,\alpha_g)=1$, the second term equals to  $$\overline{\phi(nx+\binom{n}{2}\alpha_g,\alpha_g)}=\overline{ \phi(x, n\alpha_g)}.$$ Finally, since $\varphi(g,\cdot)$ is a homomorphism, the third term equals to $\varphi(g,v)^n \cdot \varphi(g,v_g) ^{\binom{n}{2}}$. It follows from the definition of $\varphi$ that $g\mapsto \varphi(g,v)$ is a homomorphism for every $v\in \Delta_1\times \Delta_2$. Since $n$ divides $2\cdot \binom{n}{2}$ this implies that the last term is trivial\footnote{Note that if $n=2$ then $n$ does not divide $\binom{n}{2}$, however in this case we have that $\varphi(g,\cdot)=1$ immediately from the definition.}. We conclude that
	$$\prod_{k=0}^{n-1} f\cdot\chi\cdot\varphi(g,T_g^k(x,v)) = \frac{\tilde{f}(x+n\alpha_g)}{\tilde{f}(x)}\overline{ \phi(x,n\alpha_g)} = 1.$$
	In other words, $f\cdot\chi\cdot\varphi$ is a line-cocycle.\\
	
	We now prove the second property, since $\Delta_t f \cdot \chi (g,x)=  \Delta_{\alpha_g} F_t(x)$ one has that $$\Delta_{\alpha_h} f\cdot \chi(g,x) = \Delta_{\alpha_g} F_{\alpha_h}(x) = \Delta_{\alpha_g} f(h,x).$$ Recall the definition of $\sigma$ from section \ref{action}. We abuse notation and write $\sigma(g) = (\alpha_g,v_g)$ where $v_g\in \Delta_1\times \Delta_2$ (Note that if $g$ is a generator of any $\mathbb{F}_p$-component of $G$, then this $v_g$ coincides with $v_g$ from section \ref{action}). It is enough to show that
	\begin{equation}\label{equal1}\Delta_{v_h} \overline{\varphi(g,v)} \Delta_{\alpha_g} \chi(h,x)\cdot \Delta_{v_g}\varphi(h,v)=1
	\end{equation}
	for every $g,h\in G$, $x\in\mathbb{R}^2$ and $v\in \Delta_1\times \Delta_2$.
	Since $\chi(h,\alpha_g)\cdot \varphi(h,v_g)$ is a homomorphism in $g$, we conclude that the map $(g,h)\mapsto \Delta_{v_h} \overline{\varphi(g,v)} \Delta_{\alpha_g} \chi(h,x)\cdot \Delta_{v_g}\varphi(h,v)$ is bilinear in $g$ and $h$. Hence, by linearity it is enough to check equation (\ref{equal1}) in the case where $h$ and $g$ are generators. For simplicity we denote the order of $h$ and $g$ by $p_h$ and $p_g$ respectively. We begin with the case where $h=g$. In that case the claim follows since $\chi(h,\alpha_h)=1$ and the terms $\Delta_{v_h} \overline{\varphi(g,v)}$ and $\Delta_{v_g}\varphi(h,v)$ cancel each other out. Otherwise, $h\not = g$ and then the right hand side of equation (\ref{equal1} is of order $p_h$ and of order $p_g$ simultaneously. Since $p_h\not = p_g$ the claim in the equation follows.
\end{proof}
Let $\rho$ be as in (\ref{rho}), then the extension defined by $\rho$ is not Abramov of any order. Formally we prove the following result.
\begin{thm}
	The system $X=U\times_{\rho} S^1$ is an ergodic system of order $<3$ and the measurable map $(x,u)\mapsto u$ is orthogonal to all phase polynomials.
\end{thm}
\begin{proof}
	First we claim that the system is ergodic. It is enough to show that $\rho$ is not $(G,U,S^1)$-cohomologous to a cocycle taking values in some proper subgroup of $S^1$ (See \cite[Corollary 3.8]{Zim}). Suppose by contradiction that there is an $n\in\mathbb{N}$ such that $\rho^n$ is a $(G,U,S^1)$-coboundary. Then the Conze-Lesigne equation gives $$\Delta_t \rho = \lambda_t\cdot \Delta F_t$$ for every $t\in U$ where $\lambda_t:G\rightarrow S^1$ is a homomorphism such that $\lambda_t^n$ is a coboundary. The cocycle identity implies that $\lambda_{t^n}$ is a coboundary. Since $U$ is connected, it is divisible (Lemma \ref{divisible}), so every $u\in U$ can be written as $t^n$. Lemma \ref{cob:lem} then implies that $\rho$ is cohomologous to constant. This is a contradiction and so $X$ is ergodic.\\
	Now we prove that $X$ is of order $<3$, equivalently show that $\rho$ is of type $<2$. By the Conze-Lesigne equation with $t=x'\cdot x^{-1}$ we have $$\frac{\rho(x)}{\rho(x')} = \lambda_{x-x'} \Delta F_{x-x'}(x)$$ Since the action of $G$ on $U$ is given by a homomorphism we have that the map $(x,x')\mapsto \Delta F_{x-x'}(x)$ is a derivative of $G(x,x'):=F_{x-x'}(x)$ in $U\times U$ therefore is a coboundary, hence $(x,x')\mapsto \frac{\rho(x)}{\rho(x')}$ is cohomologous to $\lambda_{x-x'}$ which is invariant under the action of $G$ and so is of type $<1$. We conclude that $\rho$ is of type $<2$. Finally, since $X$ is connected there are only phase polynomials of degree $<2$ (same proof as in the claim in Example \ref{Example}). Such phase polynomials are measurable with respect to the Kronecker system $U$. In particular the map $(u,x)\mapsto x$ is orthogonal to every phase polynomial. Therefore $X$ is not Abramov as required.
\end{proof}
\section{Nilpotent systems}
The goal of this section is to prove Theorem \ref{nilpotentstructure}. 
Recall that an ergodic system of order $<3$ takes the form $X=Z\times_\rho U$ where $Z$ and $U$ are compact abelian groups and $Z$ is the Kronecker factor. Since every compact abelian group $U$ is an inverse limit of Lie groups we can assume that $U$ is a product of a torus and a finite group.

\begin{defn} [Host and Kra group for systems of order $<3$] \label{HKGroup:def} Let $X=Z\times_\rho U$ be an ergodic $G$-system of order $<3$. For every $s\in Z$ and a measurable map $F:Z\rightarrow U$, we have a measure preserving transformation $S_{s,F}(z,u)=(sz,F(z)u)$ on $X$. We let $\mathcal{G}(X)$ denote the group of all such transformations with the property that there exists $c_s:G\rightarrow U$ such that $\Delta_s \rho = c_s \cdot \Delta F$.
\end{defn} 
Equipped with the topology of convergence in measure, Host and Kra \cite[Lemma 8.7 and Corollary 5.9]{HK} proved that $\mathcal{G}(X)$ is a locally compact polish $2$-step nilpotent group. Let $p:\mathcal{G}(X)\rightarrow Z$ denote the projection to the first coordinate $S_{s,F}\mapsto s$. Observe that if $s=1$, then $\Delta F=\overline{c_1}$ and so $F\in P_{<2}(Z,U)$. In other words we can identify $\ker(p)$ with $P_{<2}(Z,U)$.\\

We claim that in order to prove Theorem  \ref{nilpotentstructure} it is enough to show that $p$ is onto. In that case Theorem \ref{openmappingthm} implies that $p$ is an open map. Let $u\in U$ and define $F_u:Z\rightarrow U$ be the constant map $F_u(z)=u$. It is an immediate application of the definitions that $S_{1,F_u}\in \mathcal{G}(X)$. Therefore, (assuming that $p$ is onto) we have that the group $\mathcal{G}(X)$ acts transitively on $X$.  Let $\Gamma$ be the stabilizer of $(1,1)\in Z\times U$ and assume further that $U$ is a Lie group (as mentioned above, using inverse limits to approximate $U$ we can assume that this is always the case). Direct computation reveals that $\Gamma\cong \text{Hom}(Z,U)$ as topological groups. Since $U$ is a Lie group, the latter is a discrete co-compact subgroup of $\mG$ and $X$ is homeomorphic to $\mG/\Gamma$ (see Theorem \ref{openmappingthm}). Moreover, it is easy to see that for every element $g\in G$ the transformation $T_g:X\rightarrow X$ belongs to $\mathcal{G}(X)$. This gives rise to a natural $G$-action on $\mathcal{G}(X)/\Gamma$ by $g.(x\Gamma)=(T_gx)\Gamma$. It follows that there is an isomorphism (of $G$-systems)  $X\cong\mathcal{G}/\Gamma$. This completes the proof of Theorem \ref{nilpotentstructure} (assuming that $p$ is onto).\\

Now we prove that $p$ is onto. Equivalently, we show that for every $s\in Z$ we can find $F_s:Z\rightarrow U$ such that $S_{s,F_s}\in\mathcal{G}(X)$. Recall that the Lie group $U$ is a direct product of a torus and a finite group. Therefore it is enough to solve this equation coordinate-wise.\\

If the coordinate is associated with the torus subgroup of $U$, we can apply the following important result of Moore and Schmidt \cite{MS}.
\begin{lem} [Cocycles of type $<1$ are cohomologous to constants] \label{type0}
	Let $X$ be an ergodic $G$-system. Suppose that $\rho:G\times X\rightarrow S^1$ is a cocycle of type $<1$. Then there exists a character $c:G\rightarrow S^1$ such that $\rho$ is $(G,X,S^1)$-cohomologous to $c$.
\end{lem}
We note that this result fails for cocycles which take values in an arbitrary compact abelian groups (in particular, it fails for cocycles into finite groups).
\begin{proof}[Proof of Theorem \ref{nilpotentstructure}]
	From the discussion above we see that we need to show that for every $s\in Z$ there exist a measurable map $F_s:Z\rightarrow U$ and a constant $c_s:Z\rightarrow U$ such that $\Delta_s \rho = c_s\cdot \Delta F_s$.\\
	
	Let $T(U)$ denote the torus subgroup of $U$ and $D(U)$ the discrete (finite) subgroup. Then, $U=T(U)\times D(U)$. Let $\rho_T$ and $\rho_D$ be the projection of $\rho$ to $T(U)$ and $D(U)$ respectively and let $X_T = Z\times_{\rho_T} T(U)$ and $X_D = Z\times_{\rho_D} D(U)$. It is enough to show that the projections $p_T:\mG(X_T)\rightarrow Z$ and $p_D:\mG(X_D)\rightarrow Z$ are onto. Indeed, if $F_{T,s}:Z\rightarrow T(U)$ and $F_{D,s}:Z\rightarrow D(U)$ are such that $S_{s,F_{T,s}}\in \mG(X_T)$ and $S_{s,F_{D,s}}\in \mG(X_D)$, then by definition, the map $F_s(z) = (F_{T,s}(z),F_{D,s}(z))$ satisfies that $S_{s,F_s}\in \mG(X)$.\\
	
	Let $s\in Z$. Lemma \ref{dif:lem} implies that $\Delta_s\rho_T$ and $\Delta_s \rho_D$ are of type $<1$. By Lemma \ref{type0} applied for $\rho_T$, we see that $\Delta_s \rho_T$ is cohomologous to constant. Equivalently, $p_T$ is onto.\\ Now we deal with the finite case. The structure theorem of finite abelian groups asserts that every finite group is a direct product of $C_{p^n}$ for some prime $p$ and $n\in\mathbb{N}$. By working out each coordinate we may assume that $U=C_{p^n}$ for such $p$ and $n$. By embedding $C_{p^n}$ in $S^1$ and applying Lemma \ref{type0} we see that
	\begin{equation} \label{Cl}\Delta_s \rho_D = c_s\cdot \Delta F_s
	\end{equation} for some constant $c_s:G\rightarrow S^1$ and $F_s:Z\rightarrow S^1$. Our goal is to replace $F_s$ and $c_s$ with some $F_s'$ and $c_s'$ such that equation (\ref{Cl}) holds and $F_s',c_s'$ takes values in $C_{p^n}$.\\
	
	As a first step we show that $\rho_D$ is $(G,Z,S^1)$-cohomologous to a phase polynomial of degree $<2$. To do this we must first eliminate the connected component of $Z$. Observe that by the cocycle identity we have, $$\Delta_{s^{p^n}}\rho_D = \Delta_s\rho_D^{p^n} \cdot \prod_{k=0}^{p^n-1} \Delta_s \Delta_{s^k} \rho_D.$$
	From equation (\ref{Cl}) we see that $\prod_{k=0}^{p^n-1} \Delta_s \Delta_{s^k} \rho_D$ is a coboundary. Moreover, since $\rho_D$ takes values in $C_{p^n}$, the term $\Delta_s\rho_D^{p^n}$ vanishes and we conclude that $\Delta_{s^{p^n}} \rho_D$ is a coboundary for every $s\in Z$. Let $Z_0$ be the connected component of the identity in $Z$. Since connected groups are divisible (Lemma \ref{divisible}), we conclude that $\Delta_s\rho_D$ is a $(G,Z,S^1)$-coboundary for every $s\in Z_0$. By Lemma \ref{cob:lem}, $\rho_D$ is $(G,Z,S^1)$-cohomologous to a cocycle $\rho'$ that is invariant with respect to the action of $Z_0$. Let $\pi_\star \rho'$ be the push-forward of $\rho'$ to $Z/Z_0$. By Lemma \ref{Cdec:lem}, $\pi_\star \rho'$ is of type $<2$. Therefore, by Theorem \ref{MainH:thm} it is cohomologous to a phase polynomial of degree $<2$. Lifting everything back to $Z$ we conclude that $\rho'$ and $\rho_D$ are $(G,Z,S^1)$-cohomologous to a phase polynomial $Q:G\times Z\rightarrow S^1$ of degree $<2$. Moreover, $Q$ is invariant to translations by $Z_0$. We write, \begin{equation} \label{Poly}\rho = Q\cdot \Delta F\end{equation} for some $F:Z\rightarrow S^1$.\\
	
	Since $\rho_D$ takes values in $C_{p^n}$ we have that \begin{equation} \label{Cl2}1=Q^{p^n}\cdot\Delta F^{p^n}.
	\end{equation}
	By taking the derivative of both sides of the equation above by $s\in Z$, we conclude that $\Delta_s F^{p^n}$ is a phase polynomial of degree $<2$. Our next goal is to replace $F$ with a function $F'$ such that $F'/F$ is a phase polynomial of degree $<3$ (and so equation (\ref{Poly}) holds if we replace $Q$ with another phase polynomial cocycle of degree $<2$) and at the same time that $\Delta_s F'^{p^n}$ is a constant.\\
	We study the phase polynomial $Q$. It is a fact that every phase polynomial of degree $<2$ is a constant multiple of a homomorphism. Therefore, we can write $Q(g,x)=c(g)\cdot q(g,x)$ where $c:G\rightarrow S^1$ and $q:G\times Z\rightarrow S^1$ is a homomorphism in the $Z$-coordinate. Since $Q$ is a cocycle $$c(g+g')q(g+g',x)=c(g)c(g')\Delta_{g'}q(g,x)\cdot q(g,x)\cdot q(g',x).$$ It follows that $q$ is bilinear in $g$ and $x$. Let $$Z'_{p} = \ker(q^{p^n}) = \{s\in Z : q(g,s)^{p^n}=1 \text{ for every } g\in G\}.$$ 
	Since $q$ is bilinear, $Z/Z'_p$ is isomorphic to a subgroup of $\hat G^{p^n}=\prod_{p\not= q\in \mathcal{P}} C_q$. We can take the derivative of both sides of equation (\ref{Cl2}) by $s\in Z_p'$. We conclude by the ergodicity of the Kronecker factor that $\Delta_s F^{p^n}$ is a constant. Therefore, by Corollary \ref{openker}, there exists an open subgroup $Z'\leq Z$ which contains $Z_p'$ such that $\Delta_s F^{p^n}$ is a constant for every $s\in Z'$. By the cocycle identity, we conclude that $\Delta_s F^{p^n}=\chi(s)$ for some character $\chi:Z'\rightarrow S^1$. Lift $\chi$ to a character of $Z$ arbitrarily, we conclude that $F^{p^n}/\chi$ is a phase polynomial which is invariant under translations by $Z'$. Since $Z'$ is open and $Z$ is compact the quotient $Z/Z'$ is a finite group. Moreover, since $Z'$ contains $Z_p'$, we conclude that the order of $Z/Z'$ is co-prime to $p$. By Theorem \ref{HTDPV:thm} applied on the finite system $Z/Z'$, we conclude that up to constant multiplication $F^{p^n}/\chi$ takes values in some finite subgroup $C_m$ of $S^1$, with $(p,m)=1$. By rotating $F$ with a $p^n$-th root of this constant we can assume that this constant is trivial. Since $p$ and $m$ are co-prime, we can find an integer $l$ such that $l\cdot p^n = 1 \mod m$. We conclude that $R:= (F^{p^n}\chi)^l$ is a phase polynomial of degree $<3$ and that $R^{p^n} = F^{p^n}/\chi$. Let $Q':=Q\cdot \Delta R$ and $F':=F/R$. Then, as in equation (\ref{Poly}), we have $$\rho_D = Q'\cdot \Delta F'$$ and $\Delta_s F'^{p^n} = \chi(s)$.\\
	
Now, by taking the derivative by $s\in Z$ on both sides of the equation above we conclude that $$\Delta_s \rho = \Delta_s Q' \cdot \Delta \Delta_s F'.$$ Observe that $c_s':=\Delta_s Q'$ is a character of $G$ and $$c_s'^{p^n} = \Delta_s Q^{p^n} \cdot \Delta \Delta_s F^{p^n}/\chi = 1$$ where the last equality follows from (\ref{Cl2}) and the fact that $\Delta \Delta_s \chi$ vanishes. It is left to change the term $\Delta_s F$. Set $F_s' := \Delta_s F'/\phi(s)$ where $\phi(s)$ is a $p^n$-th root of $\chi(s)$ in $S^1$. Then, as before we have that $$\Delta_s \rho_D = c_s'\cdot \Delta F_s'$$ but this time $c_s'^{p^n} = F_s'^{p^n}=1$. This implies that $p_D$ is onto and the proof is now complete.
\end{proof}
\section{The limit formula and the Khintchine-type recurrence} \label{khintchine}
In this section we let $\Phi_N$ be any F{\o}lner sequence of the group $G=\bigoplus_{p\in P}\mathbb{F}_p$. For a function $f:G\rightarrow \mathbb{C}$ we write $\mathbb{E}_{g\in \Phi_N} f(g)$ for the average $\frac{1}{|\Phi_N|}\sum_{g\in \Phi_N} f(g)$. We study the limit of averages of the form $$ \mathbb{E}_{g\in \Phi_N} T_g^n f_1 T_{2g}^n f_2 T_{3g}^n f_3$$ where $f_1,f_2,f_3\in L^\infty(X)$.\\ The following results for $\mathbb{F}_p^\omega$-systems can be found in \cite[Theorem 1.6]{BTZ}, the same proof holds for $\bigoplus_{p\in P}\mathbb{F}_p$.
\begin{prop}[The universal characteristic factors are characteristic] \label{characteristic}
	If $c_1,c_2,c_3< \min_{p\in\mathcal{P}}p$ then $Z_{<3}(X)$ is characteristic for the average $$\lim_{N\rightarrow\infty}\mathbb{E}_{g\in \Phi_N} T_{c_1g}^n f_1 T_{c_2g}^n f_2 T_{c_3g}^n f_3$$ Namely, if either $f_1$, $f_2$ or $f_3$ are orthogonal to $Z_{<3}(X)$, then the $L^2$-limit is zero.
\end{prop}
We note that the existence of the $L^2$-limit is already known for all countable nilpotent groups (see a proof by Walsh in \cite{Walsh}). Our goal is to prove a formula for the limit, in the special case when the underlying system is a nilpotent system.\\
We recall from the previous section that every system of order $<3$ is an inverse limit of systems of the form $X=Z\times_\rho U$ where $Z$ is the Kronecker factor and $U$ is a Lie group. Every compact abelian Lie group is a product of a torus and a finite group. In Theorem \ref{recurrence:thm} we further assume that $3<\min_{p\in\mathcal{P}} p$ which implies that the order of the finite group is odd (this fact will be used later, in particular that $U^2 = U$). We prove this below.
\begin{prop} \label{2-divisible}
    Let $X=Z\times_\rho U$ be an ergodic $G$-system of order $<3$ and suppose that $3<\min_{p\in\mathcal{P}}p$. Then $U^2=U$.
\end{prop}
\begin{proof}
    Assume by contradiction that $U^2\lneqq U$. Then by the pontryagin dual there exists a non-trivial character $\chi\in \hat U$ with values in $C_2$. By Lemma \ref{dif:lem} we see that $\Delta_s \chi\circ\rho$ is a cocycle of type $<1$ for every $s\in Z$ and by Lemma \ref{type0} this implies that $\Delta_s\chi\circ\rho(g,x) = c_{s,\chi}(g)\cdot \Delta_g F_{s,\chi}(x)$. Arguing as in the proof of Theorem \ref{nilpotentstructure} above, we can assume that $c_{s,\chi}$ takes values in $C_2$. However, $c_{s,\chi}$ is a character of $G$ and by assumption $2\not\in \mathcal{P}$. We conclude that $c_{s,\chi}=1$ and $\Delta_s\chi\circ\rho$ is a coboundary. It follows that $\chi\circ\rho$ is of type $<1$ and that $Z\times_{\chi\circ\rho} C_2$ is a Kronecker system. This is a contradiction to the maximal property of the Kronecker factor.
\end{proof}
Let $3<\min_{p\in\mathcal{P}}p$ and $X=Z\times_\rho U$ be a system of order $<3$. By Theorem \ref{nilpotentstructure} we can approximate $X$ by nilpotent homogeneous spaces. We study these systems. Assume that $X=\mathcal{G}/\Gamma$ where $\mathcal{G}$ is the Host-Kra group of $X$ and $\Gamma$ a discrete subgroup. Let $\mG_2$ be the commutator subgroup of $G$, that is the smallest closed group generated by all the commutators. An easy calculation reveals that the commutator subgroup $\mG_2$ is isomorphic to $U$. In particular, from Proposition \ref{2-divisible} we conclude that any element in $\mG_2$ has a square root. In this case we have the following formula for the limit of the multiple ergodic averages.
\begin{thm}[The limit formula] \label{formula} Let $p>3$ and let $X=\mathcal{G}/\Gamma$ be a $G$-system of order $<3$ where $\mG$ is the Host-Kra group. If $\mG_2$ is a Lie group then for all $f_1,f_2,f_3\in L^\infty(X)$ we have that for $\mu_\mG$-almost all $x\in \mG$ the average $$\mathbb{E}_{g\in\Phi_N} f_1(T_gx)\cdot  f_2(T_{2g}x)\cdot f_3(T_{3g}x)$$ converges to $$\int_X \int_{\mathcal{G}_2} f_1(xy_1\Gamma)f_2(xy_1^2y_2\Gamma)f_3(xy_1^3y_2^3\Gamma)d\mu_{\mathcal{G}_2}(y_2)d\mu_X(y_1\Gamma).$$ With the abuse of notation that $f(x)=f(x\Gamma)$.
\end{thm}
The proof follows the argument of Lesigne \cite{Les}. Lesigne's argument relies on a result of Green \cite{Green} that in the case of connected simply connected nilsystems, the ergodicity is determined by the ergodicity of the Kronecker factor. We prove a counterpart of Green's Theorem in our special case (see Proposition \ref{Parry} below) from which we deduce the limit formula. 
\subsection{The system $(\tilde{\mG}/\tilde{\Gamma},S_{g,x})$}
Let $\mathcal{G}/\Gamma$ be a $2$-step nilpotent system and assume that $\mG_2$ is a compact abelian group. Let $\tilde{\mG} := \mG \times [\mG,\mG]$ and define multiplication on $\tilde{\mG}$ by $(x_1,x_2)\cdot (y_1,y_2) = (x_1y_1, [x_1,y_1]x_2y_2)$ where $[x,y]=x^{-1}y^{-1}xy$. For each $x\in \mG$, let $S_{g,x}$ denote the action of $G$ on $\tilde{\mG}$ by left multiplication with $(a_g[a_g,x],e)$ where $g\mapsto a_g$ denotes the action of $G$ on $\mG$. Finally let $\tilde{\Gamma}=\Gamma\times \{e\}\leq \tilde{\mG}$. It is easy to see that $\tilde{\mG}$ is a locally compact $2$-step nilpotent group and $\tilde{\Gamma}$ is a discrete co-compact subgroup. We equip $\tilde{\mG}$ with the product measure, and the quotient $\tilde{\mG}/\tilde{\Gamma}$ with the induced Haar measure.\\

Let $A_x = \{(xy_1\Gamma,xy_1^2y_2\Gamma,xy_1^3y_2^3\Gamma):(y_1,y_2)\in \tilde{\mG}\}$ be a subset of $\mG/\Gamma\times\mG/\Gamma\times\mG/\Gamma$ with the induced measure and $\sigma$-algebra. Then the map $I_x(y_1,y_2)=(xy_1,xy_1^2y_2,xy_1^3y_2^3)$ defines  an isomorphism between $(\tilde{\mG}/\tilde{\Gamma},S_{g,x})$ and $(A_x,(T_g\times T_{2g}\times T_{3g}))$. Let $\mu_{\mG}$ be the Haar measure on $\mG$, our goal is to prove that for $\mu_{\mG}$-almost every $x\in \mG$ the action of $S_{g,x}$ is ergodic.
\subsection{Proving ergodicity}
We first prove the ergodicity on the Kronecker factor.
\begin{prop}[Ergodicity on the Kronecker factor]
	The induced action of $S_{g,x}$ on $\tilde{\mG}/\tilde{\mG}_2\tilde{\Gamma}$ is ergodic for $\mu_\mG$-almost every $x\in\mG$.
\end{prop}
\begin{proof}
	Using Fourier analysis, the action of $S_{g,x}$ is ergodic if and only if every character $\sigma:\tilde{G}\rightarrow S^1$ which is trivial on $\tilde{\mG}_2\tilde{\Gamma}$ and satisfies that $\sigma((a_g[a_g,x],e))=1$ is trivial.
	
	Let $\sigma$ be such a character, we denote by $\sigma_1:\mG \rightarrow S^1$ and $\sigma_2:\mG_2\rightarrow S^1$ the coordinates of $\sigma$. Namely, $\sigma_1(x):=\sigma(x,e)$ and $\sigma_2(y)=\sigma(e,y)$. Direct computation shows that
	\begin{enumerate}
		\item {$\sigma(x,y) = \sigma_1(x)\cdot \sigma_2 (y)$.}
		\item {$\sigma_1(x_1x_2)=\sigma_1(x_1)\cdot \sigma_1(x_2)\cdot \sigma_2([x_1,x_2])$.}
		\item {$\sigma_2$ is a character.}
		\item {For every $x=(x_1,x_2),y=(y_1,y_2)\in \tilde{\mG}$ we have $[x,y] = ([x_1,y_1],[x_1,y_1]^2)$. In particular, $\sigma_1(x')=\sigma_2(x')^{-2}$.}
	\end{enumerate}
	Observe that by property $(4)$ we have that $\sigma((a_g[a_g,x],e))=1$ if and only if $\sigma_1(a_g)=\sigma_2([a_g,x])^2$.\\ Since the map $(g,y)\mapsto \sigma_2([a_g,y])$ is a bilinear map $G\times \mG\rightarrow S^1$ it gives rise to a homomorphism $\mG\rightarrow \hat G$. Let $\mL \leq \mG$ be the kernel of this homomorphism.\\
	
	Since $\hat{\tilde{\mG}}$ is countable, it is enough to show that the measure of $\mL$ for every non-trivial character is zero. Suppose by contradiction that $\mL$ is of positive measure. Then it is an open normal subgroup of $\mG$. Clearly, $\mL$ also contains the elements $a_g$ for all $g\in G$. Let $\Gamma_{\mL} = \mL\cap \Gamma$, then $\mL/\Gamma_{\mL}$ can be identified with a closed and open subset of $\mG/\Gamma$ which is also $G$-invariant. Ergodicity of $\mG/\Gamma$ implies that the systems $\mL/\Gamma_{\mL}$ and $\mG/\Gamma$ are isomorphic and the identity map $x\Gamma_{\mL}\mapsto x\Gamma$ is the isomorphism. In particular, it follows that the action of $G$ on $\mL/\Gamma_{\mL}$ is ergodic.\\
	\textit{Claim 1.} We prove that $\sigma_2(x)=1$ for all $x\in [\Gamma_{\mL},\mG]$ and  $x\in \mL_2$.\\
	From the construction we see that $\sigma_2([a_g,x])=1$ for all $g\in G$ and for all $x\in \mL$. Let $\gamma\in \Gamma_{\mL}$ and look at $x\mapsto \sigma_2([\gamma,x])$. Since $\gamma\in \Gamma$ this map is trivial for $x\in \Gamma$, since $\gamma\in \mL$ this map is also trivial for $x=a_g$ for every $g\in G$ hence by ergodicity this map is trivial for all $x\in\mG$. Since $\sigma_2$ is a character this proves the first claim. Now fix $y\in \mL$. We just proved that $\sigma_2([y,x])=1$ for all $x\in \Gamma_{\mL}$, it follows by the construction of $\mL$ that the same holds for $x=a_g$ for all $g\in G$. Ergodicity implies that $\sigma_2$ is trivial on $\mL_2$ as required. $\square$\\
	\textit{Claim 2.} $\sigma_2$ is trivial on $\mG_2$ .\\
	Consider the map $\psi:y\mapsto \frac{\sigma_1(y)}{\sigma_2([y,x_0])}$ for all $y\in \mL$. By property (3) above and claim 1 this map is $G$-invariant and a homomorphism of $\mL$. Moreover, since $\sigma$ is invariant to right multiplication by an element in $\tilde{\Gamma}$, it follows from claim 1 that $\psi$ is also invariant to right multiplication by $\Gamma_{\mL}$. The ergodicity of the $G$ action on $\mL/\Gamma_{\mL}$ implies that $\psi$ is trivial. Now, since $\mG_2\subseteq \mL$, we conclude that for every $t\in \mG_2$, $\sigma_1(t)=\sigma_2([t,x_0])^2=1$. From property $(4)$ above this implies that $\sigma_2(t^2)=1$. Since $p>3$, $\mG_2^2 = \mG_2$ and the claim follows. $\square$\\
	To finish the proof of the proposition, we prove that $\sigma\equiv 1$. Just as in claim $2$, we see that $\sigma_1:\mG\rightarrow S^1$ is a homomorphism that is invariant to $\Gamma$ and to the action of $G$. Hence, it is a constant. Since $\sigma = \sigma_1$ we conclude that $\sigma$ is trivial.
\end{proof}
\begin{prop}[Ergodicity on the Kronecker factor implies ergodicity] \label{Parry}
 Let $x$ be such that the induced action of $S_{g,x}$ on $\tilde{\mG}/\tilde{\mG}_2\tilde{\Gamma}$ is ergodic. Then the original action of $S_{g,x}$ on $\tilde{\mG}/\tilde{\Gamma}$ is ergodic.
\end{prop}
\begin{proof}
We proceed as in the proof of Parry \cite{Parryc}. Fix $x$ as in the proposition and let $f:\tilde{\mG}/\tilde{\Gamma}\rightarrow S^1$ be an invariant function. The group $\tilde{\mG}_2$ is a compact abelian group which acts on $L^2(\tilde{\mG}/\Gamma)$. We deduce, using the Peter-Weyl decomposition theorem, that there is an orthogonal decomposition of $f$ as a sum of eigenfunctions $f=\sum_{\lambda}f_\lambda$ where $\lambda$ is a character of $\tilde{\mG}_2$. Since $f$ is invariant, each $f_\lambda$ is an eigenfunction of the $G$-action.
	Let $n\in\tilde{\mG}$. For every $\lambda$ and every $g\in G$ we have, $$f_\lambda(ngx) = f_\lambda(gn[n^{-1},g^{-1}]x) = \lambda([n^{-1},g^{-1}]) f_\lambda(gnx)= \lambda([n^{-1},g^{-1}])\cdot c_g f_\lambda(nx).$$
	 Hence, $\Delta_n f_\lambda$ is an eigenfunction with respect to the action of $G$. Direct computation shows that $\Delta_u \Delta_n f_\lambda=1$ for all $u\in \tilde{\mG}_2$ and so $\Delta_n f$ can be identified with an eigenfunction with respect the induced action of $S_{g,x}$ on the quotient $\tilde{\mG}/\tilde{\mG}_2\tilde{\Gamma}$. We assume that this factor is ergodic and so by Lemma \ref{sep:lem} the group of eigenfunctions modulo constants is discrete. We conclude that the group   $$\mL_\lambda:=\{n\in\mG: \Delta_n f_\lambda \text{ is a constant}\}$$ is an open subgroup of $\mG$. Moreover, $l\mapsto \Delta_l f_\lambda$ is a homomorphism from $\mL_\lambda$ to $S^1$ and so it is invariant under $[\mL_\lambda,\mL_\lambda]\leq \tilde{\mG}_2$. It is left to show that for every character $\lambda$, $[\mL_\lambda,\mL_\lambda] = \tilde{\mG}_2$.  Fix $\lambda$ and let $\mL=\mL_\lambda$. Let $\mL'=\mL\cap \mG\times\{e\}$ By property $(4)$ it is left to prove that $\mL'_2=\mG_2$. For convenience we denote $b_g=a_g[a_g,x]$. It is easy to see that $b_g\in \mL'$. Now we use the fact that $\mG$ is the Host-Kra group of the underlying space $X$. Write $X=Z\times_\rho U$ where $Z$ is the Kronecker factor and $U$ is a Lie group. Since $\mL$ is open and contains $b_g$ for all $g\in G$, the image of the projection $p:\mL\rightarrow Z$ is an invariant open subgroup and therefore the projection is onto. Fix a corss-section $s\mapsto \overline{s}=S_{s,f_s}$ where $\Delta_s \sigma = c_s\cdot \Delta f_s$ for some $c_s:G\rightarrow U$. Then $[\overline{s},b_g] =[\overline{s},a_g]= S_{1,c_s(g)}$ which we identify as an element of $U$. Suppose by contradiction that the closed group generated by these $c_s(g)$ for all $s\in Z$ and $g\in G$ is a proper subgroup of $U$. Then there exists a character $\chi\in\hat U$ such that $\chi\circ\rho$ is of type $<1$. The maximal property of the Kronecker factor provides a contradiction. This completes the proof.
\end{proof}

We conclude that the system $(\tilde{\mG}/\tilde{\Gamma},S_{g,x})$ is ergodic for $\mu_\mG$-almost every $x\in \mG$. Since it is ergodic it is automatically uniquely ergodic (see \cite[Theorem 5]{Parrya} and \cite[Section 2, Lemma 1]{Parryb}). Therefore, by the mean ergodic theorem we have that $$\lim_{N\rightarrow\infty} \mathbb{E}_{g\in\Phi_N} F(a_gxy_1,a_g^2xy_1^2y_2,a_g^3xy_1^3y_2^3) = \int_{\mG/\Gamma}\int_{\mG_2} F(xy_1,xy_1^2y_2,xy_1^3y_2^3) d\mu_{\mG_2}(y_2) d\mu_{\mG/\Gamma}(y_1)$$
 for every continuous function $F$  on $(\mathcal{G}/\Gamma)^3$ and for all $y_1,y_2\in \mG$.
Set $y_1=y_2=e$ and $F(x_1,x_2,x_3)=f(x_1)\cdot f(x_2)\cdot f(x_3)$. We conclude that Theorem \ref{formula} holds for continuous functions. The general case follows by approximating bounded functions by continuous functions and taking limits.
\subsection{Concluding the proof for the Khintchine-recurrence}
We prove Theorem \ref{recurrence:thm} following an argument of Frantzikinakis \cite{Fran}. Let $X$ be any system, by Proposition \ref{characteristic} the limit of  $\mathbb{E}_{g\in \Phi_N} T_gf_1 T_{2g}f_2 T_{3g}f_3$ in $L^2$ is equal to the limit of $\mathbb{E}_{g\in \Phi_N} T_g\tilde{f}_1 T_{2g}\tilde{f}_2 T_{3g}\tilde{f}_3$ if the latter exists where $\tilde{f}_i$ denotes the projection of $f_i$ to the factor $Z_{<3}(X)$ for each $1\leq i \leq 3$. Using Theorem \ref{nilpotentstructure} we can assume by approximation argument that $Z_{<3}(X)$ is an extension of the Kronecker by a Lie group and that $Z_{<3}(X)\cong \mG/\Gamma$ where $\mG_2$ is a Lie group. Let $A$ be a set of positive measure, set $f=f_1=f_2=f_3=1_A$ and let $\varepsilon>0$. Suppose by contradiction that the set $\{g\in G : \mu(A\cap T_g A\cap T_{2g}A\cap T_{3g}A)\geq \mu(A)^4-\varepsilon\}$ fails to be syndetic. Then, its complement contains a translation of every finite set. In particular, it must contain a F{\o }lner sequence $\Phi_N$. Therefore, for every $g\in \bigcup_{N=1}^{\infty}{\Phi_N}$ we have that $$\int_X f\cdot T_g f\cdot T_{2g} f\cdot T_{3g} f d\mu < \mu(A)^4-\varepsilon.$$
\textbf{Claim:} Let $Z$ be the Kronecker factor of $X$. Let $a_g\in Z$ Denote the element corresponding to the action of $G$ on $Z$ and $\eta:Z\rightarrow\mathbb{R}^+$ be any continuous function. Then, the average $$\mathbb{E}_{g\in \Phi_N} \eta(a_g) T_gf_1 T_{2g}f_2 T_{3g}f_3$$ converges to zero in $L^2$ if $E(f_i|Z_{<3})=0$ for any $1\leq i \leq 3$.
\begin{proof} By approximating $\eta$ by characters it is enough to assume that $\eta$ is a character of the Kronecker factor. Therefore $\eta(a_g)=\Delta_g \chi$ for some character $\chi\in\hat Z$. Let $f_1':=f_1\cdot \chi$ and then apply Proposition \ref{characteristic} for the average associated with $f_1',f_2,f_3$. Since $\chi$ is measurable with respect to Kronecker $E(f_1':Z_{<3})=0\iff E(f_1|Z_{<3})=0$ which completes the proof. 
\end{proof}
We can therefore apply Lemma \ref{formula} to the twisted average $$\mathbb{E}_{g\in \Phi_N} \eta(a_g) T_gf_1 T_{2g}f_2 T_{3g}f_3$$ and conclude that it converges to $$\int_X\int_{G_2}\eta(y_1)\tilde{f}_1(xy_1)\tilde{f}_2(xy_1^2y_2)\tilde{f}_3 (xy_1^3y_2^3)dm_{G_2}(y_2)dm_X(y_1\Gamma)$$ By taking $f_1=f_2=f_3=1$ we conclude that $$\mathbb{E}_{g\in \Phi_N}\eta(a_g)=1$$ we now consider the average $$\mathbb{E}_{g\in\Phi_N}\eta(a_g) \int_X f_0\cdot T_g f_1 \cdot T_{2g}f_2 \cdot T_{3g}f_3 dm_X$$ which by the same argument converges to
$$\int_X\int_X\int_{G_2}\eta(y_1)f_0(x)\tilde{f}_1(xy_1)\tilde{f}_2(xy_1^2y_2)\tilde{f}_3 (xy_1^3y_2^3)dm_{G_2}(y_2)dm_X(y_1\Gamma)dm_X(x\Gamma).$$
Since $\eta$ is arbitrary, we can approximate the indicator functions $1_{B(\mG_2,\delta)}$ where $B(\mG_2,\delta)$ denote the ball of radius $\delta$ around all elements of $\mG_2$. Since translations are continuous in $L^2$, taking a limit as $\delta\rightarrow 0$ the limit will be arbitrarily close to
$$\int_X \int_{\mG_2\times \mG_2} \tilde{f}(x)\tilde{f}(xy_1)\tilde{f}(xy_1^2y_3)\tilde{f}(xy_1^3y_2^3) dm_{\mG_2\times \mG_2}(y_1,y_2) dm_X(x\Gamma).$$
We integrate everything to get that this equals to
$$\int_X \int_{\mG_2\times \mG_2\times \mG_2} \tilde{f}(gx)\tilde{f}(gxy_1)\tilde{f}(gxy_1^2y_2)\tilde{f}(gxy_1^3y_2^3)dm_{\mG_2\times \mG_2\times \mG_2}(g,x,y_2) dm_X (y_1\Gamma).$$ Since the set $\{(g,gy_1,gy_1^2y_2,gy_1^3y_2^3):g,y_1,y_2\in \mG_2\}$ equals to the set $\{(h_1,h_2,h_3,h_4)\in \mG_2^4:h_1h_3^3=h_4h_2^3\}$, we can write the above integral as
$$\int_X \int_{\mG_2}\int_{h_1h_3^3=h}\tilde{f}(h_1x)\tilde{f}(h_3x)d\lambda(h_1,h_3)^2dm_{\mG_2}(h)dm_X(x\Gamma).$$ By the Cauchy-Schwartz and the triangle inequality this is greater or equal to $$\int_X \left(\int_{\mG_2} \tilde{f}(hx)dm_{\mG_2}(h)\right)^4dm_X(x\Gamma)=\left(\int_X \tilde{f}(x)dm_X(x\Gamma)\right)^4=\mu(A)^4$$
We conclude that $$\mathbb{E}_{g\in\Phi_N}\eta(a_g)\mu(A\cap T_gA\cap T_{2g}A\cap T_{3g}A)>\mu(A)^4-\varepsilon/2$$ for some $\eta(a_g)>0$ with $\mathbb{E}_{g\in \Phi_N}a_g=1$. This contradicts the inequality $\mu(A\cap T_gA\cap T_{2g}A\cap T_{3g}A)<\mu(A)^4-\varepsilon$ for all $g\in \bigcup_{N\in\mathbb{N}}\Phi_N$ as required. $\square$
\appendix
\section{Topological groups and measurable homomorphisms}
In this section we survey some results about topological groups.
\subsection{Homomorphisms of polish groups}

\begin{defn}
	We say that a topological group $G$ is a polish group if it is separable and completely metrizable. If $G$ is compact this is equivalent to the existence of some invariant metric on $G$ (i.e. a metric such that $d(x,y)=d(gx,gy)$).
\end{defn}

Every topological group is a measurable space with respect to the Borel $\sigma$-algebra. It is well known that every locally compact abelian group $G$ admits a unique (up to scalar multiplication) invariant Borel measure $\mu$. This measure is inner and outer regular and it assigns finite measure for compact subsets. In particular in the case where $G$ is compact we can normalize so that $\mu(G)=1$. The existence of such measures leads to many fruitful corollaries.
\begin{prop} [A.Weil]  \cite[Lemma 2.3]{Ch}\label{A.Weil}
	Let $G$ be a locally compact polish group and let $A\subseteq G$ be a measurable subset of positive measure. Then $A\cdot A^{-1}$ contains an open neighborhood of the identity.
\end{prop}
This implies the following useful proposition.
\begin{prop}\label{countableindex}
	Let $G$ be a locally compact polish abelian group and let $H$ be a Borel subgroup of at most countable index. Then $H$ is open.
\end{prop}
\begin{proof}
	Let $\mu$ be the Haar measure on $G$. Since $H$ has countable index there exists $g_1,g_2,...$ such that $G=\bigsqcup_{i=1}^\infty g_i H$. In particular, it follows that $$0<\mu(G)= \sum_{i=1}^\infty \mu(g_i H).$$ Since the measure is invariant the right hand side is an infinite sum of $\mu(H)$. This is only possible if the measure of $H$ is positive (Note that if $G$ is compact this also implies that the sum is finite). Now, by Proposition \ref{A.Weil} we have that $H-H$ contains an open neighborhood $U$ of the identity. Since $H-H\subseteq H$ we have that $H=\bigcup_{h\in H} hU$ and so is open.
\end{proof}
We deduce the following result.
\begin{cor}\label{openker} Let $G$ be a locally compact abelian polish group and let $L$ be a locally compact abelian group of at most countable cardinality. Then any measurable homomorphism $\varphi:G\rightarrow L$ factors through an open subgroup of $G$.
\end{cor}
\begin{proof}
	The kernel of $\varphi$ is a Borel subgroup of at most countable index. Therefore the claim follows from the previous proposition.
\end{proof}

Another important corollary of Proposition \ref{A.Weil} is the following automatic continuity lemma.

\begin{lem} [Automatic continuity of measurable homomorphisms] \cite[Theorem 2.2]{Ch}\label{AC:lem}
	Any measurable homomorphism from a locally compact polish group into a polish group is continuous.
\end{lem} 
The following result is a version of the open mapping theorem in polish groups \cite[Chapter 1]{BK}.
\begin{thm}\label{openmappingthm}
Let $\mathcal{G}$ and $\mathcal{H}$ be Polish groups and let $p: \mathcal{G} \rightarrow \mathcal{H}$ be a
group homomorphism that is continuous and onto. Then $p$ is an open map.
\end{thm}
It follows from this theorem that  $\mathcal{G}/\ker(p)\cong \mathcal{H}$ as topological groups. In particular we obtain the following result.
\begin{cor}\label{compact}
   Let $\mathcal{H}$ be a closed normal subgroup of the Polish group $\mathcal{G}$. If $\mathcal{H}$ and $\mathcal{G}/\mathcal{H}$ are locally compact, then $\mathcal{G}$ is locally compact. If $\mathcal{H}$ and $\mathcal{G}/\mathcal{H}$
are compact, then $\mathcal{G}$ is compact.
\end{cor}
\subsection{Totally disconnected groups}
\begin{defn} \cite[Exercise E8.6]{HM} Let $X$ be a locally compact Hausdorff space. Then the following are equivalent
	\begin{itemize}
		\item{Every connected component in $X$ is a singleton.}
		\item {$X$ has a basis consisting of open closed sets.}
	\end{itemize}
	We say that $X$ is totally disconnected if one of the above is satisfied.
\end{defn}

In this section we will be interested in compact (Hausdorff) totally disconnected groups. These groups are also called pro-finite groups, in fact one can show that every such group is an inverse limit of finite groups (see \cite[Proposition 1.1.3]{profinite}).

\begin{prop}\label{opensubgroup}
	Let $G$ be a compact Hausdorff totally disconnected group. Let $1\in U\subseteq G$ be an open neighborhood of the identity, then $U$ contains an open subgroup of $G$.
\end{prop}
The proof of this Proposition can be found in \cite[Proposition 1.1.3]{profinite}. As a corollary we have the following result.
\begin{cor}[The dual of totally disconnected group is a torsion group]  \label{chartdg} Let $G$ be a compact abelian totally disconnected group and let $\chi:G\rightarrow S^1$ be a continuous character. Then the image of $\chi$ is finite.
\end{cor}
\begin{proof}
	Choose an open neighborhood of the identity $U$ in $S^1$ that contains no non-trivial subgroups. Then $\chi^{-1}(U)$ is an open neighborhood of $G$. Now, let $H$ be an open subgroup such that $H\subseteq \chi^{-1}(U)$. It follows that $\chi(H)$ is trivial and so $\chi$ factors through $G/H$ which is finite.
\end{proof}
We note that the other direction also holds, but we do not use this fact here.\\

Since compact totally disconnected groups are pro-finite groups, some of the theory of finite groups can be generalized to these groups. For example, we have the following decomposition to $p$-components.
\begin{prop}[Sylow Theorem]\cite[Corollary 8.8]{HM} \label{Sylow}
	A compact abelian group is totally disconnected if and only if it is a direct
	product of $p$-groups.
\end{prop}
We also need the following structure theorem for torsion groups (of bounded torsion).
\begin{thm}[Structure theorem for abelian groups of bounded torsion]\label{torsion}\cite[Chapter 5, Theorem 18]{M}
	Let $G$ be a compact abelian group and suppose that there exists some $n\in\mathbb{N}$ such that $g^n=1_G$ for every $g\in G$. Then, $G$ is topologically and algebraically isomorphic to $\prod_{i=1}^\infty C_{m_i}$ where for every $i$, $m_i$ is an integer which divides $n$.
\end{thm}
One way to generate totally disconnected groups is to begin with an arbitrary compact abelian group and quotient it out by its connected component.
\begin{lem}\label{connectedcomponent}
	Let $G$ be a compact abelian group and let $G_0$ be the connected component of the identity. Since the multiplication and the inversion maps are continuous one has that $G_0$ is a subgroup of $G$. We have:
	\begin{itemize}
		\item{$G_0$ has no non-trivial open subgroups}
		\item {Every open subgroup of $G$ contains $G_0$}
		\item {$G/G_0$ equipped with the quotient topology is totally disconnected compact group}
	\end{itemize}
\end{lem}
As a corollary we have the following result.
\begin{prop}[Quotient of profinite group is a profinite group] \label{Quotient} 
	Let $G$ be a profinite group, let $N$ be a subgroup of $G$. Then $G/N$ with the induced topology is a profinite group.
\end{prop}
\begin{proof}
	Let $C$ denote the connected component of the identity in $G/N$. Let $x\in C$ and let $\pi:G\rightarrow G/N$ be the projection map. Then $\pi^{-1}(\{x\})$ is a closed subset of $G$. If by contradiction $x\not = 1$ then by Proposition \ref{opensubgroup} we have that the complement contains an open subgroup $V$. Quotient homomorphisms are open and so by Lemma \ref{connectedcomponent}, $\pi(V)$ contains the connected component of $G/N$ which is absurd.
\end{proof}
We also need the following important fact that connected groups are divisible \cite[Corollary 8.5]{HM}.
\begin{lem}\label{divisible}
    Let $G$ be a compact abelian connected group. Then for every $g\in G$ and $n\in\mathbb{N}$, there exists $h\in G$ such that $h^n=g$.
\end{lem}
\subsection{Lie groups}
\begin{defn}
	A topological group $G$ is said to be a Lie group if, as a topological space, it is a finite dimensional, differentiable manifold over $\mathbb{R}$ and the multiplication and inversion maps are smooth.
\end{defn}
A compact abelian group is a Lie group if and only if its Pontryagin dual is finitely generated. The structure theorem for finitely generated abelian group then implies the following result.
\begin{thm}[Structure Theorem for compact abelian Lie groups]\cite[Theorem 5.2]{S} \label{structureLieGroups} A compact abelian group $G$ is a Lie group if and only if there exists $n\in\mathbb{N}$ such that $G\cong (S^1)^n\times C_k$ where $C_k$ is some finite group with discrete topology.
\end{thm} \label{approxLieGroups}
Fortunately, a classical result of Gleason-Yamabe asserts, in particular, that all compact abelian groups can be approximated by compact abelian Lie groups
\begin{thm}\cite[Corollary 8.18]{HM} \label{GY:thm}
	Let $G$ be a compact abelian group and let $U$ be a neighborhood of the identity in $G$. Then $U$ contains a subgroup $N$ such that $G/N$ is a Lie group.
\end{thm}

It follows from the above (see also \cite[Lemma 2.2]{Iwasawa}) that any compact connected nilpotent group is abelian.
\begin{prop}\label{connilabel:prop}
	If $G$ is a compact metric connected $k$-step nilpotent group. Then $G$ is abelian.
\end{prop}
\section{Some results about phase polynomials}
In this section $G=\bigoplus_{p\in P}\mathbb{F}_p$ for some multiset of primes $P$.
\begin{prop}  [Values of phase polynomial cocycles]\label{PPC} Let $X$ be an ergodic $G$-system. Let $d\geq 0$ and $q:G\times X\rightarrow S^1$ be a phase polynomial of degree $<d$ that is also a cocycle. Then, for every $g\in G$, $q(g,\cdot)$ takes values in $C_m$ where $m$ is the order of $g$ to the power of $d$.
\end{prop} 
\begin{proof}
	We prove the proposition by induction on $d$. If $d=0$ then $q\equiv 1$ and the claim is trivial. Fix $d\geq 1$ and assume inductively that the claim holds for smaller values of $d$. Let $q:G\times X\rightarrow S^1$ be a phase polynomial of degree $<d$ and fix $g\in G$ of order $n$. The cocycle identity implies that $$1=q(ng,x)=\prod_{k=0}^{n-1}q(g,T_{kg}x).$$
	Since $q(g,T_{kg}x)=q(g,x)\cdot \Delta_{kg} q(g,x)$ we have that $q(g,x)^n \cdot \prod_{k=0}^{n-1}\Delta_{kg} q(g,x)=1$. By the induction hypothesis, $\prod_{k=0}^{n-1}\Delta_{kg} q(g,x)$ is in $C_{n^{d-1}}$ and it follows that $q(g,x)\in C_{n^d}$, as required.
\end{proof}

We need the following version of \cite[Lemma B.5 (i)]{Berg& tao & ziegler}.
\begin{lem}  [Vertical derivatives of phase polynomials are phase polynomials of smaller degree]\label{vdif:lem}
	Let $X$ be an ergodic $G$-system. Let $U$ be a compact abelian group acting freely on $X$ by automorphisms. Let $P:X\rightarrow S^1$ be a phase polynomial of degree $<d$ for some $d\geq 1$. Then $\Delta_u P$ is a phase polynomial of degree $<d-1$ for every $u\in U$.
\end{lem}
\begin{proof} 
    We prove the lemma by induction on $d$. If $d=1$, ergodicity implies that  $P$ is a constant and so $\Delta_u P =1$, as required. Let $d\geq 2$ and assume inductively that the claim is true for $d-1$. Given a phase polynomial $P:X\rightarrow S^1$ of degree $<d$, we have that $\Delta P : G\times X\rightarrow S^1$ is a phase polynomial of degree $<d-1$. By the induction hypothesis, we conclude that $\Delta_u \Delta P$ is a phase polynomial of degree $<d-2$. As the action of $U$ commutes with the action of $G$, we have that $\Delta \Delta_u P$ is a phase polynomial of degree $<d-2$. It follows that $\Delta_u P$ is a phase polynomial of degree $<d-1$ as desired.
\end{proof}
Proposition \ref{PPC} and Lemma \ref{vdif:lem} implies the following result.
\begin{cor} \label{ker:cor} 
	Let $X$ be an ergodic $G$-system and $U$ be a compact abelian group acting freely on $X$ and commuting with the action of $G$. Suppose that there exists a measurable map $u\mapsto f_u$ from $U$ to $P_{<d}(X,S^1)$ which satisfies the cocycle identity (i.e. $f_{uv}=f_u V_u f_v$) for all $u,v\in U$. Then, there exists an open subgroup $V$ of $U$ such that $f_v\in P_{<1}(X,S^1)$ for all $v\in V$.
\end{cor}
\begin{proof}
	We prove the claim by induction on $d$. If $d=1$ we can take $V=U$ and the claim follows. Let $d>1$ and assume the claim is true for smaller values of $d$. Let $u\mapsto f_u$ be a map from $U$ to $P_{<d}(X,S^1)$. The cocycle identity implies that $f_{uv}=f_u f_v \cdot \Delta_u f_v$. Applying Lemma \ref{vdif:lem} we have that $\Delta_u f_v\in P_{<d-1}(X,S^1)$ and so after quotienting out $P_{<d-1}(X,S^1)$ we have that the map $U\rightarrow P_{<d}(X,S^1)/P_{<d-1}(X,S^1)$ sending $u$ to the equivalent class of $f_u$ is a homomorphism. Since $d>1$, Lemma \ref{sep:lem} (Separation Lemma) implies that $P_{<d-1}(X,S^1)$ has at most countable index in $P_{<d}(X,S^1)$. Corollary \ref{openker} implies that the kernel, $U'$ is an open subgroup. \\
	We conclude that $f_{u'}\in P_{<d-1}(X,S^1)$ for all $u'\in U'$, and so the induction hypothesis implies that there exists an open subgroup $V$ of $U'$ such that $f_v\in P_{<1}(X,S^1)$ for all $v\in V$. As $V$ is open in $U'$ and $U'$ is open in $U$ we have that $V$ is open in $U$.
\end{proof}
We also need the following lemma from \cite[Lemma B.5 (iii)]{Berg& tao & ziegler}.
\begin{lem} [Composition of polynomials is again polynomial] \label{B.5} Let $U$ and $V$ be two abelian groups and $X=Y\times_{\rho} U$ be an ergodic extension of a $G$-system $Y$ by a phase polynomial cocycle $\rho:G\times Y\rightarrow U$ of degree $<k$ for some $k\geq 1$. Suppose that $p:X\rightarrow V$ is a phase polynomial of degree $<d$, and $v_1:X\rightarrow U$ , $v_2:X\rightarrow U$,...,$v_j:X\rightarrow U$ are phase polynomials of degree $<d_1,<d_2,...,<d_j$ and $s:X\rightarrow U$ is a phase polynomial of degree $<d'$. Then the function $P:X\rightarrow V$ given by the formula
	$$p(y,u)=(\Delta_{v_1(y,u)}...\Delta_{v_j(y,u)}p)(y,s(y,u))$$
	is a phase polynomial of degree 		 $<O_{d,j,d1,...,dj,d',k}(1)$.
\end{lem}
The next proposition follows from an argument of Bergelson Tao and Ziegler, see \cite[Lemma 2.1]{BTZ}.
\begin{prop} [Phase polynomial are invariant under connected components] \label{pinv:prop}  Let $X$ be an ergodic $G$-system of order $<k$ and $U$ be a compact abelian connected group acting freely on $X$ (not necessarily commuting with the $G$-action). Let $P:G\times X\rightarrow S^1$ be a phase polynomial of degree $<d$ such that for every $g\in G$ there exists $M_g\in\mathbb{N}$ such that $P(g,\cdot)$ takes at most $M_g$ values. (e.g. if $P$ is a phase polynomial cocycle). Then, $P$ is invariant under the action of $U$.
\end{prop}
\begin{proof}
	Fix $g\in G$ and consider the map $u\mapsto \Delta_u P(g,\cdot)$. Since $P(g,\cdot)$ is a measurable map $X\rightarrow S^1$, we have that $\Delta_u P$ converges in measure to the constant $1$ as $u$ converges to the identity in $U$. Since convergence in measure implies convergence in $L^2$, we can use Lemma \ref{sep:lem} to conclude that $\Delta_u P(g,\cdot)$ must be almost everywhere constant for $u$ close to the identity. From the cocycle identity, we have that the subset $U_g' = \{u\in U : \Delta_u P(g,\cdot) \text{ is a constant}\}$ is an open subgroup of $U$. As $U$ is connected, we conclude that $U_g'=U$ for every $g\in G$. We conclude that for every $g\in G$, there exists a character $\chi_g:U\rightarrow S^1$ such that  $\Delta_u P(g,\cdot) = \chi_g(u)$ for every $u\in U$. Since $U$ is connected and $\chi_g$ is continuous we have that the image of $\chi_g$ is either trivial or is $S^1$. But, the latter contradicts the assumption that $P(g,\cdot)$ takes finitely many values. It follows that $\Delta_u P(g,\cdot)=1$ for every $u\in U$ and $g\in G$. In other words, $P$ is invariant under the action of $U$, as required.
\end{proof}
\begin{rem}
In some cases the group $S^1$ in the proposition can be replaced by any compact abelian group using Pontryagin duality. For instance if $P:G\times X\rightarrow V$ is a phase polynomial cocycle for some compact abelian group $V$, then for every $\chi\in\hat V$ we have that $\chi\circ P:G\times X\rightarrow S^1$ is a phase polynomial cocycle of the same degree. By Proposition \ref{PPC} and Proposition \ref{pinv:prop} we have that $\chi(\Delta_uP)=1$ for every $u\in U$. As the characters separates points this would imply that $\Delta_u P = 1$, hence $P$ is invariant with respect to the action of $U$.
\end{rem}
In the rest of this section we work with a totally disconnected system $X$. We show that in this case any phase polynomial into $S^1$ takes values in a coset of a finite cyclic subgroup.
\begin{prop} [Phase polynomials on totally disconnected systems take finitely many values]\label{TDPV:prop} 
	Let $X$ be an ergodic totally disconnected $G$-system of order $<k$ (see Definition \ref{TD:def}). Let $P:X\rightarrow S^1$ be a phase polynomial of degree $<d$. Then up to a constant multiple, $P$ takes values in a finite subgroup of $S^1$.
\end{prop}
\begin{proof}
	We induct on $k$. If $k=1$ then $X$ is trivial. In particular, every function on $X$ is a constant and the claim follows. Let $k\geq 2$ and assume the claim has already been proven for $k-1$. Let $X$ be as in the Proposition, then by Proposition \ref{abelext:prop} we can write $X=Z_{<k-1}(X)\times_\rho U$.\\
	Consider the map $u\mapsto \Delta_u P$. Clearly, $u\mapsto \Delta_u P$ satisfies the cocycle identity and so Corollary \ref{ker:cor} implies that there exists an open subgroup $V$ s.t. $\Delta_u P\in P_{<1}(X,S^1)$ for every $u\in V$. Ergodicity implies that $\Delta_u P$ is a constant in $S^1$. The induced map $V\rightarrow S^1$ sending $v$ to the constant $\Delta_v P$ is a homomorphism. As $V$ is totally disconnected, the kernel $U'$ is an open subgroup which satisfies that $\Delta_u P =1$ for every $u\in U'$.\\
	Now, let $u\in U$. As $U/U'$ is a finite group, there exists $m\in\mathbb{N}$ such that $u^m\in U'$. We follow the argument in the proof of Corollary \ref{ker:cor}, but instead of passing to an open subgroup, we take a power. As in Corollary \ref{ker:cor}, we have that $U\rightarrow P_{<d-1}(X,S^1)/P_{<d-2}(X,S^1)$ sending $u$ to the equivalent class of $\Delta_u P$ is a homomorphism. As $\Delta_{u^m}P=1$ for all $u\in U$, we conclude that $\Delta_u P^m$ is a phase polynomial of degree $<d-2$. Iterating this process, we conclude that $\Delta_u P^{m^{d-1}}=1$ for all $u\in U$. In other words, $P^{m^{d-1}}$ is invariant under $U$. Viewing $P^{m^{d-1}}$ as a phase polynomial of degree $<d$ on $Z_{<k-1}(X)$ and applying the induction hypothesis, we see that, up to constant multiplication, $P^{m^{d-1}}$ takes values in some finite subgroup $H$ of $S^1$. Let $c$ be an $m^{d-1}$-th root of this constant, we have that $P/c$ takes values in the finite group $H$, as desired.
	
\end{proof}
As a corollary we have the following result.
\begin{thm} [Phase polynomials of degree $<d$ on totally disconnected systems take $<O_d(1)$ values on finitely many primes]\label{TDPV:thm}
	Let $X$ be an ergodic totally disconnected $G$-system of order $<k$ (see Definition \ref{TD:def}), let $F:X\rightarrow S^1$ be a phase polynomial of degree $<d$. Then up to constant multiplication, $F$ takes values in the group $C_m$ where $m=p_1^{l_1}\cdot...\cdot p_n^{l_n}$ for some $n\in\mathbb{N}$, distinct primes $p_1,...,p_n$ and $l_1,...,l_n=O_{d}(1)$.
\end{thm}
\begin{proof}
	By Proposition \ref{TDPV:prop} we have that up to constant multiplication $F$ takes values in $C_\alpha$ for some finite $\alpha\in\mathbb{N}$. The derivative $q:=\Delta F$ is a phase polynomial of degree $<d-1$ which is also a cocycle and it takes values in $C_\alpha$. Let $n\in\mathbb{N}$ be such that $\alpha<p_n$ where $p_n$ is the $n$-th prime and write $G=G_n\oplus G'$ where $G_n=\bigoplus_{p\in P,p<p_n}\mathbb{F}_p$ and $G'$ is its complement. From Proposition \ref{PPC} and the fact that $\alpha<p_n$ we conclude that $q(g,\cdot)=1$ for all $g\in G'$.  Let $m=\prod_{i=1}^n p_i^{d-1}$. Proposition \ref{PPC} implies that $q^m=1$ and so, by ergodicity $F^m$ is constant. Let $c$ be an $m$-th root for $\overline {F^m}$. We conclude that $c\cdot F$ takes values in $C_m$ as required.
\end{proof}
In the next theorem we will generalize the results above in the case where $p\geq k$.
\begin{thm} \label{HTDPV:thm} 
	Let $X$ be an ergodic $G$-system, let $F:X\rightarrow S^1$ be a phase polynomial of degree $<k$, and let $p$ be a prime such that $p\geq k$. Suppose now that $F$ takes values in $C_{p^m}$ for some $m\in\mathbb{N}$, then up to constant multiplication $F$ takes values in $C_p$
\end{thm}
\begin{proof}
The derivative $\Delta F:G\times X\rightarrow C_{p^m}$, is a phase polynomial of degree $<k$. Write $G=G_p\oplus G'$ where $G_p$ is the $p$-torsion subgroup of $G$. By Proposition \ref{PPC} it follows that  for every $g\in G'$, $\Delta_g F$ takes values in $C_n$ for some $n$ co-prime to $p$. As $\Delta F$ takes values in $C_{p^m}$ we conclude that $\Delta_g F=1$ for all $g\in G'$. Now, we claim that $\Delta_g F^p = 1$ for all $g\in G_p$ (Using ergodicity and the cocycle identity, this would imply that $F^p$ is constant).\\

We argue as in \cite[Lemma D.3 (i)]{Berg& tao & ziegler}: Taking logarithm it is sufficient to show that if $F:X\rightarrow \mathbb{Z}/{p^m}\mathbb{Z}$ is an (Additive) polynomial of degree $<k$, then $pF$ is a constant.
	
	Let $g\in G_p$, then $T_g^p F = F$. Write $T_g = 1+\Delta^+_g$ where $\Delta^+_g f (x)= f(T_g(x))-f(x)$ is the additive derivative. We conclude, using the binomial formula that $\sum_{i=0}^p \binom{p}{i}(\Delta_g^+)^iF=F$. Since $F$ has degree $k$ and $k\leq p$ we have that $(\Delta_g^+)^pF=0$. Therefore, $$p\Delta_g^+F+\binom{p}{2} (\Delta_g^+)^2 F+...+p(\Delta_g^+)^{p-1} F=0$$
	which we rewrite as
	$$\left(1+\frac{p-1}{2}\Delta_g^+ +...+(\Delta_g^+)^{p-2}\right)(\Delta_g^+) pF =0.$$
	Inverting the expression in the bracket using Neumann series (and using the fact that $(\Delta_g^+)^{p-1}$ annihilates $(\Delta_g^+) pF$), we conclude that $\Delta_g^+ pF=0$ for any $g\in G_p$. Since $\Delta_g^+ pF=0$ for every $g\in G'$, it follows from the cocycle identity in $g$ that that $\Delta_g^+ pF=0$ for every $g\in G$. Ergodicity implies that $pF$ is a constant, as required.
	
\end{proof}
\begin{rem}
	We can generalize the result of Theorem \ref{HTDPV:thm}. If $F:X\rightarrow S^1$ is a phase polynomial of some given degree which takes values in some finite subgroup $H$ of $S^1$, we can write $H= C_{{p_1}^{l_1}}\times...\times C_{{p_N}^{l_N}}$. Composing $F$ with each of the coordinate maps $\pi_1,...,\pi_N$ yields a polynomial as in the previous theorem, hence if $p_1,...,p_N$ are sufficiently large then $F$ takes values in $C_{p_1}\times...\times C_{p_N}$. 
\end{rem}
We generalize another result from \cite[Lemma D.3]{Berg& tao & ziegler}.

\begin{prop} [line cocycles] \label{linecocycle:prop} 
	Let $X$ be an ergodic $G$-system. Let $F:X\rightarrow S^1$ be a phase polynomials of degree $<k$ and suppose that $F$ takes values in $C_n$ for some $n=p_1\cdot...\cdot p_j$ where $k< p_1,p_2,...,p_j$. Then, for every $g\in G$ we have $\prod_{t=0}^{n-1} T_g^t F = 1$. 
\end{prop}
\begin{proof}
	Write $C_n = C_{p_1}\times...\times C_{p_j}$ and let $\pi_i : C_n\rightarrow C_{p_i}$ be the projection map.  We show that $\prod_{t=0}^{n-1} T_g^t F_i=1$ where $F_i = \pi_i \circ F$.
	
	First we decompose $G$ as $G=G_{p_i}\oplus G'$ where $G_{p_i} = \{g\in G : p_ig=0\}$ and $G'$ is its complement. Taking logarithm, it is enough to show that every polynomial $F:X\rightarrow \mathbb{R}/\mathbb{Z}$ with $nF=0$ satisfies that $\sum_{t=0}^{n-1} T_g^t F = 0$. 
	
	For every $g\in G$ we have a decomposition as  $g=g_i+g'$ where $g_i\in G_{p_i}$ and $g'\in G'$. Since $F_i$ is a phase polynomial taking values in $\mathbb{Z}/{p_i}\mathbb{Z}$, we conclude by Proposition \ref{PPC} that it is invariant under $T_{g'}$. It follows that $\sum_{t=0}^{n-1} T_g^t F=\sum_{t=0}^{n-1} T_{g_i}^t F$. If $g_i=0$, then $\sum_{t=0}^{n-1}T_gF=nF=0$. Otherwise, since $\Delta_g^+ = T_g -1$, using the binomial formula we have,
	$$\sum_{t=0}^{p_i-1}T_{g_i}^t F_i=\sum_{t=0}^{p_i-1} \binom{p_i}{t+1}(\Delta_{g_i}^+)^t F_i.$$
	Since $F$ is a phase polynomial of degree $<k$, direct computation shows that so is $F_i$. Repeated application of Lemma \ref{vdif:lem} and the fact that $k<p_i$ implies that $(\Delta^+_{g_i})^{p_i-1}$ annihilates $F_i$. Since $p_i$ divides $\binom{p_i}{t+1}$ for $0\leq t\leq p_i-1$, we conclude that
	$$\sum_{t=0}^{p_i-1}T_{g_i}^t F_i=0.$$
	As $g_i$ is of order $p_i$ and $p_i$ divides $n$ we have that $\sum_{t=0}^{n-1} T_{g_i}^t F_i$ is a constant multiple of $\sum_{t=0}^{p_i-1}T_{g_i}^t F_i=0$, hence trivial. Thus, for every $1\leq i\leq j$ and any $g\in G$ we have that $$\sum_{t=0}^{n-1}T_g F_i = 0.$$
	Since this holds for every coordinate, we conclude that $$\sum_{t=0}^{n-1} T_g F=0$$ as required.
\end{proof}

\subsection{Roots of phase polynomials} \label{roots:app}
When the multiset $P$ is unbounded, it is not true that every $\bigoplus_{p\in P}\mathbb{F}_p$-phase polynomial has an $n$-th root that is also a phase polynomial (see Example \ref{Example} and compare with \cite[Corollary D.7]{Berg& tao & ziegler}). However, when the phase polynomial takes finitely many values (for example when the underlying space is totally disconnected), we can use the tools developed by Bergelson Tao and Ziegler in \cite[Appendix D]{Berg& tao & ziegler} to construct phase polynomial roots.\\

Let $X$ be an ergodic $G$-system and $P:X\rightarrow S^1$ be a phase polynomial (of any degree). Suppose that there exists a prime $p$ and a natural number $n$ such that $P$ takes values in $C_{p^n}$. Write $G_p$ for the $p$-component of $G$ and $G=G_p\oplus G^\perp$. We see by Proposition \ref{PPC} that $P$ is invariant with respect to the action of $G^\perp$ on $X$. Let $\mathcal{B}_p$ be the $\sigma$-algebra of all $G^\perp$ invariant functions and $X_p$ be the factor of $X$ which corresponds to that $\sigma$-algebra. It is easy to see that the induced action of $G_p$ on $X_p$ is ergodic. This construction allow to generalize the following results of Bergelson Tao and Ziegler about $(\mathbb{F}_p^\omega, X,S^1)$-phase polynomials to our settings.\footnote{Equivalently, the proofs for these results can also be obtained by following the same arguments as in \cite{Berg& tao & ziegler}.}

We begin with the following version of Proposition D.5 from \cite{Berg& tao & ziegler}, the proof is identical and is therefore omitted.
\begin{prop}
	Let $P:X\rightarrow \mathbb{Z}/_{p^m}\mathbb{Z}$ be an (additive) polynomial of degree $<d$, and let $\mathbb{Z}/{p^l}\mathbb{Z}$ be a cyclic group. Embed $\{0,1,...,p-1\}$ into $\mathbb{Z}/{p^l}\mathbb{Z}$ in the obvious manner. Then for any $0\leq j\leq m-1$ the map $b_j(P)$ is a polynomial of degree $<O_{l,d,p,j}(1)$, where $b_j:\mathbb{Z}/{p^l}\mathbb{Z}\rightarrow\{0,1,...,p-1\}$ is the $j$-th digit map.
\end{prop}
Just like in \cite{Berg& tao & ziegler}, this proposition implies that functions of phase polynomials are phase polynomials. However, in our case we have to add an assumption about the values of the phase polynomials.
\begin{cor} [Functions of phase polynomials are phase polynomials] \label{funofpoly} Let $\varphi_1,...,\varphi_m$ be $m$ phase polynomials of degree $<d$ for some $d,m\geq 1$ with values in $C_{p^d}$. Let $n\geq 1$, and let $F(\varphi_1,...,\varphi_m)$ be some function of $\varphi_1,...,\varphi_m$ which takes values in the cyclic group $C_{p^n}$. Then $F(\varphi_1,...,\varphi_m)$ is a $(X,S^1)$ phase polynomial of degree $<O_{p,d,m,n}(1)$.
\end{cor}
In particular we have the following result.
\begin{cor}[Phase polynomials with finite values have phase polynomial roots] \label{roots}
Let $d,n>1$ be integers and $p$ be a prime number. Let $X$ be an ergodic $G$-system and $P:X\rightarrow S^1$ be a phase polynomial of degree $<d$ which takes values in $C_{p^d}$ and let $n>1$. Then, there exists a phase polynomial $\Psi:X\rightarrow S^1$ of degree $<O_{d,n,p}(1)$ such that $\Psi^n = P$.
    \end{cor}
	
\address{Einstein Institute of Mathematics\\
	The Hebrew University of Jerusalem\\
	Edmond J. Safra Campus, Jerusalem, 91904, Israel \\ Or.Shalom@mail.huji.ac.il}
\end{document}